\documentclass[12pt,reqno]{amsart}
\usepackage{setspace}
\usepackage[inner=1in,outer=1in,tmargin=1in,bmargin=1in]{geometry}
\usepackage{mathtools}
\usepackage{amssymb}
\usepackage[utf8]{inputenc}
\usepackage[T1]{fontenc}
\usepackage[english]{babel}
\usepackage{parskip}
\usepackage{tensor}
\usepackage{hyperref}
\usepackage{amsmath}
\usepackage{afterpage}
\usepackage{tikz}
\usepackage{tikz-3dplot}
\usetikzlibrary{decorations.markings}
\usetikzlibrary{intersections}
\usetikzlibrary{arrows.meta}
\usepackage{lmodern}
\usepackage{graphicx}
\usepackage[capposition=bottom]{floatrow}
\usepackage{empheq}
\usepackage{amsthm}
\usepackage{lipsum}
\usepackage{cite}
\usepackage{cases}
\usepackage{hyperref}
\usepackage{pgfplots}
\pgfplotsset{compat=1.18}
\usepackage{enumitem}
\usepackage{pdfpages}
\usepackage{slashed}
\usepackage[colorinlistoftodos]{todonotes}
\usepackage{comment}
\usepackage{dsfont}
\usepackage{comment}
\usepackage{mathtools}
\mathtoolsset{showonlyrefs}
\numberwithin{equation}{section}

\allowdisplaybreaks

\hypersetup{
	colorlinks,
	linkcolor={red!80!black},
	citecolor={blue!80!black},
	urlcolor={blue!80!black},
	linktoc=page
}

\usetikzlibrary{decorations.pathmorphing,calc}
\makeatletter

\newcommand{\bg}{\Box_{g}}

\makeatletter
\newcommand\incircbin
{%
  \mathpalette\@incircbin
}
\newcommand\@incircbin[2]
{%
  \mathbin%
  {%
    \ooalign{\hidewidth$#1#2$\hidewidth\crcr$#1\bigcirc$}%
  }%
}

\makeatother

\theoremstyle{definition}

\theoremstyle{plain}
\newtheorem{theorem}{Theorem}[section]
\newtheorem{lemma}[theorem]{Lemma}
\newtheorem{proposition}[theorem]{Proposition}

\newtheorem{corollary}[theorem]{Corollary}

\theoremstyle{definition}
\newtheorem{definition}[theorem]{Definition}

\theoremstyle{definition}

\newtheorem{remark}[theorem]{Remark}

\newtheorem{assumption}[theorem]{Assumption}

\newcommand{\R}{\mathbb{R}}

\newcommand{\U}{\uu}

\newcommand{\n}{\mathcal{H}}

\renewcommand{\U}{\uu}

\newcommand{\M}{\mathcal{M}}

\newcommand{\Ht}{\mathcal{H}_{\tau}}
\newcommand{\diam}{\mathsf{D}} 
\newcommand{\dg}{\diam_\gamma}

\newcommand{\uu}{M}

\DeclareMathOperator{\supp}{supp}

\let\div\relax
\DeclareMathOperator{\div}{div}

\def\R{\mathbb R}
\DeclarePairedDelimiter{\norm}{\lVert}{\rVert}

\let\div\relax
\DeclareMathOperator{\div}{div}

\begin{document}

\begin{abstract}
We consider the inverse problem of recovering a potential $q$, depending on space and time, in the wave equation $(\Box_g+q)u=0$ on a Lorentzian manifold. We prove Lipschitz stability for a formally determined version of this problem under a \textit{null-cone foliation} assumption. As a consequence, we obtain uniqueness for the Lorentzian Calder\'on problem in a fixed conformal class. Our approach uses distorted spherical waves.
\end{abstract}

\title[Stability for a Lorentzian inverse problem]{Stability for a formally determined Lorentzian inverse problem}

\author[S. FILIPPAS, L. OKSANEN]{Spyridon Filippas, Lauri Oksanen}

\address{Department of Mathematics and Statistics, University of Helsinki, PO BOX 68, 00014, HELSINKI, Finland}
\email {spyridon.filippas@helsinki.fi}

\address{Department of Mathematics and Statistics, University of Helsinki, PO BOX 68, 00014, HELSINKI, Finland}
\email {lauri.oksanen@helsinki.fi}

\maketitle
\tableofcontents

\section{Introduction}

The Lorentzian Calder\'on problem asks to recover the coefficients of a wave equation on a compact Lorentzian manifold with boundary $(\overline \uu, g)$ from the Cauchy data of its solutions. We consider the operator $\Box_g+q$ where the unknown is a smooth compactly supported potential $q$, and $\Box_{g}u= - \textnormal{div}_{g} \nabla^{g}u$. 

Our main result, Theorem~\ref{thm_main}, is Lipschitz stability for a formally determined inverse problem: the Cauchy data on $\partial\uu$ of one wave per leaf of a null-cone foliation determine $q$, with a Lipschitz estimate. The main novelty of this paper is the geometric generality under which this holds, the null-cone foliation assumption which we now state.

Let $(\M,g)$ be a globally hyperbolic Lorentzian manifold of dimension $1+n$ and let $M \subset \mathcal M$ be an open set with smooth boundary. We refer the reader to Appendix~\ref{ap_lorentz} for facts in Lorentzian geometry used in this article.

\begin{assumption}[Null-cone foliation]
\label{assumption_null_cone_foliation}
We assume that there exists an open causal diamond $\diam = I^+(p_-) \cap I^-(p_+)$ generated by two points $p_-, p_+ \in \M$, such that the following holds. There is a smooth future-directed timelike curve $\gamma \colon I \to \diam$ such that
\[
  \uu \Subset \diam,
  \qquad
  \overline{\uu} \cap \gamma(I) = \emptyset.
\]
For every $\tau \in I$, consider the \emph{regular part} of the future null
cone of $\gamma(\tau)$,
\[
  \mathcal{H}_\tau
  =
  \bigl(\partial J^+(\gamma(\tau)) \setminus \{\gamma(\tau)\}\bigr)
  \cap \diam,
\]
and assume that the family
$\{\mathcal{H}_\tau\}_{\tau \in I}$ forms a smooth foliation of
$\dg=\diam \setminus \gamma(I)$.
More precisely,
\[
  \dg = \bigsqcup_{\tau \in I} \mathcal{H}_\tau,
\]
and there exists a smooth function $f \colon \dg
\to I$ with $df \neq 0$ such that $\mathcal{H}_\tau = f^{-1}(\tau)$
for all $\tau \in I$.
\end{assumption}

\medskip

Before stating the stability result, let us give a uniqueness result it implies as a corollary. For $q\in C_0^\infty(\uu)$, the Cauchy data set of $\Box_g+q$ on $\uu$ is
\begin{equation}\label{eq_cds}
\mathcal C_q(\uu)=\bigl\{(u|_{\partial\uu},\partial_{\nu}u|_{\partial\uu}) : u\in C^\infty(\overline\uu),\ (\Box_g+q)u=0\text{ in }\uu\bigr\},
\end{equation}
where $\partial_\nu$ is the outward normal derivative of $u$ with respect to some fixed reference Riemannian metric on $\overline \uu$.

\begin{theorem}\label{thm_calderon}
Under Assumption \ref{assumption_null_cone_foliation}, let $q_1,q_2\in C_0^\infty(\uu)$. If $\mathcal C_{q_1}(\uu)\subset\mathcal C_{q_2}(\uu)$, then $q_1=q_2$.
\end{theorem}

\begin{figure}[ht]
\centering
\begin{tikzpicture}[scale=0.8]

\fill[gray!10, draw=black, thick]
    (-1.5,0) -- (5,6.2) -- (11.5,0) -- (5,-6.2) -- cycle;

\node[above] at (5,6.2) {\Large $p_+$};
\node[below] at (5,-6.2) {\Large $p_-$};

\draw[blue!70!black, very thick, domain=-0.8:0.8, samples=100]
    plot ({2.25+2.75*\x*\x},{6.2*\x});

\node[blue!70!black] at (2,2.7) {\Large $\gamma$};

\fill[white, opacity=0.45]
    (3.15,-2.05)
    .. controls (3.65,-2.45) and (8.15,-2.30) .. (8.75,-0.95)
    .. controls (9.05,0.0) and (8.55,1.25) .. (7.15,1.65)
    .. controls (5.25,1.95) and (3.55,1.15) .. (3.30,0.05)
    .. controls (3.00,-0.75) and (2.95,-1.45) .. cycle;

\draw[black!50, thick, dashed]
    (3.15,-2.05)
    .. controls (3.65,-2.45) and (8.15,-2.30) .. (8.75,-0.95)
    .. controls (9.05,0.0) and (8.55,1.25) .. (7.15,1.65)
    .. controls (5.25,1.95) and (3.55,1.15) .. (3.30,0.05)
    .. controls (3.00,-0.75) and (2.95,-1.45) .. cycle;

\node at (8.05,-2.15) {\Large $\uu$};

\fill[black!10]
    (3.55,-1.45)
    .. controls (4.00,-1.85) and (7.85,-1.70) .. (8.30,-0.70)
    .. controls (8.55,0.0) and (8.05,0.90) .. (6.95,1.25)
    .. controls (5.35,1.45) and (3.85,0.90) .. (3.65,0.0)
    .. controls (3.35,-0.60) and (3.30,-1.05) .. cycle;

\draw[black!70, thick]
    (3.55,-1.45)
    .. controls (4.00,-1.85) and (7.85,-1.70) .. (8.30,-0.70)
    .. controls (8.55,0.0) and (8.05,0.90) .. (6.95,1.25)
    .. controls (5.35,1.45) and (3.85,0.90) .. (3.65,0.0)
    .. controls (3.35,-0.60) and (3.30,-1.05) .. cycle;

\node at (5.95,-0.65) {\Large $K$};


\coordinate (G1) at ({2.25+2.75*0.25},{6.2*(-0.50)});

\coordinate (G2) at ({2.25+2.75*0.42*0.42},{6.2*(-0.42)});

\coordinate (G3) at ({2.25+2.75*0.34*0.34},{6.2*(-0.34)});

\coordinate (G4) at ({2.25+2.75*0.26*0.26},{6.2*(-0.26)});

\fill[blue!70!black] (G1) circle (2.5pt);
\fill[blue!70!black] (G2) circle (2.5pt);
\fill[blue!70!black] (G3) circle (2.5pt);
\fill[blue!70!black] (G4) circle (2.5pt);


\draw[red, very thick]
    (G1) -- (-0.81545,0.65295);
\draw[red, very thick]
    (G1) -- (8.70423,2.66673);

\draw[red, very thick]
    (G2) -- (-0.66519,0.79629);
\draw[red, very thick]
    (G2) -- (8.34678,3.00768);

\draw[red, very thick]
    (G3) -- (-0.49690,0.95680);
\draw[red, very thick]
    (G3) -- (8.00735,3.33145);

\draw[red, very thick]
    (G4) -- (-0.31060,1.13450);
\draw[red, very thick]
    (G4) -- (7.68593,3.63803);

\node[red] at (7.75,-0.05) {\Large $\mathcal H_\tau$};

\end{tikzpicture}

\caption{The null-cone foliation of the causal diamond
$\diam=I^+(p_-)\cap I^-(p_+)$ containing $\overline\uu$. The timelike curve $\gamma$ joins $p_-$ to
$p_+$ and remains disjoint from $\overline{\uu}$. The family
$\{\mathcal H_\tau\}_{\tau\in I}$ foliates $\dg$. Here $K\subset\uu$ is a compact set containing $\supp q$.}
\label{fig:null_cone_foliation}
\end{figure}
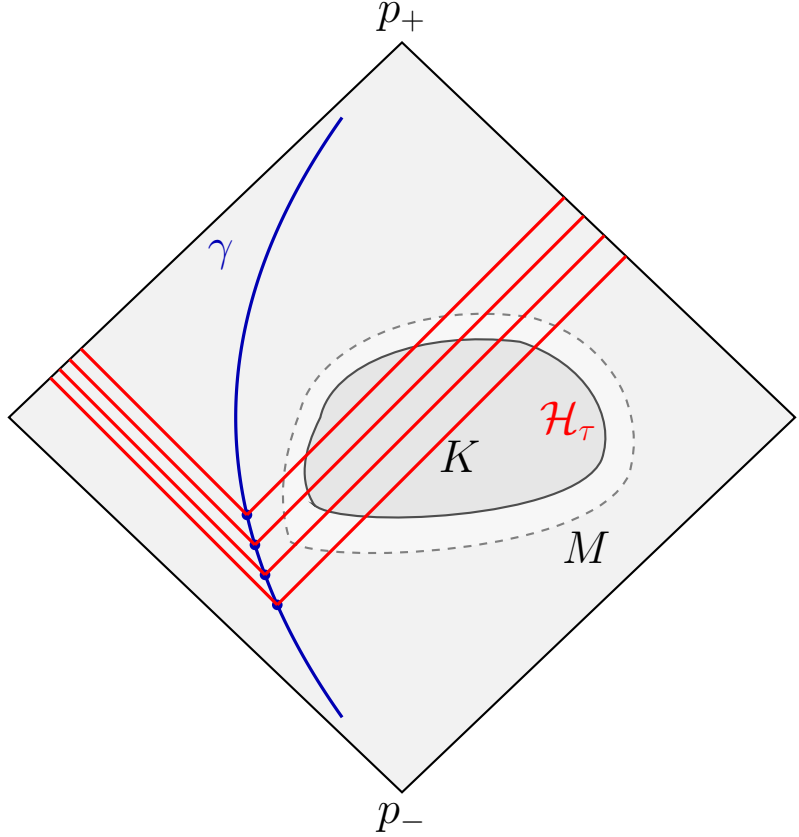
\begin{remark}
The condition $\overline{\U}\cap\gamma(I)=\emptyset$
ensures that $\overline\uu$ stays uniformly away
from the vertices of the null cones. This allows us to disregard the
singular behavior of the cones at their vertices.
For each $\tau\in I$, the set $\n_\tau$ is assumed to be a smooth null hypersurface. This is not automatic on a
general Lorentzian manifold: null cones may develop conjugate points
and caustics, and their causal boundaries need not be smooth globally.
Thus, this condition amounts to assuming that the relevant portions of
the null cones are free of such singularities in $\diam$.
\end{remark}

\subsection{The main stability result}

The quantities $ \uu, \diam, f$ are from now on fixed by Assumption~\ref{assumption_null_cone_foliation}. We fix as well a compact set $K \subset \uu$. For $\tau\in I$ we write
 \begin{equation}
 \label{def_of_quantities}
     \varphi_{\tau} = f-\tau, \qquad \mathcal{H}_{\tau} = \{\varphi_{\tau}=0\}, \qquad Q_{\tau} = \{f>\tau\},\qquad L = -\nabla_g f .
 \end{equation}
Notice that since $f$ is defined on $\dg$, we have $Q_\tau\subset\dg$. The vector field $L$ is null, tangent to every leaf, and future-directed (see Lemma~\ref{lem_monotone}(c)).

We fix the Riemannian metric $h$ having the same density as $g$ given in Lemma~\ref{lem_aux_riem_metric}. Throughout the article, $\nu$ will be the outward unit normal to $\partial\uu$ with respect to the auxiliary Riemannian metric $h$, and $\partial_\nu$ the corresponding normal derivative.

To define a distorted spherical wave adapted to our geometric setting, we solve an inhomogeneous wave equation whose source is concentrated on one leaf $\mathcal H_\tau$ and localized by a cut-off $\chi$. For $x\in\dg$, the null generator of $\mathcal H_{f(x)}$ through $x$ is a null geodesic emanating from the vertex $\gamma(f(x))$ (see Section~\ref{sec_progressive_wave}). We will require that $\chi$ meets every such generator segment ending in $K$.

\begin{assumption}[On the sources]\label{assumption_sources}
We assume that $\chi\in C^\infty_c(\dg; [0,1])$ satisfies the following. For every $x\in K$ the segment of the null generator through $x$ from $\gamma(f(x))$ to $x$ meets $\{\chi>0\}$.
\end{assumption}

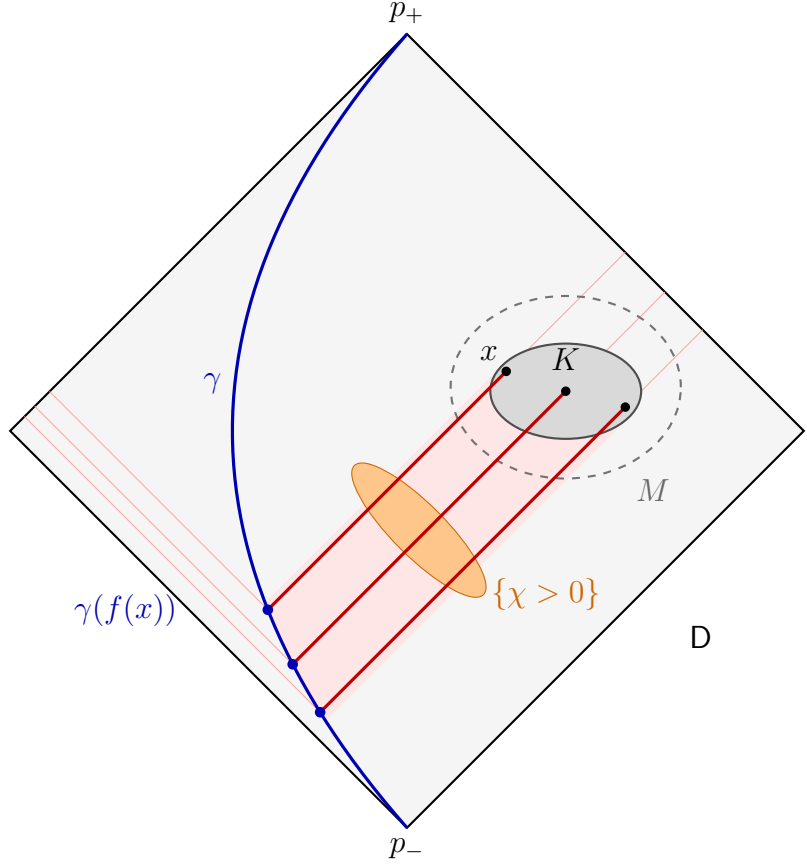
\begin{figure}[ht]
\centering
\begin{tikzpicture}[scale=1.05]
\fill[gray!8, draw=black, thick] (0,-5) -- (5,0) -- (0,5) -- (-5,0) -- cycle;
\node[below] at (0,-5) {$p_-$};
\node[above] at (0,5) {$p_+$};
\node at (3.7,-2.6) {$\diam$};

\coordinate (V1) at (-1.753,-2.253);\coordinate (X1) at (1.25,0.75);
\coordinate (R1) at (2.750,2.250);\coordinate (L1) at (-4.503,0.497);
\coordinate (V2) at (-1.440,-2.940);\coordinate (X2) at (2.0,0.5);
\coordinate (R2) at (3.250,1.750);\coordinate (L2) at (-4.690,0.310);
\coordinate (V3) at (-1.094,-3.544);\coordinate (X3) at (2.75,0.3);
\coordinate (R3) at (3.725,1.275);\coordinate (L3) at (-4.819,0.181);

\fill[red!10] (-1.789,-2.162) -- (1.194,0.821) -- (2.806,0.179) -- (-1.026,-3.653)
  -- plot[domain=-3.653:-2.162, samples=20] ({-2.2*(1-(\x/5)*(\x/5))},{\x}) -- cycle;

\foreach \i in {1,2,3}
  \draw[red!35] (L\i) -- (V\i) -- (R\i);

\draw[black!55, thick, dashed] (2.0,0.55) ellipse[x radius=1.45, y radius=1.15];
\node[black!60] at (3.1,-0.75) {$\uu$};
\fill[black!15, draw=black!70, thick] (2.0,0.5) ellipse[x radius=0.95, y radius=0.6];
\node at (2,0.9) {$K$};

\fill[orange!45, draw=orange!80!black, rotate around={-45:(0.15,-1.25)}]
  (0.15,-1.25) ellipse[x radius=1.15, y radius=0.33];
\node[orange!80!black] at (1.75,-2.05) {$\{\chi>0\}$};

\foreach \i in {1,2,3}
  \draw[red!75!black, very thick] (V\i) -- (X\i);
\foreach \i in {1,2,3}
  \fill (X\i) circle (1.7pt);
\node[above left, xshift=1pt] at (X1) {$x$};

\draw[blue!70!black, very thick, domain=-5:5, samples=60]
  plot ({-2.2*(1-(\x/5)*(\x/5))},{\x});
\node[blue!70!black, left] at (-2.2,0.6) {$\gamma$};
\foreach \i in {1,2,3}
  \fill[blue!70!black] (V\i) circle (1.9pt);
\node[blue!70!black, left, xshift=-30pt] at (V1) {$\gamma(f(x))$};
\end{tikzpicture}
\caption{For $x\in K$, the null generator segment from $\gamma(f(x))$ to $x$ (thick) is part of the leaf $\mathcal H_{f(x)}$ (light). These segments sweep out the shaded region and Assumption~\ref{assumption_sources} requires that each of them meets $\{\chi>0\}$.}
\label{fig:sources}
\end{figure}

Since such a  $\chi$ is compactly supported in $\dg$ one has in particular that $\supp(\chi)\cap\gamma(I)=\emptyset$. Any $\chi\in C^\infty_c(\U)$ with $\chi \geq 0$ and $\chi=1$ on $K$ satisfies Assumption~\ref{assumption_sources}, since $x$ itself lies on the segment. Lemma~\ref{lem_exterior_source} gives $\chi$ satisfying Assumption~\ref{assumption_sources} with $\supp\chi\cap\overline\uu=\emptyset$.

\begin{remark}\label{rem_D_globhyp}
The causal diamond $\diam=I^+(p_-)\cap I^-(p_+)$ is globally
hyperbolic. It is causally convex: if $x,y\in\diam$ and
$\lambda$ is a future-directed causal curve from $x$ to $y$, then
for any $z$ on $\lambda$ one has $p_-\ll x\leq z$, hence $p_-\ll z$,
and $z\leq y\ll p_+$, hence\footnote{This uses the chronological transitivity property: $x \ll y$ and $y \leq z$ implies $x \ll z$.}$z\ll p_+$, so $z\in\diam$. For $x,y\in\diam$, the diamond $J^+(x)\cap J^-(y)$ in
$\diam$ equals the one in $\mathcal{M}$
hence is compact. As a consequence $\diam$ is globally hyperbolic.
\end{remark}

Global hyperbolicity and Proposition~\ref{prop_wave_lorentz} allow to define $V_{\tau} \in \mathcal{D}^\prime(\diam)$ the unique solution of
\begin{equation}
 \label{eq_wave_lorentz}
\begin{cases}
    (\bg+q) V_{\tau}=\chi \delta^\prime(\varphi_\tau)  \quad \textnormal{in } \diam,\\
    \supp (V_{\tau}) \subset J^+(\supp (\chi \delta^\prime(\varphi_\tau))).
    \end{cases}
\end{equation}

The source in~\eqref{eq_wave_lorentz} is a distribution supported on the null leaf $\mathcal H_\tau$. We refer to Appendix~\ref{ap_distrib_manifold} where we have collected some useful facts about distributions on manifolds.

A key feature of the null-cone foliation assumption is that it allows to prove that $V=V_\tau$ admits a progressive wave representation. Indeed, we prove in Section~\ref{sec_progressive_wave} that $V=V_\tau$ can be written
\begin{equation}\label{eq_progressive_decomposition}
V=\alpha\delta(\varphi_\tau)+v\,H(\varphi_\tau)\qquad\text{in }\diam,
\end{equation}
where $\alpha$ and $v=v_{q,\tau}$ belong to $C^\infty(\diam)$. The important feature of the coefficients $\alpha,v $  is that $\alpha$ solves a transport equation involving $\theta=\Box f$ and $\chi$ and it is independent of $q$. On the other hand, $v$ solves a transport equation involving $\alpha$ and $q$, see~\eqref{system_a_v}.

Consider the set
\begin{equation}
\Gamma_\tau=\overline{Q_\tau}\cap\partial\uu.
\end{equation}
We define as well $\mathcal R\subset I$ the set of regular values of $f|_{\partial\uu}$. We restrict to $\tau\in\mathcal R$ because, for such $\tau$, $\mathcal H_\tau$ is transversal to $\partial\uu$. In particular, this will allow to integrate by parts in $Q_\tau\cap\uu$ (see Section~\ref{sec_carleman}). The restriction is harmless since by Sard's theorem, $I\setminus\mathcal R$ has measure zero.

Now for $\tau \in \mathcal R$ the traces of $V_\tau$ and $\partial_\nu V_\tau$ on $\partial\uu$ are well defined \cite[Theorem~8.2.4 and Corollary~8.2.7]{hormander1990} and one can define the map
\begin{equation}\label{eq_data_V}
\tilde{\mathcal L}:\ q\longmapsto\bigl(V_{q,\tau}|_{\Gamma_\tau},\ \partial_\nu V_{q,\tau}|_{\Gamma_\tau}\bigr)_{\tau\in\mathcal R}.
\end{equation}

It is however more convenient to work with smooth functions. Since $V=v$ on $Q_\tau$, $V=0$ on $\{\varphi_\tau<0\}$, and $\alpha$ does not depend on $q$, these data are equivalent to
\begin{equation}\label{eq_data_v}
\mathcal L:\ q\longmapsto\bigl(v_{q,\tau}|_{\Gamma_\tau},\ \partial_\nu v_{q,\tau}|_{\Gamma_\tau}\bigr)_{\tau\in\mathcal R},
\end{equation}
that is, $\tilde{\mathcal L}q_1=\tilde{\mathcal L}q_2$ if and only if $(\mathcal Lq_1)_\tau=(\mathcal Lq_2)_\tau$ for all $\tau\in\mathcal R$. Given two potentials $q_1,q_2$ we shall write $v_{j,\tau}=v_{q_j,\tau}$, $j \in \{1,2\}$, for the corresponding data used in $\mathcal{L}q_j$, and $w_\tau=v_{1,\tau}-v_{2,\tau}$.

Fix $A>0$, an integer $k\ge0$, and $[\tau_a,\tau_b]\subset I$ with $f(K)\subset[\tau_a,\tau_b]$. The integer $k$ is fixed large enough depending only on $n$. A potential $q\in C^\infty(\M)$ is \textit{admissible} if
\begin{equation}
\label{def_of_admissible}
\supp q\subset K \textnormal{ and } \|q\|_{C^k}\le A.
\end{equation}
For admissible $q_1,q_2$ and $\tau\in\mathcal R$ we set
\begin{equation}\label{eq_data_E}
 E_\tau=|v_{1,\tau}-v_{2,\tau}|^2_{\Gamma_\tau}+|\partial_\nu v_{1,\tau}-\partial_\nu v_{2,\tau}|^2_{\Gamma_\tau}+|\nabla_h^{\tan} v_{1,\tau}-\nabla_h^{\tan}v_{2,\tau}|^2_{\Gamma_\tau},\qquad \mathcal E= \sqrt{\int_{\tau_a}^{\tau_b} E_\tau\,d\tau },
\end{equation}
where $|\cdot|_{\Gamma_\tau}$ is the $L^2$ norm on $\Gamma_\tau$ for the measure induced by $h$, and $\nabla_h^{\tan}$ is the tangential $h$-gradient on $\partial\uu$. Since $\nabla_h^{\tan}w_\tau$ is determined by $w_\tau|_{\Gamma_\tau}$, $\mathcal E$ is determined by $\mathcal Lq_1$ and $\mathcal Lq_2$.

\begin{theorem}\label{thm_main}
There is $C>0$, depending only on $A$ in~\eqref{def_of_admissible} and on the geometric quantities of Assumption~\ref{assumption_null_cone_foliation} such that for all admissible $q_1,q_2$
\[
\|q_1-q_2\|_{L^2(K)}\le C\mathcal E.
\]
In particular, if $(\mathcal Lq_1)_\tau=(\mathcal Lq_2)_\tau$ for almost every $\tau\in I$, then $q_1=q_2$.
\end{theorem}

So Theorem~\ref{thm_main}, applied with sources outside $\overline\uu$, is a formally determined Calder\'on-type result: the traces on $\partial\uu$ of one wave per leaf, sent from outside $\overline\uu$, determine $q$, with Lipschitz stability.

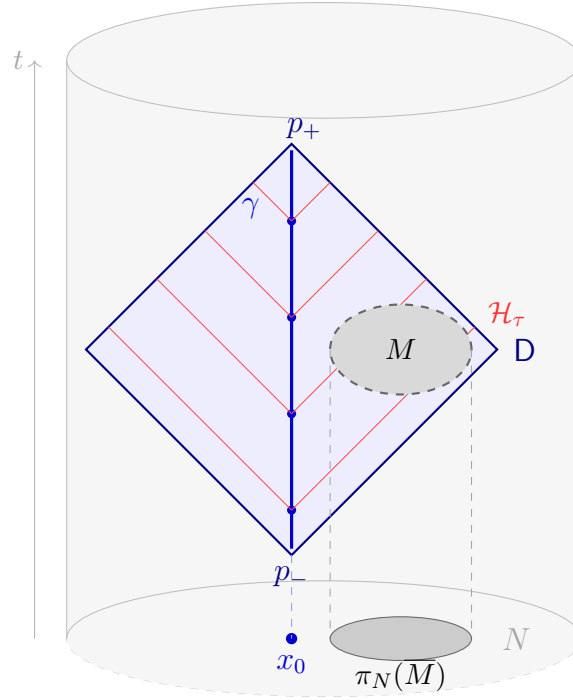
\begin{figure}[ht]
\centering
\begin{tikzpicture}[scale=0.85]


\draw[gray!40, dashed, thin]
    (0,-4.5) ++(0:4 and 0.9) arc(0:-180:4 and 0.9);

\fill[gray!7]
    (-4,-4.5) -- (-4,4.5)
    arc(180:0:4 and 0.9)
    -- (4,-4.5)
    arc(0:-180:4 and 0.9)
    -- cycle;

\draw[gray!50, thin] (-4,-4.5) -- (-4, 4.5);
\draw[gray!50, thin] ( 4,-4.5) -- ( 4, 4.5);

\draw[gray!50, thin] (0,4.5) ellipse (4 and 0.9);

\draw[gray!50, thin]
    (0,-4.5) ++(0:4 and 0.9) arc(0:180:4 and 0.9);

\node[gray!70] at (3.0, -4.5) {$N$};

%
%

\fill[blue!8, opacity=0.8]
    (-0.5,-3.2) -- (2.7,0) -- (-0.5,3.2) -- (-3.7,0) -- cycle;
\draw[blue!50!black, thick]
    (-0.5,-3.2) -- (2.7,0) -- (-0.5,3.2) -- (-3.7,0) -- cycle;

\node[blue!50!black, right=2pt] at (2.7,0) {$\diam$};
\node[above, blue!50!black] at (-0.3,3.1) {$p_+$};
\node[below, blue!50!black] at (-0.5,-3.2) {$p_-$};

\draw[blue!80!black, very thick] (-0.5,-3.1) -- (-0.5,3.1);
\node[blue!80!black, left=3pt] at (-0.7, 2.2) {$\gamma$};

\foreach \t in {-2.5,-1.0,0.5,2.0} {
    \fill[blue!80!black] (-0.5,\t) circle (2pt);
}

\draw[red!70, thin] (-0.5,-2.5) -- (2.35, 0.35);
\draw[red!70, thin] (-0.5,-2.5) -- (-3.35, 0.35);
\draw[red!70, thin] (-0.5,-1.0) -- (1.6, 1.1);
\draw[red!70, thin] (-0.5,-1.0) -- (-2.6, 1.1);
\draw[red!70, thin] (-0.5, 0.5) -- (0.85, 1.85);
\draw[red!70, thin] (-0.5, 0.5) -- (-1.85, 1.85);
\draw[red!70, thin] (-0.5, 2.0) -- (0.1,  2.6);
\draw[red!70, thin] (-0.5, 2.0) -- (-1.1, 2.6);

\node[red!80, right=1pt] at (2.35,0.55) {\small$\mathcal{H}_\tau$};

\fill[black!15] (1.2,0) ellipse (1.1 and 0.7);
\draw[black!60, dashed, thick]
    (1.2,0) ellipse (1.1 and 0.7);
\node at (1.2,0) {\small$\uu$};

\draw[black!35, dashed] (0.1, 0) -- (0.1, -4.5);
\draw[black!35, dashed] (2.3, 0) -- (2.3, -4.5);

\fill[black!20, draw=black!60]
    (1.2,-4.5) ellipse (1.1 and 0.34);
\node[below=1pt] at (1.2,-4.6) {\small$\pi_N(\overline\uu)$};

\fill[blue!80!black] (-0.5,-4.5) circle (2.5pt);
\node[below=2pt, blue!80!black] at (-0.5,-4.5) {$x_0$};
\draw[blue!40, dashed] (-0.5,-3.2) -- (-0.5,-4.5);

\draw[->, gray!60, thin] (-4.5,-4.5) -- (-4.5,4.5)
    node[left, gray!70] {\small$t$};

\end{tikzpicture}
\caption{The spacetime $\mathbb{R}\times N$ represented as a cylinder. The null-cone foliation assumption~\ref{assumption_null_cone_foliation} is satisfied as soon as there is a point $x_0$ outside of the spatial projection of $\overline\uu$, such that $\exp_{x_0}$ is a global diffeomorphism.}
\label{fig:ultrastatic-3d}
\end{figure}

\medskip

We conclude this subsection with two remarks concerning the Lorentzian Calder\'on problem.

\begin{remark}
In the Lorentzian Calder\'on problem, only a compact Lorentzian manifold with boundary is given, and the data are the Cauchy data of the solutions on it. One can obtain an \textit{intrinsic} version of Theorem~\ref{thm_calderon} by stating it in terms of a \textit{null-cone extension} assumption:  a globally hyperbolic Lorentzian manifold $(\M,\tilde g)$ of dimension $1+n$, together with an isometric, time-orientation preserving embedding $\overline\uu\subset\M$, such that Assumption~\ref{assumption_null_cone_foliation} holds for $(\M,\tilde g)$ and $\uu$. Since $\mathcal C_q(\uu)$ depends only on $(\overline\uu,g)$ and $q$, Theorem~\ref{thm_calderon} then holds for every $(\overline\uu,g)$ admitting a null-cone extension.
\end{remark}

\begin{remark}\label{rem_conformal}
For $n\ge2$, Theorem~\ref{thm_calderon} also recovers positive conformal factors $c$ inside $\uu$ by the same gauge transformation as in \cite[Section~1]{alexakis2022}. Let $c_1,c_2$ be positive conformal factors with $c_j=1$ near $\partial\uu$, and assume that the Cauchy data set of $\Box_{c_1g}$ on $\uu$ is contained in that of $\Box_{c_2g}$. The map $u\mapsto c_j^{(n-1)/4}u$ sends the solutions of $\Box_{c_jg}u=0$ to those of $(\Box_g+q_j)w=0$, where
\[
q_j=-c_j^{-(n-1)/4}\,\Box_g\,c_j^{(n-1)/4}\in C_0^\infty(\uu),
\]
and it does not change the Cauchy data on $\partial\uu$. Hence $\mathcal C_{q_1}(\uu)\subset\mathcal C_{q_2}(\uu)$, and Theorem~\ref{thm_calderon} gives $q_1=q_2$. Then $v=c_1^{(n-1)/4}-c_2^{(n-1)/4}\in C_0^\infty(\uu)$ solves $(\Box_g+q_1)v=0$, so $v=0$ by uniqueness for the Cauchy problem, and we obtain $c_1=c_2$.
\end{remark}

\subsection{Examples}
\label{sec_examples}

\subsubsection{Ultrastatic spaces}  

Let $(N,h_N)$ be a complete Riemannian manifold, and let $\overline\uu$ be a compact domain with smooth boundary in the ultrastatic spacetime $(\R\times N,-dt^2+h_N)$. Assume that there is a point $x_0\in N$ outside the compact set $\pi_N(\overline\uu)$ such that $\exp_{x_0}\colon T_{x_0}N\to N$ is a diffeomorphism, where $\pi_N\colon\R\times N\to N$ is the projection. Then Assumption~\ref{assumption_null_cone_foliation} holds, by Lemma~\ref{lem:ultrastatic-foliation} (see Figure~\ref{fig:ultrastatic-3d}). In particular, Assumption~\ref{assumption_null_cone_foliation} holds for every compact domain with smooth boundary in Minkowski space. Since Assumption~\ref{assumption_null_cone_foliation} only involves the causal structure, it depends only on the conformal class of the metric. Hence the same holds for $c\,(-dt^2+h_N)$, with $c>0$ smooth on $\R\times N$ and possibly depending on time.

\begin{figure}[ht]
\centering
\begin{tikzpicture}[scale=1.5, baseline={(0,1)}]
\fill[black!4, even odd rule] (-0.931,-0.372) -- (-0.980,-0.215) -- (-0.997,-0.103) -- (-1.003,0.000) -- (-0.997,0.103) -- (-0.980,0.215) -- (-0.953,0.314) -- (-0.911,0.419) -- (-0.864,0.509) -- (-0.802,0.603) -- (-0.694,0.724) -- (-0.610,0.796) -- (-0.520,0.858) -- (-0.363,0.935) -- (-0.251,0.971) -- (-0.163,0.990) -- (-0.064,1.001) -- (0.022,1.003) -- (0.141,0.993) -- (0.230,0.976) -- (0.302,0.957) -- (0.406,0.917) -- (0.468,0.887) -- (0.569,0.825) -- (0.648,0.765) -- (0.729,0.689) -- (0.831,0.562) -- (0.889,0.465) -- (0.931,0.372) -- (0.980,0.215) -- (0.997,0.103) -- (1.003,-0.000) -- (0.997,-0.103) -- (0.980,-0.215) -- (0.953,-0.314) -- (0.911,-0.419) -- (0.864,-0.509) -- (0.802,-0.603) -- (0.694,-0.724) -- (0.610,-0.796) -- (0.520,-0.858) -- (0.363,-0.935) -- (0.251,-0.971) -- (0.163,-0.990) -- (0.064,-1.001) -- (-0.022,-1.003) -- (-0.141,-0.993) -- (-0.230,-0.976) -- (-0.302,-0.957) -- (-0.406,-0.917) -- (-0.468,-0.887) -- (-0.569,-0.825) -- (-0.648,-0.765) -- (-0.729,-0.689) -- (-0.831,-0.562) -- (-0.889,-0.465) -- cycle;
\fill[black!15, even odd rule] (-0.976,-0.233) -- (-0.992,-0.148) -- (-1.001,-0.054) -- (-0.996,0.116) -- (-0.983,0.201) -- (-0.959,0.292) -- (-0.899,0.441) -- (-0.896,0.438) -- (-0.902,0.416) -- (-0.896,0.396) -- (-0.859,0.361) -- (-0.771,0.321) -- (-0.664,0.292) -- (-0.502,0.263) -- (-0.346,0.246) -- (-0.000,0.231) -- (0.346,0.246) -- (0.502,0.263) -- (0.664,0.292) -- (0.771,0.321) -- (0.859,0.361) -- (0.896,0.396) -- (0.902,0.416) -- (0.899,0.442) -- (0.942,0.343) -- (0.976,0.232) -- (0.999,0.086) -- (1.001,-0.065) -- (0.980,-0.211) -- (0.939,-0.353) -- (0.901,-0.440) -- (0.873,-0.470) -- (0.807,-0.505) -- (0.689,-0.542) -- (0.561,-0.567) -- (0.411,-0.587) -- (0.071,-0.607) -- (-0.142,-0.605) -- (-0.313,-0.596) -- (-0.473,-0.580) -- (-0.615,-0.558) -- (-0.790,-0.511) -- (-0.850,-0.484) -- (-0.882,-0.463) -- (-0.901,-0.440) -- (-0.931,-0.374) -- cycle;
\draw[black!60, thin] plot[smooth, tension=0.5] coordinates {(0.000,0.228) (0.063,0.229) (0.126,0.230) (0.188,0.232) (0.249,0.236) (0.309,0.240) (0.372,0.245) (0.432,0.251) (0.493,0.259) (0.554,0.268) (0.614,0.278) (0.675,0.291) (0.734,0.307) (0.791,0.325) (0.846,0.350) (0.894,0.387) (0.900,0.435)};
\draw[black!60, thin] plot[smooth, tension=0.5] coordinates {(-0.901,0.435) (-0.884,0.377) (-0.833,0.343) (-0.776,0.320) (-0.718,0.302) (-0.659,0.288) (-0.600,0.276) (-0.538,0.265) (-0.476,0.257) (-0.414,0.249) (-0.354,0.243) (-0.291,0.238) (-0.230,0.235) (-0.169,0.232) (-0.106,0.230) (-0.043,0.229) (-0.000,0.228)};
\draw[black!60, thin] plot[smooth, tension=0.5] coordinates {(0.000,-0.605) (0.063,-0.604) (0.126,-0.603) (0.188,-0.600) (0.249,-0.597) (0.309,-0.593) (0.372,-0.588) (0.432,-0.582) (0.493,-0.574) (0.554,-0.565) (0.614,-0.555) (0.675,-0.542) (0.734,-0.526) (0.791,-0.508) (0.846,-0.483) (0.894,-0.446) (0.901,-0.435)};
\draw[black!60, thin] plot[smooth, tension=0.5] coordinates {(-0.900,-0.435) (-0.857,-0.477) (-0.803,-0.503) (-0.746,-0.523) (-0.685,-0.539) (-0.626,-0.552) (-0.566,-0.563) (-0.506,-0.572) (-0.446,-0.580) (-0.382,-0.587) (-0.321,-0.592) (-0.257,-0.597) (-0.196,-0.600) (-0.134,-0.603) (-0.071,-0.604) (-0.008,-0.605) (-0.000,-0.605)};
\draw[red!50, thin] plot[smooth, tension=0.5] coordinates {(0.000,0.978) (0.000,0.928) (0.000,0.877) (0.000,0.826) (0.000,0.774) (0.000,0.722) (0.000,0.671) (0.000,0.620) (0.000,0.570) (0.000,0.517) (0.000,0.466) (0.000,0.413) (0.000,0.362) (0.000,0.310) (0.000,0.257) (0.000,0.204) (0.000,0.154) (0.000,0.103) (0.000,0.052) (0.000,0.001) (0.000,-0.050) (0.000,-0.101) (0.000,-0.152) (0.000,-0.202) (0.000,-0.256) (0.000,-0.309) (0.000,-0.360) (0.000,-0.411) (0.000,-0.464) (0.000,-0.516) (0.000,-0.568) (0.000,-0.619) (0.000,-0.670) (0.000,-0.721) (0.000,-0.773) (0.000,-0.825) (0.000,-0.877) (0.000,-0.927) (0.000,-0.978) (0.000,-1.000)};
\draw[red!40, thin, densely dashed, opacity=0.6] plot[smooth, tension=0.5] coordinates {(0.000,-1.000) (0.000,-0.978)};
\draw[red!50, thin] plot[smooth, tension=0.5] coordinates {(0.000,0.978) (0.048,0.961) (0.094,0.938) (0.137,0.912) (0.177,0.882) (0.217,0.847) (0.253,0.811) (0.288,0.771) (0.319,0.730) (0.350,0.686) (0.377,0.643) (0.403,0.597) (0.426,0.553) (0.447,0.506) (0.468,0.459) (0.487,0.410) (0.504,0.359) (0.520,0.308) (0.533,0.259) (0.545,0.210) (0.556,0.160) (0.565,0.110) (0.572,0.059) (0.578,0.009) (0.583,-0.042) (0.586,-0.093) (0.588,-0.143) (0.588,-0.197) (0.586,-0.250) (0.582,-0.303) (0.577,-0.354) (0.570,-0.405) (0.561,-0.454) (0.549,-0.505) (0.536,-0.555) (0.520,-0.605) (0.502,-0.654) (0.480,-0.702) (0.456,-0.747) (0.429,-0.792) (0.398,-0.834) (0.364,-0.872) (0.327,-0.906) (0.286,-0.937) (0.240,-0.962) (0.193,-0.979) (0.149,-0.989)};
\draw[red!40, thin, densely dashed, opacity=0.6] plot[smooth, tension=0.5] coordinates {(0.149,-0.989) (0.097,-0.992) (0.046,-0.988) (0.000,-0.978)};
\draw[red!50, thin] plot[smooth, tension=0.5] coordinates {(0.000,0.978) (0.052,0.973) (0.104,0.965) (0.156,0.954) (0.208,0.940) (0.258,0.924) (0.308,0.904) (0.357,0.882) (0.402,0.859) (0.446,0.834) (0.488,0.807) (0.529,0.777) (0.569,0.746) (0.607,0.712) (0.643,0.677) (0.678,0.640) (0.712,0.601) (0.743,0.561) (0.772,0.519) (0.800,0.476) (0.825,0.431) (0.848,0.386) (0.869,0.339) (0.888,0.292) (0.904,0.243) (0.918,0.195) (0.929,0.145) (0.938,0.096) (0.945,0.046) (0.949,-0.005) (0.951,-0.055) (0.950,-0.108) (0.946,-0.162) (0.940,-0.215) (0.930,-0.267) (0.918,-0.319) (0.903,-0.369) (0.885,-0.418) (0.864,-0.467) (0.841,-0.513) (0.815,-0.559) (0.787,-0.602) (0.757,-0.644) (0.724,-0.683) (0.689,-0.721) (0.652,-0.756) (0.613,-0.790) (0.572,-0.820) (0.538,-0.843)};
\draw[red!40, thin, densely dashed, opacity=0.6] plot[smooth, tension=0.5] coordinates {(0.538,-0.843) (0.494,-0.869) (0.449,-0.893) (0.402,-0.914) (0.354,-0.932) (0.305,-0.947) (0.255,-0.960) (0.204,-0.969) (0.153,-0.976) (0.101,-0.979) (0.049,-0.980) (0.000,-0.978)};
\draw[red!50, thin] plot[smooth, tension=0.5] coordinates {(0.000,0.978) (0.052,0.980) (0.104,0.979) (0.156,0.975) (0.208,0.969) (0.258,0.959) (0.308,0.946) (0.357,0.931) (0.405,0.912) (0.452,0.891) (0.497,0.867) (0.541,0.841)};
\draw[red!40, thin, densely dashed, opacity=0.6] plot[smooth, tension=0.5] coordinates {(0.541,0.841) (0.584,0.812) (0.624,0.780) (0.663,0.747) (0.699,0.710) (0.733,0.672) (0.766,0.632) (0.796,0.590) (0.823,0.546) (0.848,0.500) (0.870,0.453) (0.890,0.405) (0.907,0.355) (0.922,0.304) (0.933,0.252) (0.942,0.200) (0.948,0.147) (0.951,0.093) (0.951,0.039) (0.948,-0.011) (0.943,-0.061) (0.936,-0.111) (0.926,-0.160) (0.914,-0.210) (0.899,-0.258) (0.882,-0.306) (0.863,-0.353) (0.841,-0.400) (0.817,-0.445) (0.791,-0.489) (0.763,-0.532) (0.733,-0.573) (0.702,-0.613) (0.668,-0.651) (0.632,-0.688) (0.595,-0.723) (0.557,-0.756) (0.516,-0.786) (0.475,-0.815) (0.432,-0.842) (0.388,-0.867) (0.344,-0.889) (0.294,-0.910) (0.244,-0.929) (0.193,-0.945) (0.142,-0.958) (0.090,-0.968) (0.037,-0.975) (0.000,-0.978)};
\draw[red!50, thin] plot[smooth, tension=0.5] coordinates {(0.000,0.978) (0.051,0.989) (0.101,0.993) (0.151,0.989)};
\draw[red!40, thin, densely dashed, opacity=0.6] plot[smooth, tension=0.5] coordinates {(0.151,0.989) (0.202,0.977) (0.248,0.958) (0.294,0.931) (0.335,0.900) (0.371,0.864) (0.404,0.826) (0.434,0.784) (0.461,0.740) (0.484,0.694) (0.504,0.648) (0.522,0.599) (0.538,0.548) (0.551,0.498) (0.562,0.447) (0.571,0.398) (0.578,0.347) (0.583,0.295) (0.586,0.243) (0.588,0.189) (0.588,0.139) (0.586,0.089) (0.583,0.038) (0.578,-0.012) (0.572,-0.063) (0.564,-0.114) (0.555,-0.164) (0.544,-0.214) (0.532,-0.263) (0.519,-0.311) (0.503,-0.363) (0.485,-0.413) (0.466,-0.462) (0.446,-0.510) (0.424,-0.556) (0.401,-0.600) (0.375,-0.646) (0.348,-0.689) (0.317,-0.733) (0.286,-0.773) (0.251,-0.813) (0.214,-0.849) (0.175,-0.884) (0.133,-0.915) (0.090,-0.941) (0.044,-0.963) (0.000,-0.978)};
\draw[red!50, thin] plot[smooth, tension=0.5] coordinates {(0.000,0.978) (0.000,1.000)};
\draw[red!40, thin, densely dashed, opacity=0.6] plot[smooth, tension=0.5] coordinates {(0.000,1.000) (0.000,0.949) (0.000,0.898) (0.000,0.847) (0.000,0.795) (0.000,0.745) (0.000,0.693) (0.000,0.640) (0.000,0.588) (0.000,0.536) (0.000,0.485) (0.000,0.433) (0.000,0.382) (0.000,0.331) (0.000,0.278) (0.000,0.225) (0.000,0.175) (0.000,0.125) (0.000,0.074) (0.000,0.023) (0.000,-0.028) (0.000,-0.079) (0.000,-0.130) (0.000,-0.181) (0.000,-0.231) (0.000,-0.284) (0.000,-0.336) (0.000,-0.388) (0.000,-0.438) (0.000,-0.490) (0.000,-0.541) (0.000,-0.592) (0.000,-0.645) (0.000,-0.697) (0.000,-0.749) (0.000,-0.801) (0.000,-0.852) (0.000,-0.902) (0.000,-0.953) (0.000,-0.978)};
\draw[red!50, thin] plot[smooth, tension=0.5] coordinates {(-0.000,0.978) (-0.051,0.989) (-0.101,0.993) (-0.151,0.989)};
\draw[red!40, thin, densely dashed, opacity=0.6] plot[smooth, tension=0.5] coordinates {(-0.151,0.989) (-0.202,0.977) (-0.248,0.958) (-0.294,0.931) (-0.335,0.900) (-0.371,0.864) (-0.404,0.826) (-0.434,0.784) (-0.461,0.740) (-0.484,0.694) (-0.504,0.648) (-0.522,0.599) (-0.538,0.548) (-0.551,0.498) (-0.562,0.447) (-0.571,0.398) (-0.578,0.347) (-0.583,0.295) (-0.586,0.243) (-0.588,0.189) (-0.588,0.139) (-0.586,0.089) (-0.583,0.038) (-0.578,-0.012) (-0.572,-0.063) (-0.564,-0.114) (-0.555,-0.164) (-0.544,-0.214) (-0.532,-0.263) (-0.519,-0.311) (-0.503,-0.363) (-0.485,-0.413) (-0.466,-0.462) (-0.446,-0.510) (-0.424,-0.556) (-0.401,-0.600) (-0.375,-0.646) (-0.348,-0.689) (-0.317,-0.733) (-0.286,-0.773) (-0.251,-0.813) (-0.214,-0.849) (-0.175,-0.884) (-0.133,-0.915) (-0.090,-0.941) (-0.044,-0.963) (-0.000,-0.978)};
\draw[red!50, thin] plot[smooth, tension=0.5] coordinates {(-0.000,0.978) (-0.052,0.980) (-0.104,0.979) (-0.156,0.975) (-0.208,0.969) (-0.258,0.959) (-0.308,0.946) (-0.357,0.931) (-0.405,0.912) (-0.452,0.891) (-0.497,0.867) (-0.541,0.841)};
\draw[red!40, thin, densely dashed, opacity=0.6] plot[smooth, tension=0.5] coordinates {(-0.541,0.841) (-0.584,0.812) (-0.624,0.780) (-0.663,0.747) (-0.699,0.710) (-0.733,0.672) (-0.766,0.632) (-0.796,0.590) (-0.823,0.546) (-0.848,0.500) (-0.870,0.453) (-0.890,0.405) (-0.907,0.355) (-0.922,0.304) (-0.933,0.252) (-0.942,0.200) (-0.948,0.147) (-0.951,0.093) (-0.951,0.039) (-0.948,-0.011) (-0.943,-0.061) (-0.936,-0.111) (-0.926,-0.160) (-0.914,-0.210) (-0.899,-0.258) (-0.882,-0.306) (-0.863,-0.353) (-0.841,-0.400) (-0.817,-0.445) (-0.791,-0.489) (-0.763,-0.532) (-0.733,-0.573) (-0.702,-0.613) (-0.668,-0.651) (-0.632,-0.688) (-0.595,-0.723) (-0.557,-0.756) (-0.516,-0.786) (-0.475,-0.815) (-0.432,-0.842) (-0.388,-0.867) (-0.344,-0.889) (-0.294,-0.910) (-0.244,-0.929) (-0.193,-0.945) (-0.142,-0.958) (-0.090,-0.968) (-0.037,-0.975) (-0.000,-0.978)};
\draw[red!50, thin] plot[smooth, tension=0.5] coordinates {(-0.000,0.978) (-0.052,0.973) (-0.104,0.965) (-0.156,0.954) (-0.208,0.940) (-0.258,0.924) (-0.308,0.904) (-0.357,0.882) (-0.402,0.859) (-0.446,0.834) (-0.488,0.807) (-0.529,0.777) (-0.569,0.746) (-0.607,0.712) (-0.643,0.677) (-0.678,0.640) (-0.712,0.601) (-0.743,0.561) (-0.772,0.519) (-0.800,0.476) (-0.825,0.431) (-0.848,0.386) (-0.869,0.339) (-0.888,0.292) (-0.904,0.243) (-0.918,0.195) (-0.929,0.145) (-0.938,0.096) (-0.945,0.046) (-0.949,-0.005) (-0.951,-0.055) (-0.950,-0.108) (-0.946,-0.162) (-0.940,-0.215) (-0.930,-0.267) (-0.918,-0.319) (-0.903,-0.369) (-0.885,-0.418) (-0.864,-0.467) (-0.841,-0.513) (-0.815,-0.559) (-0.787,-0.602) (-0.757,-0.644) (-0.724,-0.683) (-0.689,-0.721) (-0.652,-0.756) (-0.613,-0.790) (-0.572,-0.820) (-0.538,-0.843)};
\draw[red!40, thin, densely dashed, opacity=0.6] plot[smooth, tension=0.5] coordinates {(-0.538,-0.843) (-0.494,-0.869) (-0.449,-0.893) (-0.402,-0.914) (-0.354,-0.932) (-0.305,-0.947) (-0.255,-0.960) (-0.204,-0.969) (-0.153,-0.976) (-0.101,-0.979) (-0.049,-0.980) (-0.000,-0.978)};
\draw[red!50, thin] plot[smooth, tension=0.5] coordinates {(-0.000,0.978) (-0.048,0.961) (-0.094,0.938) (-0.137,0.912) (-0.177,0.882) (-0.217,0.847) (-0.253,0.811) (-0.288,0.771) (-0.319,0.730) (-0.350,0.686) (-0.377,0.643) (-0.403,0.597) (-0.426,0.553) (-0.447,0.506) (-0.468,0.459) (-0.487,0.410) (-0.504,0.359) (-0.520,0.308) (-0.533,0.259) (-0.545,0.210) (-0.556,0.160) (-0.565,0.110) (-0.572,0.059) (-0.578,0.009) (-0.583,-0.042) (-0.586,-0.093) (-0.588,-0.143) (-0.588,-0.197) (-0.586,-0.250) (-0.582,-0.303) (-0.577,-0.354) (-0.570,-0.405) (-0.561,-0.454) (-0.549,-0.505) (-0.536,-0.555) (-0.520,-0.605) (-0.502,-0.654) (-0.480,-0.702) (-0.456,-0.747) (-0.429,-0.792) (-0.398,-0.834) (-0.364,-0.872) (-0.327,-0.906) (-0.286,-0.937) (-0.240,-0.962) (-0.193,-0.979) (-0.149,-0.989)};
\draw[red!40, thin, densely dashed, opacity=0.6] plot[smooth, tension=0.5] coordinates {(-0.149,-0.989) (-0.097,-0.992) (-0.046,-0.988) (-0.000,-0.978)};
\draw[black!75, thick] plot[smooth, tension=0.5] coordinates {(-0.027,-1.000) (-0.071,-0.997) (-0.111,-0.994) (-0.152,-0.988) (-0.192,-0.981) (-0.232,-0.973) (-0.272,-0.962) (-0.312,-0.950) (-0.351,-0.936) (-0.389,-0.921) (-0.428,-0.904) (-0.465,-0.886) (-0.500,-0.866) (-0.535,-0.845) (-0.570,-0.822) (-0.604,-0.797) (-0.636,-0.771) (-0.668,-0.745) (-0.698,-0.716) (-0.726,-0.687) (-0.754,-0.657) (-0.780,-0.626) (-0.805,-0.594) (-0.828,-0.560) (-0.850,-0.526) (-0.871,-0.492) (-0.890,-0.456) (-0.908,-0.420) (-0.924,-0.383) (-0.938,-0.345) (-0.952,-0.307) (-0.963,-0.269) (-0.973,-0.230) (-0.982,-0.191) (-0.989,-0.151) (-0.994,-0.111) (-0.997,-0.071) (-1.000,-0.031) (-1.000,0.009) (-0.999,0.049) (-0.996,0.089) (-0.992,0.129) (-0.986,0.169) (-0.978,0.208) (-0.969,0.247) (-0.958,0.286) (-0.946,0.324) (-0.932,0.362) (-0.917,0.400) (-0.900,0.436) (-0.882,0.472) (-0.862,0.507) (-0.840,0.542) (-0.818,0.575) (-0.794,0.608) (-0.768,0.640) (-0.742,0.671) (-0.714,0.701) (-0.684,0.729) (-0.654,0.757) (-0.622,0.783) (-0.589,0.808) (-0.554,0.832) (-0.520,0.854) (-0.485,0.874) (-0.449,0.893) (-0.412,0.911) (-0.375,0.927) (-0.336,0.942) (-0.298,0.955) (-0.258,0.966) (-0.217,0.976) (-0.177,0.984) (-0.134,0.991) (-0.093,0.996) (-0.049,0.999) (0.027,1.000) (0.071,0.997) (0.111,0.994) (0.152,0.988) (0.192,0.981) (0.232,0.973) (0.272,0.962) (0.312,0.950) (0.351,0.936) (0.389,0.921) (0.428,0.904) (0.465,0.886) (0.500,0.866) (0.535,0.845) (0.570,0.822) (0.604,0.797) (0.636,0.771) (0.668,0.745) (0.698,0.716) (0.726,0.687) (0.754,0.657) (0.780,0.626) (0.805,0.594) (0.828,0.560) (0.850,0.526) (0.871,0.492) (0.890,0.456) (0.908,0.420) (0.924,0.383) (0.938,0.345) (0.952,0.307) (0.963,0.269) (0.973,0.230) (0.982,0.191) (0.989,0.151) (0.994,0.111) (0.997,0.071) (1.000,0.031) (1.000,-0.009) (0.999,-0.049) (0.996,-0.089) (0.992,-0.129) (0.986,-0.169) (0.978,-0.208) (0.969,-0.247) (0.958,-0.286) (0.946,-0.324) (0.932,-0.362) (0.917,-0.400) (0.900,-0.436) (0.882,-0.472) (0.862,-0.507) (0.840,-0.542) (0.818,-0.575) (0.794,-0.608) (0.768,-0.640) (0.742,-0.671) (0.714,-0.701) (0.684,-0.729) (0.654,-0.757) (0.622,-0.783) (0.589,-0.808) (0.554,-0.832) (0.520,-0.854) (0.485,-0.874) (0.449,-0.893) (0.412,-0.911) (0.375,-0.927) (0.336,-0.942) (0.298,-0.955) (0.258,-0.966) (0.217,-0.976) (0.177,-0.984) (0.134,-0.991) (0.093,-0.996) (0.049,-0.999) (0.027,-1.000)};
\fill[blue!80!black] (0,0.978) circle (1.3pt);
\node[blue!80!black, above, font=\small] at (0,1.0) {$x_0$};
\draw[black!80, fill=white] (0,-0.978) circle (1.3pt);
\node[below, font=\small] at (0,-1.0) {$x_S$};
\draw[black!70, thin] (-0.98,0.02) -- (-1.3,0.2) node[left, font=\small] {$\pi_N(\overline\uu)$};
\node[font=\small] at (-1.0,0.8) {$S^2$};
\end{tikzpicture}
\hspace{1.3cm}
\begin{tikzpicture}[scale=1.5, baseline={(0,0.7)}]
\fill[black!4, even odd rule] (-1.688,-2.174) -- (-1.807,-2.136) -- (-1.906,-2.095) -- (-1.985,-2.053) -- (-2.041,-2.009) -- (-2.076,-1.963) -- (-2.085,-1.939) -- (-2.085,-1.893) -- (-2.062,-1.847) -- (-2.016,-1.802) -- (-0.896,-1.010) -- (-0.822,-0.931) -- (-0.794,-0.881) -- (-0.775,-0.825) -- (-0.768,-0.765) -- (-0.773,-0.705) -- (-0.790,-0.646) -- (-0.818,-0.592) -- (-0.870,-0.526) -- (-1.019,-0.389) -- (-1.076,-0.327) -- (-1.142,-0.230) -- (-1.176,-0.154) -- (-1.199,-0.053) -- (-1.202,0.030) -- (-1.189,0.112) -- (-1.146,0.222) -- (-1.082,0.320) -- (-1.001,0.406) -- (-0.895,0.489) -- (-0.750,0.572) -- (-0.600,0.634) -- (-0.461,0.676) -- (-0.289,0.710) -- (-0.131,0.727) -- (0.044,0.731) -- (0.229,0.718) -- (0.416,0.687) -- (0.581,0.641) -- (0.722,0.585) -- (0.868,0.507) -- (1.001,0.406) -- (1.082,0.320) -- (1.146,0.222) -- (1.189,0.112) -- (1.203,-0.005) -- (1.187,-0.121) -- (1.158,-0.199) -- (1.097,-0.300) -- (1.019,-0.389) -- (0.870,-0.526) -- (0.804,-0.617) -- (0.781,-0.672) -- (0.769,-0.732) -- (0.770,-0.796) -- (0.794,-0.881) -- (0.842,-0.958) -- (0.922,-1.030) -- (2.016,-1.802) -- (2.062,-1.846) -- (2.076,-1.869) -- (2.088,-1.916) -- (2.085,-1.940) -- (2.062,-1.986) -- (1.985,-2.053) -- (1.859,-2.116) -- (1.688,-2.174) -- (1.475,-2.226) -- (1.226,-2.270) -- (0.947,-2.305) -- (0.645,-2.331) -- (0.326,-2.347) -- (0.000,-2.353) -- (-0.326,-2.347) -- (-0.645,-2.331) -- (-0.947,-2.305) -- (-1.226,-2.270) -- (-1.475,-2.226) -- cycle;
\fill[blue!8, even odd rule] (-1.820,-2.131) -- (-1.895,-2.100) -- (-1.985,-2.053) -- (-2.029,-2.020) -- (-2.062,-1.986) -- (-2.081,-1.951) -- (-2.088,-1.916) -- (-2.081,-1.881) -- (-2.062,-1.846) -- (-2.009,-1.796) -- (-0.917,-1.027) -- (-0.852,-0.969) -- (-0.800,-0.895) -- (-0.777,-0.835) -- (-0.768,-0.770) -- (-0.774,-0.702) -- (-0.794,-0.638) -- (-0.824,-0.583) -- (-0.862,-0.535) -- (-0.987,-0.418) -- (-0.986,-0.413) -- (-0.919,-0.452) -- (-0.842,-0.481) -- (-0.627,-0.529) -- (-0.313,-0.564) -- (-0.000,-0.575) -- (0.313,-0.564) -- (0.627,-0.529) -- (0.842,-0.481) -- (0.919,-0.452) -- (0.987,-0.413) -- (0.844,-0.556) -- (0.807,-0.611) -- (0.782,-0.671) -- (0.769,-0.735) -- (0.771,-0.801) -- (0.787,-0.866) -- (0.834,-0.948) -- (0.874,-0.991) -- (0.946,-1.048) -- (2.009,-1.796) -- (2.047,-1.829) -- (2.073,-1.864) -- (2.088,-1.916) -- (2.081,-1.951) -- (2.062,-1.986) -- (2.029,-2.020) -- (1.985,-2.053) -- (1.859,-2.116) -- (1.688,-2.174) -- (1.475,-2.226) -- (1.226,-2.270) -- (0.873,-2.313) -- (0.566,-2.336) -- (0.245,-2.350) -- (-0.164,-2.351) -- (-0.487,-2.341) -- (-0.798,-2.320) -- (-1.090,-2.289) -- (-1.416,-2.237) -- (-1.638,-2.187) -- cycle;
\fill[black!15, even odd rule] (-1.168,-0.175) -- (-1.191,-0.102) -- (-1.202,-0.034) -- (-1.202,0.025) -- (-1.193,0.093) -- (-1.174,0.160) -- (-1.142,0.230) -- (-1.101,0.296) -- (-1.055,0.351) -- (-1.052,0.346) -- (-1.069,0.320) -- (-1.071,0.296) -- (-1.060,0.271) -- (-1.033,0.246) -- (-0.937,0.198) -- (-0.817,0.162) -- (-0.632,0.126) -- (-0.450,0.104) -- (-0.210,0.088) -- (-0.000,0.083) -- (0.210,0.088) -- (0.450,0.104) -- (0.632,0.126) -- (0.817,0.162) -- (0.937,0.198) -- (1.033,0.246) -- (1.060,0.271) -- (1.071,0.296) -- (1.069,0.320) -- (1.052,0.346) -- (1.055,0.351) -- (1.116,0.273) -- (1.161,0.191) -- (1.189,0.110) -- (1.202,0.034) -- (1.198,-0.069) -- (1.168,-0.175) -- (1.110,-0.282) -- (1.037,-0.367) -- (0.959,-0.408) -- (0.818,-0.452) -- (0.632,-0.488) -- (0.412,-0.514) -- (0.126,-0.529) -- (-0.210,-0.526) -- (-0.526,-0.502) -- (-0.761,-0.465) -- (-0.939,-0.416) -- (-1.037,-0.367) -- (-1.110,-0.282) -- cycle;
\draw[black!60, thin] plot[smooth, tension=0.5] coordinates {(0.000,0.080) (0.061,0.081) (0.122,0.082) (0.182,0.084) (0.246,0.086) (0.310,0.090) (0.372,0.094) (0.433,0.099) (0.497,0.106) (0.558,0.113) (0.621,0.121) (0.680,0.131) (0.740,0.142) (0.799,0.154) (0.859,0.169) (0.917,0.187) (0.975,0.209) (1.028,0.239) (1.070,0.283) (1.060,0.342) (1.057,0.345)};
\draw[black!60, thin] plot[smooth, tension=0.5] coordinates {(-1.058,0.344) (-1.071,0.284) (-1.030,0.240) (-0.976,0.210) (-0.919,0.188) (-0.862,0.170) (-0.802,0.155) (-0.740,0.142) (-0.680,0.131) (-0.621,0.121) (-0.558,0.113) (-0.497,0.106) (-0.433,0.099) (-0.372,0.094) (-0.310,0.090) (-0.246,0.086) (-0.182,0.084) (-0.122,0.082) (-0.061,0.081) (-0.000,0.080)};
\draw[black!60, thin] plot[smooth, tension=0.5] coordinates {(0.000,-0.528) (0.061,-0.527) (0.122,-0.526) (0.182,-0.524) (0.246,-0.522) (0.310,-0.518) (0.372,-0.514) (0.433,-0.509) (0.497,-0.502) (0.558,-0.495) (0.621,-0.487) (0.680,-0.477) (0.740,-0.466) (0.799,-0.454) (0.859,-0.439) (0.917,-0.421) (0.975,-0.398) (1.028,-0.369) (1.058,-0.344)};
\draw[black!60, thin] plot[smooth, tension=0.5] coordinates {(-1.057,-0.345) (-1.009,-0.381) (-0.954,-0.407) (-0.897,-0.427) (-0.839,-0.444) (-0.780,-0.458) (-0.719,-0.470) (-0.658,-0.481) (-0.597,-0.490) (-0.538,-0.498) (-0.476,-0.504) (-0.416,-0.510) (-0.355,-0.515) (-0.292,-0.519) (-0.228,-0.522) (-0.164,-0.525) (-0.103,-0.527) (-0.042,-0.527) (-0.000,-0.528)};
\draw[red!50, thin] plot[smooth, tension=0.5] coordinates {(0.000,0.685) (0.000,0.634) (0.000,0.583) (0.000,0.531) (0.000,0.480) (0.000,0.428) (0.000,0.377) (0.000,0.326) (0.000,0.274) (0.000,0.221) (0.000,0.168) (0.000,0.115) (0.000,0.063) (0.000,0.013) (0.000,-0.039) (0.000,-0.093) (0.000,-0.144) (0.000,-0.195) (0.000,-0.246) (0.000,-0.299) (0.000,-0.350) (0.000,-0.401) (0.000,-0.452) (0.000,-0.502) (0.000,-0.554) (0.000,-0.605) (0.000,-0.657) (0.000,-0.708) (0.000,-0.760) (0.000,-0.812) (0.000,-0.866) (0.000,-0.918) (0.000,-0.969) (0.000,-1.020) (0.000,-1.070) (0.000,-1.123) (0.000,-1.174) (0.000,-1.224) (0.000,-1.275) (0.000,-1.328) (0.000,-1.379) (0.000,-1.430) (0.000,-1.481) (0.000,-1.533) (0.000,-1.584) (0.000,-1.635) (0.000,-1.686) (0.000,-1.738) (0.000,-1.789) (0.000,-1.840) (0.000,-1.892) (0.000,-1.943) (0.000,-1.994) (0.000,-2.045) (0.000,-2.097) (0.000,-2.148) (0.000,-2.199) (0.000,-2.250) (0.000,-2.302) (0.000,-2.350)};
\draw[red!50, thin] plot[smooth, tension=0.5] coordinates {(0.000,0.685) (0.050,0.669) (0.097,0.650) (0.145,0.629) (0.189,0.605) (0.233,0.579) (0.277,0.550) (0.318,0.520) (0.359,0.487) (0.396,0.453) (0.433,0.416) (0.469,0.377) (0.502,0.337) (0.534,0.295) (0.562,0.253) (0.589,0.208) (0.614,0.160) (0.636,0.114) (0.656,0.065) (0.672,0.014) (0.685,-0.035) (0.695,-0.085) (0.702,-0.135) (0.705,-0.186) (0.704,-0.239) (0.699,-0.291) (0.690,-0.340) (0.676,-0.390) (0.657,-0.438) (0.633,-0.484) (0.603,-0.527) (0.572,-0.566) (0.541,-0.606) (0.512,-0.650) (0.488,-0.696) (0.470,-0.745) (0.458,-0.796) (0.451,-0.847) (0.450,-0.901) (0.453,-0.951) (0.460,-1.002) (0.472,-1.051) (0.489,-1.102) (0.509,-1.150) (0.531,-1.195) (0.558,-1.241) (0.586,-1.285) (0.613,-1.328) (0.641,-1.371) (0.669,-1.414) (0.697,-1.457) (0.726,-1.500) (0.754,-1.543) (0.782,-1.586) (0.810,-1.629) (0.838,-1.672) (0.866,-1.714) (0.894,-1.757) (0.922,-1.800) (0.950,-1.843) (0.978,-1.886) (1.006,-1.929) (1.034,-1.972) (1.062,-2.015) (1.090,-2.058) (1.118,-2.101) (1.146,-2.144) (1.174,-2.187) (1.202,-2.230) (1.226,-2.267)};
\draw[red!50, thin] plot[smooth, tension=0.5] coordinates {(0.000,0.685) (0.053,0.680) (0.105,0.675) (0.158,0.667) (0.210,0.659) (0.262,0.649) (0.314,0.637) (0.366,0.624) (0.417,0.609) (0.468,0.593) (0.519,0.575) (0.569,0.555) (0.619,0.534) (0.668,0.510) (0.712,0.487) (0.756,0.462) (0.798,0.435) (0.840,0.407) (0.880,0.376) (0.919,0.344) (0.956,0.309) (0.991,0.272) (1.024,0.233) (1.054,0.192) (1.080,0.148) (1.103,0.102) (1.120,0.055) (1.133,0.006) (1.140,-0.044) (1.141,-0.098) (1.134,-0.152) (1.121,-0.203) (1.102,-0.252) (1.078,-0.298) (1.049,-0.341) (1.015,-0.381) (0.979,-0.418) (0.941,-0.452) (0.902,-0.486) (0.865,-0.522) (0.829,-0.559) (0.797,-0.599) (0.769,-0.644) (0.748,-0.692) (0.734,-0.743) (0.728,-0.796) (0.730,-0.850) (0.739,-0.900) (0.755,-0.948) (0.777,-0.994) (0.805,-1.037) (0.837,-1.077) (0.872,-1.113) (0.909,-1.148) (0.947,-1.181) (0.986,-1.214) (1.025,-1.246) (1.064,-1.279) (1.103,-1.312) (1.142,-1.344) (1.180,-1.377) (1.219,-1.409) (1.258,-1.442) (1.297,-1.475) (1.336,-1.507) (1.375,-1.540) (1.413,-1.572) (1.452,-1.605) (1.491,-1.637) (1.530,-1.670) (1.569,-1.703) (1.608,-1.735) (1.646,-1.768) (1.685,-1.800) (1.724,-1.833) (1.763,-1.866) (1.802,-1.898) (1.841,-1.931) (1.879,-1.963) (1.918,-1.996) (1.957,-2.028) (1.983,-2.050)};
\draw[red!50, thin] plot[smooth, tension=0.5] coordinates {(0.000,0.685) (0.053,0.688) (0.105,0.689) (0.158,0.689) (0.210,0.687) (0.262,0.684) (0.314,0.679) (0.366,0.673) (0.417,0.665) (0.468,0.656) (0.519,0.645) (0.569,0.632) (0.619,0.617) (0.668,0.600) (0.716,0.582) (0.763,0.561) (0.809,0.538) (0.853,0.512) (0.874,0.500)};
\draw[red!50, thin] plot[smooth, tension=0.5] coordinates {(0.000,0.685) (0.050,0.697) (0.100,0.706) (0.150,0.712) (0.202,0.714) (0.253,0.712) (0.292,0.707)};
\draw[red!50, thin] plot[smooth, tension=0.5] coordinates {(0.000,0.685) (0.000,0.729)};
\draw[red!50, thin] plot[smooth, tension=0.5] coordinates {(-0.000,0.685) (-0.050,0.697) (-0.100,0.706) (-0.150,0.712) (-0.202,0.714) (-0.253,0.712) (-0.292,0.707)};
\draw[red!50, thin] plot[smooth, tension=0.5] coordinates {(-0.000,0.685) (-0.053,0.688) (-0.105,0.689) (-0.158,0.689) (-0.210,0.687) (-0.262,0.684) (-0.314,0.679) (-0.366,0.673) (-0.417,0.665) (-0.468,0.656) (-0.519,0.645) (-0.569,0.632) (-0.619,0.617) (-0.668,0.600) (-0.716,0.582) (-0.763,0.561) (-0.809,0.538) (-0.853,0.512) (-0.874,0.500)};
\draw[red!50, thin] plot[smooth, tension=0.5] coordinates {(-0.000,0.685) (-0.053,0.680) (-0.105,0.675) (-0.158,0.667) (-0.210,0.659) (-0.262,0.649) (-0.314,0.637) (-0.366,0.624) (-0.417,0.609) (-0.468,0.593) (-0.519,0.575) (-0.569,0.555) (-0.619,0.534) (-0.668,0.510) (-0.712,0.487) (-0.756,0.462) (-0.798,0.435) (-0.840,0.407) (-0.880,0.376) (-0.919,0.344) (-0.956,0.309) (-0.991,0.272) (-1.024,0.233) (-1.054,0.192) (-1.080,0.148) (-1.103,0.102) (-1.120,0.055) (-1.133,0.006) (-1.140,-0.044) (-1.141,-0.098) (-1.134,-0.152) (-1.121,-0.203) (-1.102,-0.252) (-1.078,-0.298) (-1.049,-0.341) (-1.015,-0.381) (-0.979,-0.418) (-0.941,-0.452) (-0.902,-0.486) (-0.865,-0.522) (-0.829,-0.559) (-0.797,-0.599) (-0.769,-0.644) (-0.748,-0.692) (-0.734,-0.743) (-0.728,-0.796) (-0.730,-0.850) (-0.739,-0.900) (-0.755,-0.948) (-0.777,-0.994) (-0.805,-1.037) (-0.837,-1.077) (-0.872,-1.113) (-0.909,-1.148) (-0.947,-1.181) (-0.986,-1.214) (-1.025,-1.246) (-1.064,-1.279) (-1.103,-1.312) (-1.142,-1.344) (-1.180,-1.377) (-1.219,-1.409) (-1.258,-1.442) (-1.297,-1.475) (-1.336,-1.507) (-1.375,-1.540) (-1.413,-1.572) (-1.452,-1.605) (-1.491,-1.637) (-1.530,-1.670) (-1.569,-1.703) (-1.608,-1.735) (-1.646,-1.768) (-1.685,-1.800) (-1.724,-1.833) (-1.763,-1.866) (-1.802,-1.898) (-1.841,-1.931) (-1.879,-1.963) (-1.918,-1.996) (-1.957,-2.028) (-1.983,-2.050)};
\draw[red!50, thin] plot[smooth, tension=0.5] coordinates {(-0.000,0.685) (-0.050,0.669) (-0.097,0.650) (-0.145,0.629) (-0.189,0.605) (-0.233,0.579) (-0.277,0.550) (-0.318,0.520) (-0.359,0.487) (-0.396,0.453) (-0.433,0.416) (-0.469,0.377) (-0.502,0.337) (-0.534,0.295) (-0.562,0.253) (-0.589,0.208) (-0.614,0.160) (-0.636,0.114) (-0.656,0.065) (-0.672,0.014) (-0.685,-0.035) (-0.695,-0.085) (-0.702,-0.135) (-0.705,-0.186) (-0.704,-0.239) (-0.699,-0.291) (-0.690,-0.340) (-0.676,-0.390) (-0.657,-0.438) (-0.633,-0.484) (-0.603,-0.527) (-0.572,-0.566) (-0.541,-0.606) (-0.512,-0.650) (-0.488,-0.696) (-0.470,-0.745) (-0.458,-0.796) (-0.451,-0.847) (-0.450,-0.901) (-0.453,-0.951) (-0.460,-1.002) (-0.472,-1.051) (-0.489,-1.102) (-0.509,-1.150) (-0.531,-1.195) (-0.558,-1.241) (-0.586,-1.285) (-0.613,-1.328) (-0.641,-1.371) (-0.669,-1.414) (-0.697,-1.457) (-0.726,-1.500) (-0.754,-1.543) (-0.782,-1.586) (-0.810,-1.629) (-0.838,-1.672) (-0.866,-1.714) (-0.894,-1.757) (-0.922,-1.800) (-0.950,-1.843) (-0.978,-1.886) (-1.006,-1.929) (-1.034,-1.972) (-1.062,-2.015) (-1.090,-2.058) (-1.118,-2.101) (-1.146,-2.144) (-1.174,-2.187) (-1.202,-2.230) (-1.226,-2.267)};
\draw[black!75, thick] plot[smooth, tension=0.5] coordinates {(-1.999,-1.793) (-1.964,-1.768) (-1.930,-1.744) (-1.895,-1.720) (-1.860,-1.695) (-1.826,-1.671) (-1.791,-1.647) (-1.756,-1.622) (-1.722,-1.598) (-1.687,-1.573) (-1.652,-1.549) (-1.618,-1.525) (-1.583,-1.500) (-1.548,-1.476) (-1.514,-1.452) (-1.479,-1.427) (-1.444,-1.403) (-1.410,-1.378) (-1.375,-1.354) (-1.341,-1.330) (-1.306,-1.305) (-1.271,-1.281) (-1.237,-1.257) (-1.202,-1.232) (-1.167,-1.208) (-1.133,-1.183) (-1.098,-1.159) (-1.063,-1.135) (-1.029,-1.110) (-0.994,-1.086) (-0.960,-1.062) (-0.927,-1.038) (-0.894,-1.012) (-0.864,-0.985) (-0.836,-0.955) (-0.812,-0.922) (-0.792,-0.885) (-0.778,-0.846) (-0.768,-0.806) (-0.765,-0.764) (-0.767,-0.722) (-0.776,-0.681) (-0.789,-0.641) (-0.807,-0.605) (-0.829,-0.571) (-0.855,-0.539) (-0.883,-0.508) (-0.913,-0.480) (-0.944,-0.453) (-0.975,-0.425) (-1.007,-0.396) (-1.037,-0.367) (-1.065,-0.336) (-1.091,-0.303) (-1.115,-0.269) (-1.137,-0.233) (-1.156,-0.196) (-1.171,-0.158) (-1.184,-0.119) (-1.193,-0.079) (-1.198,-0.038) (-1.200,0.003) (-1.198,0.044) (-1.192,0.085) (-1.182,0.125) (-1.169,0.165) (-1.153,0.203) (-1.133,0.239) (-1.112,0.275) (-1.087,0.308) (-1.061,0.341) (-1.032,0.372) (-1.002,0.401) (-0.970,0.429) (-0.937,0.455) (-0.902,0.480) (-0.869,0.503) (-0.834,0.524) (-0.799,0.544) (-0.762,0.563) (-0.725,0.581) (-0.686,0.598) (-0.646,0.614) (-0.608,0.628) (-0.569,0.642) (-0.529,0.654) (-0.490,0.665) (-0.450,0.676) (-0.407,0.685) (-0.367,0.694) (-0.327,0.701) (-0.284,0.708) (-0.242,0.714) (-0.199,0.719) (-0.155,0.723) (-0.106,0.726) (-0.063,0.728) (0.039,0.728) (0.080,0.727) (0.128,0.725) (0.171,0.721) (0.212,0.717) (0.253,0.712) (0.294,0.707) (0.336,0.700) (0.379,0.691) (0.419,0.683) (0.461,0.673) (0.501,0.662) (0.542,0.650) (0.582,0.637) (0.621,0.624) (0.658,0.609) (0.698,0.593) (0.736,0.575) (0.774,0.557) (0.810,0.538) (0.845,0.517) (0.879,0.496) (0.912,0.473) (0.947,0.448) (0.980,0.421) (1.011,0.393) (1.041,0.363) (1.069,0.332) (1.094,0.299) (1.118,0.265) (1.139,0.229) (1.158,0.192) (1.173,0.153) (1.185,0.114) (1.194,0.074) (1.199,0.033) (1.200,-0.008) (1.197,-0.050) (1.191,-0.090) (1.181,-0.130) (1.167,-0.169) (1.151,-0.207) (1.131,-0.244) (1.109,-0.279) (1.084,-0.312) (1.057,-0.345) (1.029,-0.375) (0.998,-0.404) (0.968,-0.431) (0.936,-0.459) (0.906,-0.486) (0.877,-0.514) (0.849,-0.545) (0.824,-0.578) (0.803,-0.613) (0.786,-0.650) (0.773,-0.690) (0.766,-0.731) (0.765,-0.773) (0.770,-0.815) (0.780,-0.855) (0.796,-0.893) (0.817,-0.929) (0.841,-0.961) (0.869,-0.991) (0.900,-1.017) (0.932,-1.042) (0.966,-1.066) (1.000,-1.090) (1.035,-1.115) (1.069,-1.139) (1.104,-1.163) (1.139,-1.188) (1.173,-1.212) (1.208,-1.236) (1.243,-1.261) (1.277,-1.285) (1.312,-1.310) (1.347,-1.334) (1.381,-1.358) (1.416,-1.383) (1.451,-1.407) (1.485,-1.431) (1.520,-1.456) (1.555,-1.480) (1.589,-1.505) (1.624,-1.529) (1.658,-1.553) (1.693,-1.578) (1.728,-1.602) (1.762,-1.626) (1.797,-1.651) (1.832,-1.675) (1.866,-1.700) (1.901,-1.724) (1.936,-1.748) (1.970,-1.773) (1.999,-1.793)};
\draw[black!75, thick] plot[smooth, tension=0.5] coordinates {(0.000,-2.350) (0.064,-2.349) (0.127,-2.349) (0.191,-2.348) (0.254,-2.346) (0.317,-2.345) (0.380,-2.342) (0.442,-2.340) (0.504,-2.337) (0.566,-2.333) (0.627,-2.330) (0.687,-2.325) (0.756,-2.320) (0.823,-2.314) (0.889,-2.308) (0.955,-2.302) (1.019,-2.294) (1.082,-2.287) (1.143,-2.279) (1.203,-2.270) (1.269,-2.260) (1.333,-2.249) (1.395,-2.238) (1.455,-2.227) (1.519,-2.213) (1.580,-2.199) (1.643,-2.183) (1.703,-2.166) (1.763,-2.148) (1.824,-2.126) (1.882,-2.103) (1.937,-2.077) (1.989,-2.047) (2.038,-2.008) (2.075,-1.960) (2.083,-1.899) (2.056,-1.845) (2.012,-1.802) (1.999,-1.793)};
\draw[black!75, thick] plot[smooth, tension=0.5] coordinates {(-2.002,-1.795) (-2.048,-1.835) (-2.081,-1.888) (-2.079,-1.948) (-2.047,-1.999) (-1.999,-2.039) (-1.947,-2.072) (-1.890,-2.099) (-1.832,-2.123) (-1.773,-2.144) (-1.713,-2.163) (-1.654,-2.180) (-1.591,-2.196) (-1.531,-2.210) (-1.468,-2.224) (-1.409,-2.236) (-1.347,-2.247) (-1.284,-2.258) (-1.218,-2.268) (-1.158,-2.277) (-1.097,-2.285) (-1.035,-2.293) (-0.971,-2.300) (-0.906,-2.307) (-0.840,-2.313) (-0.773,-2.319) (-0.705,-2.324) (-0.644,-2.328) (-0.583,-2.332) (-0.522,-2.336) (-0.460,-2.339) (-0.398,-2.342) (-0.335,-2.344) (-0.272,-2.346) (-0.209,-2.348) (-0.145,-2.349) (-0.082,-2.349) (-0.018,-2.350) (-0.000,-2.350)};
\draw[green!45!black, line width=1pt] plot[smooth, tension=0.5] coordinates {(0.000,-0.249) (0.063,-0.249) (0.125,-0.248) (0.188,-0.246) (0.249,-0.244) (0.311,-0.241) (0.371,-0.237) (0.435,-0.233) (0.498,-0.227) (0.559,-0.221) (0.623,-0.213) (0.684,-0.205) (0.747,-0.195) (0.807,-0.185) (0.867,-0.173) (0.926,-0.159) (0.986,-0.142) (1.044,-0.123) (1.100,-0.099) (1.154,-0.069) (1.194,-0.024) (1.200,0.001)};
\draw[green!45!black, thin, densely dashed, opacity=0.3] plot[smooth, tension=0.5] coordinates {(1.200,0.001) (1.172,0.054) (1.122,0.088) (1.067,0.114) (1.009,0.135) (0.949,0.153) (0.888,0.168) (0.826,0.181) (0.763,0.193) (0.701,0.202) (0.640,0.211) (0.577,0.219) (0.517,0.225) (0.454,0.231) (0.391,0.236) (0.331,0.240) (0.270,0.243) (0.208,0.246) (0.146,0.248) (0.084,0.249) (0.021,0.249) (-0.042,0.249) (-0.105,0.249) (-0.167,0.247) (-0.229,0.245) (-0.290,0.242) (-0.351,0.239) (-0.415,0.234) (-0.478,0.229) (-0.540,0.223) (-0.600,0.216) (-0.662,0.208) (-0.722,0.199) (-0.783,0.189) (-0.845,0.177) (-0.906,0.164) (-0.965,0.148) (-1.023,0.130) (-1.081,0.108) (-1.135,0.081) (-1.182,0.043) (-1.200,0.000)};
\draw[green!45!black, line width=1pt] plot[smooth, tension=0.5] coordinates {(-1.200,0.000) (-1.172,-0.054) (-1.122,-0.088) (-1.067,-0.114) (-1.009,-0.135) (-0.949,-0.153) (-0.888,-0.168) (-0.826,-0.181) (-0.763,-0.193) (-0.701,-0.202) (-0.640,-0.211) (-0.577,-0.219) (-0.517,-0.225) (-0.454,-0.231) (-0.391,-0.236) (-0.331,-0.240) (-0.270,-0.243) (-0.208,-0.246) (-0.146,-0.248) (-0.084,-0.249) (-0.021,-0.249) (-0.000,-0.249)};
\draw[green!45!black, thin, densely dashed, opacity=0.3] plot[smooth, tension=0.5] coordinates {(1.105,0.097) (1.065,0.094) (1.024,0.092) (0.982,0.089) (0.940,0.087) (0.899,0.085) (0.857,0.083) (0.816,0.081) (0.774,0.080) (0.733,0.079) (0.691,0.079) (0.650,0.078) (0.609,0.078) (0.567,0.079) (0.524,0.079) (0.484,0.080) (0.443,0.081) (0.402,0.082) (0.360,0.084) (0.318,0.085) (0.275,0.087) (0.232,0.089) (0.189,0.092) (0.146,0.095) (0.102,0.098) (0.058,0.101) (0.015,0.104) (-0.029,0.108) (-0.073,0.111) (-0.117,0.115) (-0.160,0.120) (-0.204,0.124) (-0.247,0.128) (-0.290,0.133) (-0.332,0.138) (-0.374,0.142) (-0.415,0.147) (-0.456,0.152) (-0.496,0.157) (-0.539,0.162) (-0.582,0.167) (-0.623,0.172) (-0.663,0.177) (-0.705,0.182) (-0.746,0.187) (-0.788,0.191) (-0.828,0.196) (-0.869,0.200) (-0.910,0.203) (-0.950,0.206) (-0.992,0.207) (-1.033,0.208) (-1.073,0.206) (-1.113,0.200) (-1.151,0.185) (-1.169,0.164)};
\draw[green!45!black, line width=1pt] plot[smooth, tension=0.5] coordinates {(-1.169,0.164) (-1.165,0.124) (-1.143,0.090) (-1.115,0.060) (-1.085,0.033) (-1.052,0.007) (-1.018,-0.017) (-0.985,-0.040) (-0.950,-0.061) (-0.915,-0.082) (-0.878,-0.102) (-0.841,-0.121) (-0.805,-0.139) (-0.768,-0.157) (-0.728,-0.175) (-0.691,-0.191) (-0.652,-0.206) (-0.612,-0.222) (-0.571,-0.237) (-0.533,-0.250) (-0.494,-0.263) (-0.454,-0.275) (-0.413,-0.287) (-0.372,-0.298) (-0.331,-0.309) (-0.289,-0.319) (-0.246,-0.329) (-0.203,-0.337) (-0.160,-0.345) (-0.117,-0.353) (-0.073,-0.359) (-0.030,-0.365) (0.014,-0.370) (0.058,-0.374) (0.101,-0.377) (0.145,-0.380) (0.188,-0.382) (0.231,-0.383) (0.274,-0.383) (0.316,-0.382) (0.358,-0.381) (0.400,-0.378) (0.441,-0.375) (0.481,-0.371) (0.521,-0.366) (0.564,-0.360) (0.606,-0.353) (0.647,-0.345) (0.687,-0.336) (0.726,-0.326) (0.767,-0.314) (0.807,-0.302) (0.845,-0.288) (0.885,-0.272) (0.922,-0.255) (0.960,-0.236) (0.996,-0.216) (1.031,-0.193) (1.064,-0.168) (1.096,-0.141) (1.125,-0.111) (1.151,-0.078) (1.171,-0.044) (1.186,-0.006) (1.194,0.035) (1.193,0.076) (1.192,0.085)};
\draw[green!45!black, thin, densely dashed, opacity=0.3] plot[smooth, tension=0.5] coordinates {(1.192,0.085) (1.179,0.124) (1.159,0.160) (1.133,0.190) (1.102,0.217) (1.067,0.241) (1.031,0.260) (0.994,0.276) (0.955,0.290) (0.917,0.301) (0.876,0.311) (0.835,0.318) (0.796,0.324) (0.754,0.329) (0.712,0.333) (0.671,0.335) (0.629,0.336) (0.585,0.336) (0.544,0.335) (0.503,0.333) (0.460,0.331) (0.417,0.327) (0.376,0.324) (0.336,0.319) (0.294,0.314) (0.253,0.308) (0.211,0.301) (0.168,0.294) (0.126,0.287) (0.083,0.278) (0.040,0.270) (-0.003,0.261) (-0.046,0.251) (-0.089,0.241) (-0.131,0.231) (-0.174,0.220) (-0.216,0.210) (-0.258,0.199) (-0.300,0.188) (-0.341,0.176) (-0.382,0.165) (-0.422,0.154) (-0.459,0.143)};
\draw[green!45!black, thin, densely dashed, opacity=0.3] plot[smooth, tension=0.5] coordinates {(-0.706,0.202) (-0.742,0.179) (-0.775,0.155) (-0.810,0.130) (-0.843,0.104) (-0.874,0.079) (-0.906,0.052) (-0.937,0.023) (-0.967,-0.007) (-0.996,-0.038) (-1.023,-0.069) (-1.048,-0.102) (-1.072,-0.136) (-1.092,-0.172) (-1.108,-0.210) (-1.115,-0.249) (-1.105,-0.284)};
\draw[green!45!black, line width=1pt] plot[smooth, tension=0.5] coordinates {(-1.105,-0.284) (-1.071,-0.307) (-1.032,-0.312) (-0.991,-0.310) (-0.950,-0.304) (-0.909,-0.295) (-0.870,-0.285) (-0.830,-0.273) (-0.789,-0.261) (-0.749,-0.248) (-0.709,-0.235) (-0.670,-0.223) (-0.629,-0.209) (-0.591,-0.196) (-0.551,-0.183) (-0.510,-0.170) (-0.471,-0.158) (-0.432,-0.145) (-0.392,-0.133) (-0.352,-0.121) (-0.311,-0.109) (-0.269,-0.097) (-0.227,-0.085) (-0.185,-0.074) (-0.143,-0.063) (-0.100,-0.052) (-0.057,-0.042) (-0.014,-0.033) (0.029,-0.023) (0.071,-0.014) (0.114,-0.006) (0.157,0.002) (0.199,0.009) (0.242,0.016) (0.284,0.022) (0.325,0.028) (0.367,0.033) (0.407,0.038) (0.448,0.042) (0.488,0.045) (0.531,0.048) (0.573,0.051) (0.615,0.053) (0.656,0.054) (0.696,0.054) (0.738,0.054) (0.779,0.053) (0.822,0.051) (0.863,0.048) (0.906,0.043) (0.946,0.038) (0.986,0.032) (1.026,0.024) (1.065,0.013) (1.104,0.001) (1.142,-0.016) (1.175,-0.038) (1.194,-0.074) (1.194,-0.074)};
\draw[green!45!black, thin, densely dashed, opacity=0.3] plot[smooth, tension=0.5] coordinates {(1.194,-0.074) (1.165,-0.102) (1.126,-0.113) (1.086,-0.118) (1.044,-0.120) (1.003,-0.119) (0.963,-0.116) (0.923,-0.112) (0.882,-0.106) (0.843,-0.100) (0.801,-0.093) (0.761,-0.085) (0.720,-0.076) (0.681,-0.067) (0.640,-0.057) (0.599,-0.046) (0.559,-0.035) (0.520,-0.024) (0.479,-0.012) (0.438,0.001) (0.400,0.013) (0.361,0.026) (0.322,0.039) (0.282,0.052) (0.242,0.066) (0.201,0.080) (0.161,0.095) (0.120,0.110) (0.079,0.125) (0.038,0.140) (-0.004,0.155) (-0.045,0.170) (-0.086,0.186) (-0.127,0.201) (-0.168,0.216) (-0.208,0.231) (-0.249,0.245) (-0.288,0.259) (-0.328,0.272) (-0.366,0.285) (-0.404,0.298) (-0.445,0.310) (-0.485,0.322) (-0.524,0.333) (-0.565,0.343) (-0.605,0.353) (-0.647,0.361) (-0.687,0.368) (-0.729,0.374) (-0.771,0.378) (-0.811,0.380) (-0.852,0.381) (-0.892,0.378) (-0.934,0.373) (-0.974,0.364) (-1.012,0.351) (-1.048,0.332) (-1.082,0.308) (-1.110,0.277) (-1.115,0.270)};
\fill[blue!80!black] (0,0.685) circle (1.3pt);
\node[blue!80!black, above, font=\small] at (0,0.7) {$x_0$};
\draw[black!70, thin] (-1.18,0.02) -- (-1.5,0.2) node[left, font=\small] {$\pi_N(\overline\uu)$};
\draw[green!45!black, thin] (1.17,-0.06) -- (1.55,0.15) node[right, font=\small, align=left] {trapped\\[-2pt]geodesics};
\node[font=\small] at (-1.2,0.75) {$N$};
\end{tikzpicture}
\caption{Left: on the round sphere, Assumption~\ref{assumption_null_cone_foliation} holds if $x_0,x_S\notin\pi_N(\overline\uu)$ and $\overline\uu$ is short in time. Right: the cap of $N$ containing $x_0$ and $\pi_N(\overline\uu)$ has positive, non-constant curvature, and $\pi_N(\overline\uu)$ contains an open set of trapped geodesics, yet Assumption~\ref{assumption_null_cone_foliation} holds whenever $x_0 \notin \pi_N(\overline \uu)$.}
\label{fig:flare}
\end{figure}
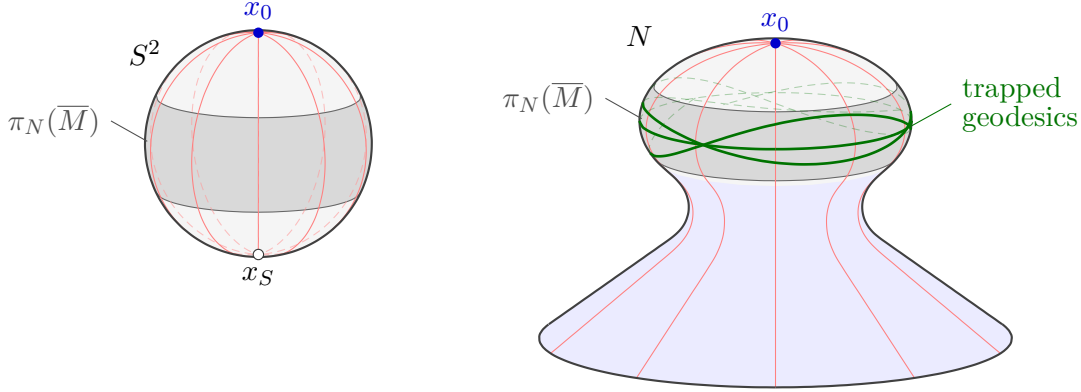

\subsubsection{Positive curvature and trapping}

An important point about Assumption~\ref{assumption_null_cone_foliation} is that it only constrains the null geodesics emanating from $\gamma$, and allows positive curvature, conjugate points and trapping elsewhere. A first example is the round sphere $N=S^2$. If there exists $x_0\in S^2$ such that neither $x_0$ nor its antipode $x_S$ lies in $\pi_N(\overline\uu)$, then Assumption~\ref{assumption_null_cone_foliation} holds provided that $\overline\uu$ is short in time. This follows from Lemma~\ref{lem:ultrastatic-foliation} combined with Remark~\ref{rem_local_ultrastatic}.

Moreover, there is a large family of examples with positive curvature and no time restriction on $\overline\uu$. Let $N\subset\R^3$ be a smooth surface of revolution, obtained by rotating around the vertical axis a curve $r\mapsto(\rho(r),z(r))$, $r\ge0$, parametrized by arclength, with $\rho(0)=0$ and $\rho>0$ on $(0,\infty)$. In polar coordinates $(r,\theta)$ around the pole $x_0=\{r=0\}$, its metric is
\[
h_N=dr^2+\rho(r)^2\,d\theta^2 .
\]
The geodesics from $x_0$ are the meridians $\theta=\text{const}$, and since $\rho>0$ they never meet again, so $\exp_{x_0}$ is a global diffeomorphism, and Assumption~\ref{assumption_null_cone_foliation} holds whenever $x_0\notin\pi_N(\overline\uu)$. If $\rho$ has a non-degenerate local maximum at some $r_0>0$, then $N$ has positive curvature near the circle $\{r=r_0\}$, which is a closed geodesic, and by Clairaut's relation \cite[Section~\mbox{4-4}]{docarmo2016} an open set of geodesics stays forever close to it. Thus $\pi_N(\overline\uu)$ may contain positive curvature, conjugate points and an open set of trapped geodesics, see Figure~\ref{fig:flare}. In fact, an arbitrarily large part of $N$ may have positive curvature: all of $N$ except a thin annulus around a circle $\{r=r_1\}$ where $\rho$ has a local minimum.

\subsection{Previous literature}

Similarly to the problem in a fixed conformal class, as formulated in Remark~\ref{rem_conformal}, recovering a Riemannian metric $g=c\,g_0$, with $g_0$ known, reduces to recovering a potential. In the elliptic case, the seminal result \cite{sylvester1987} recovers $c$ when $g_0$ is the Euclidean metric, and \cite{dos-santos-ferreira2009,dos-santos-ferreira2016,carstea2023} treat $g_0=dt^2+h$ under geometric conditions on $h$. Local rigidity was recently proved in \cite{lin2026rigidity,munozthon2026rigidity}: for $n\ge3$, a metric close to the Euclidean one \cite{lin2026rigidity}, or more generally to a conformally Euclidean one \cite{munozthon2026rigidity}, and with the same Dirichlet-to-Neumann map is its pullback by a diffeomorphism fixing the boundary. A complete solution of the elliptic anisotropic Calder\'on problem, that is, the recovery of a smooth metric up to such a diffeomorphism, was recently announced in \cite{openai2026calderon}. In dimension $n\ge3$, the formally determined elliptic problem is largely open. 

Let us now turn to the hyperbolic case. 
For time-dependent potentials in Minkowski space, uniqueness goes back to \cite{stefanov1989uniqueness} for inverse scattering. Stability estimates of logarithmic type were obtained in \cite{ben-aicha2015}. These results use the full Dirichlet-to-Neumann map, which is overdetermined. Beyond Minkowski space, a major approach is the Boundary Control method \cite{belishev1988}, which uses Tataru's unique continuation theorem for the wave equation across non-characteristic hypersurfaces \cite{tataru1995}. This theorem holds for coefficients depending real analytically on time, in particular for time-independent ones. Accordingly, the method extends to time-dependent coefficients only in the real analytic case \cite{eskin2007}. For smooth coefficients, unique continuation may fail across hypersurfaces that are not pseudoconvex \cite{alinhac1983}.

On Lorentzian manifolds, Stefanov and Yang \cite{stefanov2018} showed that the light ray transform of $q$ is determined by the Dirichlet-to-Neumann map. This reduces uniqueness to the injectivity of the light ray transform, which is known in real analytic \cite{stefanov2017support} and in stationary or static \cite{feizmohammadi2021light} geometries; see also \cite{kian2019a,feizmohammadi2021b} for product geometries $-dt^2+h$. In more general geometric settings, it was shown in~\cite{alexakis2022} that the potential can be recovered in the Lorentzian Calder\'on problem under curvature bounds that are stable under perturbations of the metric. This was achieved by combining elements of the Boundary Control method with unique continuation along a pseudoconvex foliation; see also \cite{alexakis2023,gomes2025}. For ultrastatic metrics $-dt^2+h$, these curvature bounds force $h$ to have nonpositive curvature, and \cite{alexakis2022} also assumes that null geodesics are non-trapped. In contrast, the results of this paper allow for positive curvature, conjugate points and trapped geodesics, see Section~\ref{sec_examples}. Note also that the Lorentzian Calder\'on problem is not solvable on certain noncompact domains \cite{oksanen2026counter}.

Another line of research concerns the rigidity problem. In~\cite{oksanen2024rigiditylorentziancalderonproblem} the metric itself was recovered, up to a diffeomorphism, when one of the two metrics is the Minkowski metric. This was subsequently generalized in~\cite{oksanen2026semiglobaluniquenesslorentziancalderon} when one of them is close to it. These are uniqueness results, and they rely on the proximity to the Minkowski metric. In the present paper the unknown is a potential. However, the metric $g$ is a general Lorentzian metric satisfying the null-cone foliation assumption, not assumed to be close to the Minkowski metric, and we prove Lipschitz stability. 

Finally, another major approach for uniqueness in inverse problems with a formally determined amount of data is the Bukhgeim--Klibanov method \cite{bukhgeim1981global}, which is based on Carleman estimates; see \cite{imanuvilov2001,bellassoued2017book} for hyperbolic equations with time-independent coefficients. This method gives Lipschitz stability with very few measurements, but at the price of an a priori positivity assumption on the initial data. On Lorentzian manifolds, the Bukhgeim--Klibanov method was used in~\cite{takase2023} to prove Lipschitz stability for an inverse source problem from boundary observations. Rakesh and Salo \cite{rakesh2020fixed,rakesh2020two} used incoming plane waves in this method, for fixed angle scattering. Krishnan, Rakesh and Senapati \cite{doi:10.1137/21M1400596} modified it further to recover coefficients depending on space and time in Minkowski space, from a formally determined family of incoming plane waves, one for each delay, carried by the null hyperplanes $\{t=\tau+x\cdot\omega\}$; see also \cite{krishnan2023point} for point sources. A crucial aspect of the null hyperplane approach is that it gives for free the positivity property that is needed in the Bukhgeim--Klibanov approach and hence works in less restrictive settings. Our work builds on this approach: it is a modification of the Bukhgeim--Klibanov method with distorted spherical waves, close to~\cite{doi:10.1137/21M1400596}, but in a general Lorentzian setting. Under the null-cone foliation assumption, the null hyperplanes of \cite{doi:10.1137/21M1400596} are replaced by the leaves of a null-cone foliation.

\subsection{Plan of the paper}

In Section~\ref{sec_progressive_wave} we prove the progressive wave expansion property of $V_\tau$. The fact that the wave $V_\tau$ admits this representation is a crucial consequence of the null-cone foliation assumption. Section~\ref{sec_carleman} is devoted to the Carleman estimate near a null cone, which is the tool propagating information from a null leaf to the data. In Section~\ref{sec_proof_main} we prove Theorem~\ref{thm_main}, and then Theorem~\ref{thm_calderon} by a gluing argument. Appendix~\ref{ap_lorentz} recalls facts from Lorentzian geometry, constructs the auxiliary metric $h$, and verifies Assumption~\ref{assumption_null_cone_foliation} in ultrastatic spacetimes. Appendix~\ref{ap_distrib_manifold} recalls facts on distributions on manifolds.

\bigskip

\noindent \textbf{AI Disclosure.} Claude Opus 5.5 (Anthropic) was used for literature research and proofreading. This led to a minor correction in the proof of Proposition~\ref{prop_prog_wave} and helped formulate the surfaces of revolution examples of Figure~\ref{fig:flare}. Claude assisted with the TikZ figures of the article. Proofs and text are the responsibility of the authors.

\bigskip

\noindent \textbf{Acknowledgements.} The authors were supported by the Research Council of Finland (Centre of Excellence in Inverse Modelling and Imaging and
FAME Flagship, grants 353096, 359182, 353091 and 359208, as well as a
project grant 347715) and by the European Research
Council of the European Union, grant 101086697 (LoCal).

\section{Progressive wave representation}
\label{sec_progressive_wave}

Given a null hypersurface $\n$, at each point $p \in \n $ the tangent space $T_p(\n)$ contains a null vector which is unique up to multiplication by a non zero scalar (see for instance~\cite[Lemma 28]{oneill1983}). We recall that the vector field $L$ is defined by
\begin{equation}
    \label{def_of_null_vector_L}
    L=-\nabla_g f.
\end{equation}
It is a smooth null vector field on $\dg$, tangent to every leaf $\mathcal H_\tau$, $\tau\in I$. Its integral curves are called the \textit{null generators} of the leaves. They are null geodesics since
$$
\nabla_LL=\frac12\nabla_g\bigl(g(\nabla f,\nabla f)\bigr)=0.
$$
We write
\begin{equation}
    P=\bg+q,
  \qquad
  \theta=\Box_g\varphi_\tau=\Box_gf,
  \label{eq:phi}
\end{equation}
where we recall that $\varphi_\tau$ is defined in~\eqref{def_of_quantities}. Note that $\theta$ does not depend on $\tau$.

\subsection{A monotonicity lemma}

The following lemma is a consequence of Assumption~\ref{assumption_null_cone_foliation}. It states that the null cones behave as one expects with respect to the causal structure: $\{f\ge s\}$ is the causal future of $\gamma(s)$ in $\dg$, and $f$ does not decrease along future-directed causal curves. The progressive wave representation below relies heavily on it.

\begin{lemma}\label{lem_monotone}
Under Assumption~\ref{assumption_null_cone_foliation}, the following hold.
\begin{enumerate}[label=(\alph*),leftmargin=*]
\item (Level sets) For every $s\in I$,
\begin{equation}\label{eq:levelsets}
J^+(\gamma(s))\cap\dg=\{f\ge s\},
\qquad
I^+(\gamma(s))\cap\dg=\{f>s\}.
\end{equation}

\item (Monotonicity) Let $x,y\in\dg$. Then
\begin{equation}\label{eq:mono}
y\in J^+(x)\ \Longrightarrow\ f(y)\ge f(x),
\qquad
y\in I^+(x)\ \Longrightarrow\ f(y)>f(x).
\end{equation}
In particular, for every $\tau\in I$,
\begin{equation*}
J^+(\Ht)\cap\dg\subset\{f\ge\tau\}.
\end{equation*}
\item (Time orientation) The vector field $L=-\nabla_gf$ is future-directed.
\end{enumerate}
\end{lemma}

\begin{proof}
We use the following standard facts on the globally hyperbolic manifold $\M$ \cite{oneill1983}. For every $p\in\M$, the set $J^+(p)$ is closed and the set $I^+(p)$ is open. Moreover, the following holds: for $x,y,z\in\M$,
\begin{equation}\label{eq:pushup}
x\ll y\le z
\quad\text{or}\quad
x\le y\ll z
\qquad\Longrightarrow\qquad
x\ll z.
\end{equation}
Finally, since $\gamma$ is a future-directed timelike curve,
\begin{equation}\label{eq:gamma_order}
\gamma(s)\ll\gamma(s')\qquad\text{for } s<s'.
\end{equation}

(a) Let $x\in\dg$ and set $\sigma=f(x)$. By Assumption~\ref{assumption_null_cone_foliation},
\begin{equation*}
x\in\mathcal H_\sigma\subset\partial J^+(\gamma(\sigma)).
\end{equation*}
Since $J^+(\gamma(\sigma))$ is closed, its boundary is contained in it. Since $I^+(\gamma(\sigma))$ is open and contained in $J^+(\gamma(\sigma))$, it does not meet $\partial J^+(\gamma(\sigma))$. Hence
\begin{equation}\label{eq:AB}
\gamma(\sigma)\le x
\qquad\text{and}\qquad
\gamma(\sigma)\not\ll x.
\end{equation}
Now fix $s\in I$ and compare $\sigma$ with $s$.

If $\sigma>s$, then \eqref{eq:gamma_order} and \eqref{eq:AB} give
\begin{equation*}
\gamma(s)\ll\gamma(\sigma)\le x,
\end{equation*}
so $\gamma(s)\ll x$ by \eqref{eq:pushup}. In particular $\gamma(s)\le x$.

If $\sigma=s$, then $\gamma(s)\le x$ and $\gamma(s)\not\ll x$, by \eqref{eq:AB}.

If $\sigma<s$, then $\gamma(s)\not\le x$. Indeed, otherwise
\begin{equation*}
\gamma(\sigma)\ll\gamma(s)\le x,
\end{equation*}
and \eqref{eq:pushup} gives $\gamma(\sigma)\ll x$, which contradicts \eqref{eq:AB}. A fortiori, $\gamma(s)\not\ll x$.

Thus
\begin{equation*}
\gamma(s)\le x\iff f(x)\ge s,
\qquad
\gamma(s)\ll x\iff f(x)>s,
\end{equation*}
which is \eqref{eq:levelsets}.

(b) Let $x,y\in\dg$ with $y\in J^+(x)$, and set $s=f(x)$. By (a), $\gamma(s)\le x$, hence
\begin{equation*}
\gamma(s)\le x\le y.
\end{equation*}
So $\gamma(s)\le y$, and $f(y)\ge s$ by (a). If moreover $y\in I^+(x)$, then
\begin{equation*}
\gamma(s)\le x\ll y,
\end{equation*}
so $\gamma(s)\ll y$ by \eqref{eq:pushup}, and $f(y)>s$ by (a). Finally, if $y\in J^+(\Ht)\cap\dg$, then $y\in J^+(x)$ for some $x\in\Ht$, and $f(y)\ge f(x)=\tau$.

(c) Let $x\in\dg$ and let $Y\in T_x\M$ be future-directed timelike. Since $\dg$ is open, there is a future-directed timelike curve $c\colon[0,\delta)\to\dg$ with $c(0)=x$ and $\dot c(0)=Y$. For $0<t<\delta$ we have $c(t)\in I^+(x)$, so $f(c(t))>f(x)$ by (b). Hence
\begin{equation*}
df(Y)=\frac{d}{dt}f(c(t))\Big|_{t=0}\ge0.
\end{equation*}
The vector $\nabla_gf$ is normal to the null hypersurfaces $\mathcal H_\tau$, so it is null, and it is nonzero since $df\ne0$. A nonzero null vector is not orthogonal to a timelike vector, so $df(Y)=g(\nabla_gf,Y)\ne0$. Therefore
\begin{equation*}
g(L,Y)=-df(Y)<0,
\end{equation*}
and the null vector $L$ is future-directed.
\end{proof}

\subsection{The progressive wave expansion}

We prove a progressive wave expansion property analogous to~\cite[Proposition 1.2]{doi:10.1137/21M1400596}. The following identity is the key calculation.

\begin{lemma}
\label{lem:key}
Let $\alpha,v\in C^\infty(\dg)$. Then, in $\dg$,
\begin{equation}
\begin{aligned}
  P\bigl(\alpha\,\delta(\varphi)+v\,H(\varphi)\bigr)
  ={}&
  \bigl(2L(\alpha)+\theta\alpha\bigr)\delta'(\varphi)
  \\
  &+
  \bigl(P\alpha+2L(v)+\theta v\bigr)\delta(\varphi)
  +
  (Pv)H(\varphi).
\end{aligned}
\label{eq:key}
\end{equation}
\end{lemma}

\begin{proof}
Since $\Box_g=-\div_g\nabla$, $g(\nabla\varphi,\nabla\varphi)=0$, $L=-\nabla\varphi$ and $\theta=\Box_g\varphi$, for $a\in C^\infty(\dg)$ and $F\in\mathcal D'(\R)$ we have
\begin{equation}\label{eq:PaF}
P\bigl(a\,F(\varphi)\bigr)=F(\varphi)\,Pa+F'(\varphi)\,\bigl(2L(a)+\theta a\bigr).
\end{equation}
In particular
\begin{equation}
  \Box_gH(\varphi)=\theta\,\delta(\varphi),
  \qquad
  \Box_g\delta(\varphi)=\theta\,\delta'(\varphi).
  \label{eq:box-H-delta}
\end{equation}
With $F=H$ and $F=\delta$, \eqref{eq:PaF} yields
\begin{equation}
  P(v\,H(\varphi))
  =
  (Pv)\,H(\varphi)
  +
  (2L(v)+\theta v)\,\delta(\varphi),
  \label{eq:PvH}
\end{equation}
and
\begin{equation}
  P(\alpha\,\delta(\varphi))
  =
  (P\alpha)\,\delta(\varphi)
  +
  (2L(\alpha)+\theta\alpha)\,\delta'(\varphi).
  \label{eq:Palpha}
\end{equation}
Adding \eqref{eq:PvH} and \eqref{eq:Palpha} proves \eqref{eq:key}.
\end{proof}

\begin{proposition}
\label{prop_prog_wave}
Let $V$ be the unique solution of~\eqref{eq_wave_lorentz}. Then, there exist functions $\alpha,v\in C^\infty(\diam)$, with $\alpha$ vanishing near $\gamma(I)$,
such that
\begin{equation}
  V=\alpha\,\delta(\varphi)+v\,H(\varphi)
  \qquad\text{in }\diam.
  \label{eq:main}
\end{equation}
Moreover, $\alpha$ satisfies
\begin{equation}
  2L(\alpha)+\theta\alpha=\chi
  \qquad\text{in }\dg,
  \label{eq:alpha}
\end{equation}
and there is $\delta_\alpha>0$ such that
\begin{equation}
  \alpha\ge\delta_\alpha\qquad\text{on }K.
  \label{eq:alpha-lower}
\end{equation}
\end{proposition}

\begin{proof}
\textbf{1. The ansatz and the transport equations.}
We look for a parametrix of the form
\begin{equation}\label{def_of_parametrix_r}
R_\infty=\alpha\,\delta(\varphi)+\widetilde v\,H(\varphi),
\qquad
\widetilde v\sim\sum_{k\ge0}w_k\,\varphi^k\quad\text{near }\Ht,
\end{equation}
with smooth functions $\alpha,w_0,w_1,\dots$, such that $PR_\infty-\chi\,\delta'(\varphi)$ is smooth. Then $V-R_\infty$ solves a wave equation with a smooth source, so it is smooth.

To find the equations for $\alpha$ and the $w_k$, let
\begin{equation*}
v_N=\sum_{k=0}^Nw_k\varphi^k .
\end{equation*}
Since $L\varphi=0$, \eqref{eq:PaF} gives
\begin{equation*}
P(w_k\varphi^k)=\varphi^kPw_k+k\varphi^{k-1}(2Lw_k+\theta w_k),
\qquad
(2L+\theta)v_N=\sum_{k=0}^N\varphi^k(2L+\theta)w_k .
\end{equation*}
As $\varphi^k\delta(\varphi)=0$ for $k\ge1$, Lemma~\ref{lem:key} gives
\begin{equation}\label{eq:ansatz}
\begin{aligned}
&P\bigl(\alpha\,\delta(\varphi)+v_N\,H(\varphi)\bigr)-\chi\,\delta'(\varphi)\\
&\quad=(2L\alpha+\theta\alpha-\chi)\,\delta'(\varphi)+(2Lw_0+\theta w_0+P\alpha)\,\delta(\varphi)\\
&\qquad+\sum_{k=1}^N\varphi^{k-1}\bigl(k(2Lw_k+\theta w_k)+Pw_{k-1}\bigr)H(\varphi)+\varphi^N(Pw_N)\,H(\varphi).
\end{aligned}
\end{equation}
So we impose the transport equations
\begin{equation}\label{eq:recursion}
\left\{
\begin{aligned}
&2L\alpha+\theta\alpha=\chi,\\
&2Lw_0+\theta w_0=-P\alpha,\\
&k\,(2Lw_k+\theta w_k)=-Pw_{k-1},\qquad k\ge1.
\end{aligned}
\right.
\end{equation}
Then only the last term of \eqref{eq:ansatz} remains. In particular,
\begin{equation}\label{eq:Pwk}
Pv_N=\varphi^N\,Pw_N .
\end{equation}
This error is only in $C^{N-1}$ and will be removed by letting $N\to\infty$ via Borel's theorem.

\medskip\noindent
\textbf{2. Solving the transport equations on all leaves.}
Each equation in \eqref{eq:recursion} has the form
\begin{equation*}
2La+\theta a=b,
\end{equation*}
with $b$ given by the previous equations. Since $L$ is tangent to the leaves, this is an ODE along each null generator. The generator through $x\in\dg$ is the null geodesic from the vertex $\gamma(f(x))$ to $x$. We solve it on every generator of every leaf $\mathcal H_s$, $s\in I$ and the solutions are then defined on all of $\dg$. This is important since $Pa$ involves derivatives of $a$ in directions transversal to $\mathcal H_\tau$.

Let $\lambda\colon(s_0,s_1)\to\dg$ be a generator, with $\dot\lambda=L\circ\lambda$ and $\lambda(s)$ tending to its vertex as $s\to s_0$. For $b$ vanishing near $\gamma(I)$, the unique solution along $\lambda$ that vanishes near the vertex is
\begin{equation}
  a(\lambda(s))
  =
  \frac{1}{2}
  \int_{s_0}^{s}
  b(\lambda(\sigma))
  \exp\!\left(
  -\frac{1}{2}\int_\sigma^s\theta(\lambda(r))\,dr
  \right)
  d\sigma .
  \label{eq:alpha-explicit}
\end{equation}
It is smooth on $\dg$, by smooth dependence of the flow of $L$ on the initial point.

It remains to check that the right-hand sides vanish near \textit{a whole neighbourhood} of $\gamma(I)$ in $\diam$. By \eqref{eq:alpha-explicit}, $a(x)$ depends only on $b$ along the generator segment from $\gamma(f(x))$ to $x$. By Lemma~\ref{lem_monotone}, the points $y$ of this segment satisfy
\begin{equation}\label{eq:segment}
\gamma(f(x))\le y\le x .
\end{equation}
Let $s\in I$. Since $\M$ is globally hyperbolic, there is $\varepsilon>0$ such that
\begin{equation*}
J^+(\gamma(s-\varepsilon))\cap J^-(\gamma(s+\varepsilon))\cap\supp\chi=\emptyset .
\end{equation*}
Let
\begin{equation*}
\mathcal U_s=I^+(\gamma(s-\varepsilon))\cap I^-(\gamma(s+\varepsilon)),
\end{equation*}
an open neighbourhood of $\gamma(s)$ on which $\chi=0$. If $x\in\mathcal U_s\cap\dg$, then $f(x)>s-\varepsilon$ by Lemma~\ref{lem_monotone}(a). So, by \eqref{eq:segment}, every point $y$ of the generator segment to $x$ satisfies
\begin{equation*}
\gamma(s-\varepsilon)\ll\gamma(f(x))\le y\le x\ll\gamma(s+\varepsilon),
\end{equation*}
and hence $y\in\mathcal U_s$. Now $\alpha$ is determined by~\eqref{eq:alpha-explicit} for $b=\chi$ and it vanishes in $\cup_{s \in I} \mathcal U _s$ which is a neighbourhood of $\gamma(I)$. Moreover, the lower bound~\eqref{eq:alpha-lower} is an immediate consequence of the same formula
\begin{equation*}
\alpha(x)=\frac12\int_{s_0}^{0}\chi(\lambda(\sigma))\exp\Bigl(-\frac12\int_\sigma^0\theta(\lambda(r))\,dr\Bigr)d\sigma,
\end{equation*}
combined with the assumption~\ref{assumption_sources} on $\chi$. Indeed, since each generator meets the set $\{\chi >0\}$ we have that $\alpha(x)>0$ for all $x \in K$ and hence the lower bound by compactness.

Once $\alpha$ is constructed the subsequent $w_k$ are determined by solving the transport equations~\eqref{eq:recursion}. They vanish near $\gamma(I)$, and we extend them by zero to $\diam$. In the same way, they vanish in the open set $\dg\setminus J^+(\supp\chi)$: by \eqref{eq:segment}, the generator segment to a point $x$ of this set lies in $J^-(x)$, hence in this set. Thus
\begin{equation*}
\supp\alpha\cup\bigcup_{k\ge0}\supp w_k\subset J^+(\supp\chi).
\end{equation*}

\medskip\noindent
\textbf{3. The parametrix.}
Since $d\varphi=df\neq0$, $\varphi$ can be used as a local coordinate near $\Ht$. By Borel's theorem~\cite[Theorem 1.2.6]{hormander1990}, applied in the variable $\varphi$, there is a function $\widetilde v_0$, smooth in a neighbourhood $\mathcal O$ of $\Ht$ in $\dg$, with prescribed Taylor coefficients in $\varphi$:
\begin{equation}
  \widetilde v_0-\sum_{k<N}w_k\varphi^k\in\varphi^N\,C^\infty(\mathcal O)\qquad\text{for every }N .
  \label{eq:borel}
\end{equation}
The function $\widetilde v_0$ is not yet what we need since it is defined only in $\mathcal O$, and only its Taylor coefficients on $\Ht$ are prescribed. We use a cut-off to extend to the whole $\diam$.

Let $\mathcal N\subset\diam$ be an open neighbourhood of $\gamma(I)$ on which all $w_k$ vanish. We pick $\vartheta\in C^\infty(\diam)$ with
\begin{equation*}
\vartheta=1\quad\text{near }\Ht\setminus\mathcal N,
\qquad
\supp\vartheta\subset\mathcal O .
\end{equation*}
Since $\supp\vartheta$ is closed in $\diam$ and does not meet $\gamma(I)$, $\vartheta$ vanishes near $\gamma(I)$. Let $\widetilde v=\vartheta\,\widetilde v_0$, extended by zero to $\diam$. Then
\begin{equation*}
\widetilde v\in C^\infty(\diam),
\qquad
\widetilde v=0\quad\text{near }\gamma(I).
\end{equation*}
The cut-off does not change the Taylor expansion on $\Ht$, so $\widetilde v$ still satisfies \eqref{eq:borel} near $\Ht$. By \eqref{eq:borel} and \eqref{eq:Pwk},
\begin{equation*}
P\widetilde v\in\varphi^NC^\infty\quad\text{near }\Ht,\ \text{for every }N,
\end{equation*}
that is, $P\widetilde v$ vanishes to infinite order on $\Ht=\{\varphi=0\}$. Now recalling that $R_\infty=\alpha \delta(\varphi)+\widetilde vH(\varphi)$ Lemma~\ref{lem:key} gives
\begin{equation}
  F_\infty
  =
  PR_\infty-\chi\,\delta'(\varphi)
  =
  (P\widetilde{v})\,H(\varphi)
  \in C^\infty(\diam).
  \label{eq:Finfty}
\end{equation}
The term $F_\infty$ is smooth because $P\widetilde v$ vanishes to infinite order on $\{\varphi=0\}$, and $F_\infty=0$ near $\gamma(I)$, since $\widetilde v=0$ there.

\medskip\noindent
\textbf{4. The smooth remainder.}
Set
\begin{equation}
  W=V-R_\infty,
  \qquad\text{so that}\qquad
  PW=-F_\infty\in C^\infty(\diam)
  \label{eq:W}
\end{equation}
by \eqref{eq_wave_lorentz} and \eqref{eq:Finfty}. Since $\diam$ is globally hyperbolic, it has a Cauchy temporal function $t_{\diam}$. Let
\begin{equation*}
C=\supp\chi\cup\{\gamma(\tau)\},
\end{equation*}
a compact set, choose $t_0<\min_Ct_{\diam}$, and let
\begin{equation*}
\Sigma=\{t_{\diam}=t_0\},
\end{equation*}
an initial spacelike Cauchy hypersurface in the past of $C$. Since $t_{\diam}$ increases along future-directed causal curves,
\begin{equation*}
J^+(C)\subset\{t_{\diam}\ge\min_Ct_{\diam}\},
\end{equation*}
so $J^+(C)$ does not meet the neighbourhood $\{t_{\diam}<\min_Ct_{\diam}\}$ of $\Sigma$. Now:
\begin{equation*}
\begin{aligned}
&\supp V\subset J^+(\supp(\chi\,\delta'(\varphi)))\subset J^+(C)&&\text{by \eqref{eq_wave_lorentz}},\\
&\supp\alpha\subset J^+(\supp\chi)\subset J^+(C)&&\text{by Step~2},\\
&\supp\bigl(\widetilde v\,H(\varphi)\bigr)\subset\overline{\{f\ge\tau\}}\subset J^+(\gamma(\tau))\subset J^+(C)&&\text{by Lemma~\ref{lem_monotone}(a)}.
\end{aligned}
\end{equation*}
Hence
\begin{equation*}
W=V-\alpha\,\delta(\varphi)-\widetilde v\,H(\varphi)=0\quad\text{near }\Sigma,
\end{equation*}
so $W$ solves $PW=-F_\infty$ with zero Cauchy data on $\Sigma$ and a smooth source. By well-posedness of the Cauchy problem we get that it is the unique smooth solution of this wave equation and in particular
\begin{equation}
  W\in C^\infty(\diam).
  \label{eq:W-smooth}
\end{equation}

\medskip\noindent
\textbf{5. Conclusion.}
Since $\supp V\subset J^+(\supp\chi\cap\Ht)$, Lemma~\ref{lem_monotone}(b) gives
\begin{equation*}
\supp V\cap\dg\subset\{\varphi\ge0\}.
\end{equation*}
By construction, $\supp R_\infty\subset\{\varphi\ge0\}$. So the smooth function $W=V-R_\infty$ vanishes on $\{\varphi<0\}$, that is,
\begin{equation*}
W=W\,H(\varphi).
\end{equation*}
Going back to $V=R_\infty+W$, we obtain
\begin{equation*}
V=\alpha\,\delta(\varphi)+v\,H(\varphi),\qquad v=\widetilde v+W\in C^\infty(\diam),
\end{equation*}
which finishes the proof of~\eqref{eq:main}.
\end{proof}

\begin{remark}[The whole diamond is foliated]\label{rem_whole_diamond}
It is important that the null cones foliate the whole of $\dg$, and not only a neighbourhood of $\overline\uu$. This is so even if $\supp\chi\subset\uu$. Indeed, by finite speed of propagation, $V_\tau$ in $\uu$ depends on the geometry of the region
\begin{equation*}
J^+(\supp\chi)\cap J^-(\overline\uu),
\end{equation*}
which in general is not contained in $\overline\uu$: causal curves, and in particular the null generators of $\mathcal H_\tau$, may leave $\overline\uu$ and come back. Without the foliation outside $\overline\uu$, the cone $\mathcal H_\tau$ could develop caustics there, and singularities of $V_\tau$ could re-enter $\uu$ away from $\mathcal H_\tau$. The progressive wave representation of Proposition~\ref{prop_prog_wave} would then fail in $\uu$.
\end{remark}

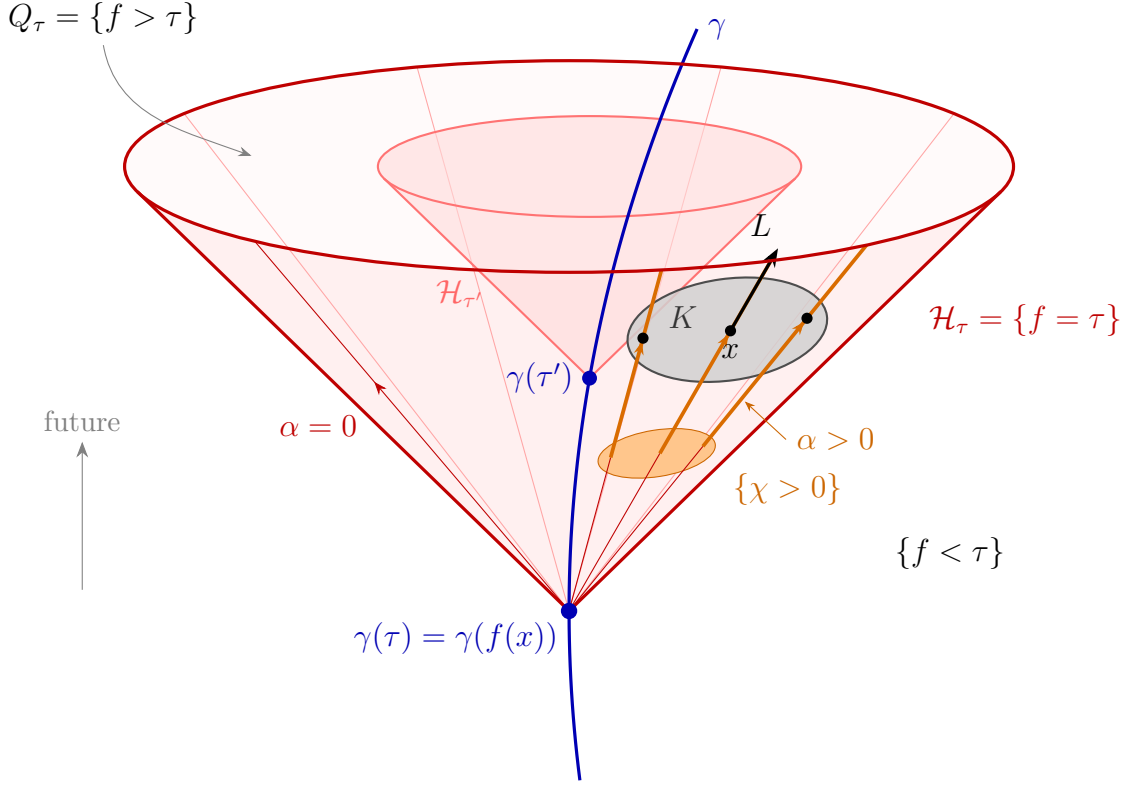
\begin{figure}[ht]
\centering
\begin{tikzpicture}[scale=1.4,
  gen/.style={red!75!black, thin},
  genback/.style={red!35, thin},
  flow/.style={postaction={decorate},
    decoration={markings, mark=at position 0.62 with {\arrow{Stealth[length=5pt]}}}}]

\def\Ra{4.2}\def\Hh{4.2}\def\Rb{1.0}
\pgfmathsetmacro\ps{asin(-\Rb/\Hh)}
\pgfmathsetmacro\sx{\Ra*cos(\ps)}
\pgfmathsetmacro\sy{\Hh+\Rb*sin(\ps)}
\def\vx{0.1936}\def\vy{2.2}
\pgfmathsetmacro\aa{\Hh-\vy}
\pgfmathsetmacro\bb{\Rb*\aa/\Ra}
\pgfmathsetmacro\pps{asin(-\bb/\aa)}
\def\fc{0.45}\def\fk{0.80}

\fill[red!6] (0,0) -- (\sx,\sy)
  arc[start angle=\ps, end angle=180-\ps, x radius=\Ra, y radius=\Rb] -- cycle;
\fill[red!2] (0,\Hh) ellipse[x radius=\Ra, y radius=\Rb];

\foreach \t in {30,70,110,150}
  \draw[genback] (0,0) -- ({\Ra*cos(\t)},{\Hh+\Rb*sin(\t)});

\fill[red!14, opacity=0.6] (\vx,\vy) -- ({\vx+\aa*cos(\pps)},{\Hh+\bb*sin(\pps)})
  arc[start angle=\pps, end angle=180-\pps, x radius=\aa, y radius=\bb] -- cycle;
\draw[red!55, thick] ({\vx+\aa*cos(\pps)},{\Hh+\bb*sin(\pps)}) -- (\vx,\vy)
  -- ({\vx-\aa*cos(\pps)},{\Hh+\bb*sin(\pps)});
\draw[red!55, thick] (\vx,\Hh) ellipse[x radius=\aa, y radius=\bb];

\draw[blue!70!black, very thick, domain=-1.6:5.5, samples=60]
  plot ({0.04*\x*\x},{\x});
\node[blue!70!black, right] at ({0.04*5.5*5.5},5.5) {$\gamma$};

\coordinate (Kc) at ($(0,0)!\fk!({\Ra*cos(297)},{\Hh+\Rb*sin(297)})+(-0.03,0.01)$);
\fill[black!18, opacity=0.85, rotate around={8:(Kc)}] (Kc) ellipse[x radius=0.95, y radius=0.48];
\draw[black!70, thick, rotate around={8:(Kc)}] (Kc) ellipse[x radius=0.95, y radius=0.48];

\coordinate (Cc) at ($(0,0)!\fc!({\Ra*cos(297)},{\Hh+\Rb*sin(297)})+(-0.03,0)$);
\fill[orange!45, draw=orange!80!black, rotate around={9:(Cc)}] (Cc) ellipse[x radius=0.56, y radius=0.22];

\foreach \t in {282,297,312}{
  \coordinate (R) at ({\Ra*cos(\t)},{\Hh+\Rb*sin(\t)});
  \draw[gen] (0,0) -- ($(0,0)!\fc!(R)$);
  \draw[orange!85!black, line width=1.4pt, flow] ($(0,0)!\fc!(R)$) -- (R);
  \fill ($(0,0)!\fk!(R)$) circle (1.5pt);
}
\draw[gen, flow] (0,0) -- ({\Ra*cos(225)},{\Hh+\Rb*sin(225)});

\coordinate (P) at ({\Ra*cos(297)},{\Hh+\Rb*sin(297)});
\coordinate (X) at ($(0,0)!\fk!(P)$);
\draw[-{Stealth[length=7pt]}, very thick, black] (X) -- ($(X)!0.9cm!(P)$)
  node[above left, xshift=2pt] {$L$};
\node[below, yshift=-1pt] at (X) {$x$};

\draw[red!75!black, very thick] (\sx,\sy) -- (0,0) -- (-\sx,\sy);
\draw[red!75!black, very thick] (0,\Hh) ellipse[x radius=\Ra, y radius=\Rb];

\fill[blue!70!black] (0,0) circle (2.2pt);
\node[blue!70!black, below left] at (0,0) {$\gamma(\tau)=\gamma(f(x))$};
\fill[blue!70!black] (\vx,\vy) circle (2pt);
\node[blue!70!black, left, xshift=-2pt] at (\vx,\vy) {$\gamma(\tau')$};

\node[black!80] at ($(Kc)+(-0.42,0.12)$) {$K$};
\node[orange!80!black, right] at ($(Cc)+(0.62,-0.34)$) {$\{\chi>0\}$};
\node[orange!80!black, right] (alab) at (2.05,1.62) {$\alpha>0$};
\draw[orange!80!black, -{Stealth[length=4pt]}] (alab.west) -- (1.66,1.98);
\node[red!75!black, left] at (-1.9,1.75) {$\alpha=0$};
\node[red!75!black, right] at (3.3,2.75) {$\mathcal H_\tau=\{f=\tau\}$};
\node[red!60, left] at (-0.7,3.0) {$\mathcal H_{\tau'}$};
\node (Qlab) at (-4.4,5.6) {$Q_\tau=\{f>\tau\}$};
\draw[-{Stealth[length=5pt]}, gray] (Qlab.south) to[out=-80, in=160] (-3.0,4.3);
\node at (3.6,0.55) {$\{f<\tau\}$};

\draw[-{Stealth[length=6pt]}, gray] (-4.6,0.2) -- (-4.6,1.6) node[above] {\small future};
\end{tikzpicture}
\caption{The leaf $\mathcal H_\tau=\{f=\tau\}$ is the future null cone of $\gamma(\tau)$, ruled by its null generators, the integral curves of $L=-\nabla_gf$ (arrows). The transport equations are solved across those null generators.}
\label{fig:generators}
\end{figure}

Combining Proposition~\ref{prop_prog_wave} with Lemma~\ref{lem:key}, we obtain that the unique solution of~\eqref{eq_wave_lorentz} can be decomposed as
\begin{equation}
    V=\alpha \delta(\varphi) + v H(\varphi),
\end{equation}
with $\alpha \in C^\infty(\diam)$ and $v \in C^\infty(\diam)$ satisfying
\begin{equation}
\label{system_a_v}
\left\{
\begin{aligned}
&(\Box_g + q)v = 0
&& \text{in } J^+(\mathcal{H}_\tau)\cap\dg, \\[0.6em]
&2L(\alpha) + \theta \alpha = \chi
&& \text{near } \mathcal{H}_\tau, \\[0.6em]
&\alpha=0 && \text{near } \gamma(I), \\[0.6em]
&2L(v) + \theta v
= -(\Box_g + q)\alpha
&& \text{on } \mathcal{H}_\tau.
\end{aligned}
\right.
\end{equation}
Indeed, $PV=\chi\delta'(\varphi)$, so by \eqref{eq:key} and \eqref{eq:alpha} we have $(P\alpha+2L(v)+\theta v)\delta(\varphi)+(Pv)H(\varphi)=0$. This gives $Pv=0$ in $Q_\tau=\{f>\tau\}$, hence, by continuity, on $\{f\ge\tau\}$, which contains $J^+(\Ht)\cap\dg$ by Lemma~\ref{lem_monotone}(b). It also gives the fourth line.

\begin{remark}
    Notice that the fact that $\alpha$ vanishes near $\gamma(I)$ makes $\alpha\,\delta(\varphi)$ a well defined distribution on the whole diamond $\diam$.
\end{remark}

\section{The Carleman estimate}\label{sec_carleman}

\subsection{Setup}\label{sec_setup}

We prove in this section the Carleman estimate that will allow to propagate the information from a leaf $\mathcal H_\tau$ of the foliation to the boundary of $M$, which is the data region. Let $\tau\in\mathcal R$. The estimate is proved on the compact set
\[
X_\tau:=\overline{Q_\tau}\cap\overline\uu,\qquad \partial X_\tau=(\mathcal H_\tau\cap\overline\uu)\cup\Gamma_\tau,\qquad \Gamma_\tau=\overline{Q_\tau}\cap\partial\uu .
\]
Since $\tau$ is a regular value of $f|_{\partial\uu}$, $\mathcal H_\tau$ intersects $\partial \uu$ transversally. Thus, $X_\tau$ is a compact manifold with corners, the corner being $\mathcal H_\tau\cap\partial\uu$, and the standard integration by parts formulae hold in $X_\tau$~\cite{lee2002}.

 Let $d\sigma_h$ and $d\sigma_h^\Gamma$ be the measures induced by the Riemannian metric $h$ on $\mathcal H_\tau$ and on $\Gamma_\tau$. For $u\in C^1(X_\tau)$ we write $\partial_\nu u=du(\nu)$ and $\nabla_hu=(\partial_\nu u)\nu+\nabla_h^{\tan}u$ on $\Gamma_\tau$, where $\nabla_h^{\tan}u$ is determined by $u|_{\Gamma_\tau}$. We write $\|\cdot\|_{Q_\tau}$ for the $L^2$ norm over $Q_\tau\cap\uu$ with respect to $dV_g=dV_h$, $|\cdot|_{\mathcal H_\tau}$ for the $L^2$ norm over $\mathcal H_\tau\cap\overline\uu$ with respect to $d\sigma_h$, and $|\cdot|_{\Gamma_\tau}$ for the $L^2$ norm over $\Gamma_\tau$ with respect to $d\sigma_h^\Gamma$.

 \medskip

 Recall that $P=\Box_g+q$. We have the following.

\begin{theorem}\label{thm_carleman}
Let $\psi\in C^\infty(\diam)$ satisfy $g( \nabla \psi,\nabla \psi)<0$ and such that $\psi$ increases along future-directed timelike curves. Let $\tilde\psi = e^{\zeta\psi}$. Then there is $\zeta_0>0$ such that the following holds for every $\zeta\ge\zeta_0$. There are $C,\sigma_0>0$ such that for every $u\in C^\infty(X_\tau)$ one has,
\begin{multline*}
\sigma^3|e^{\sigma\tilde\psi}u|^2_{\mathcal H_\tau} + \sigma|e^{\sigma\tilde\psi}Lu|^2_{\mathcal H_\tau} + \sigma^3\|e^{\sigma\tilde\psi}u\|^2_{Q_\tau} + \sigma\|e^{\sigma\tilde\psi}\nabla_hu\|^2_{Q_\tau}\\
\le\; C\,\|e^{\sigma\tilde\psi}Pu\|^2_{Q_\tau} \;+\; C\Bigl(\sigma^3|e^{\sigma\tilde\psi}u|^2_{\Gamma_\tau} + \sigma|e^{\sigma\tilde\psi}\partial_\nu u|^2_{\Gamma_\tau} + \sigma|e^{\sigma\tilde\psi}\nabla_h^{\tan}u|^2_{\Gamma_\tau}\Bigr),
\end{multline*}
for all $\sigma \geq \sigma_0$.
\end{theorem}

\begin{remark}
    The constants $C,\sigma_0, \zeta_0$ appearing in Theorem~\ref{thm_carleman} depend on $M,g,h,\psi, \norm{q}_{L^\infty}$ but are independent of the foliation parameter $\tau$.
\end{remark}

\begin{remark}\label{rem_powers}
The Carleman estimate of Theorem~\ref{thm_carleman} is closely related to~\cite[Proposition~2.5]{oksanen2026semiglobaluniquenesslorentziancalderon}. There, the terms on the null hypersurface come with the same powers $\sigma^3$ and $\sigma$, but the bulk terms are stronger: $\sigma^3$ and $\sigma$ are replaced by $\sigma^4$ and $\sigma^2$. Moreover, the weight is $e^{\sigma\psi}$ itself, without convexification. This reflects the fact that the level sets of $\psi$ are spacelike, since $d\psi$ is timelike. So information is propagated across spacelike hypersurfaces, a direction in which the Cauchy problem for the wave equation is well posed. As a consequence, the estimate is semiclassically elliptic, and not only subelliptic contrasting to a standard Carleman estimate. See as well \cite[Remark~C.9]{oksanen2026semiglobaluniquenesslorentziancalderon}.

Notice that the estimate of~\cite[Proposition~2.5]{oksanen2026semiglobaluniquenesslorentziancalderon} does not apply in our setting. There, the lateral boundary is timelike and the Cauchy data vanish on it, while here $\partial\uu$ has no causal type and its Cauchy data appear on the right-hand side. We expect however that Theorem~\ref{thm_carleman} also holds without the loss of a half derivative, that is, with $\sigma^4$ and $\sigma^2$ in the bulk terms. We do not pursue this, since Theorem~\ref{thm_carleman} suffices for the proof of Theorem~\ref{thm_main}, and its proof is relatively short and self-contained.
\end{remark}

\subsection{Preliminary lemmas}

In this subsection and the next one, $\tau$ is fixed and we drop it from the notation for the sake of readability.

The next lemma is a general pointwise identity which holds on any semi-Riemannian manifold. We refer to~\cite[Lemma 2.1]{FO25explicit} for a proof. We write $G(X)=g(X,X)$, $\Delta_g=\div_g\nabla^g=-\Box_g$, and $D^2$ stands for the Hessian.

\begin{lemma}\label{lem_carleman_eq}
Let $U\subset\M$ be open and $\ell, \lambda,v$ be smooth real valued functions on $U$. Then, in $U$,
\begin{align*}
\tfrac12\bigl|e^\ell\Box_g(e^{-\ell}v)\bigr|^2
&= Q_+(\nabla v) + Q_-(\nabla\ell)v^2 + \div(B) + R \\
&\quad + \tfrac12\bigl|(\Delta_g+q_\ell+\lambda)v\bigr|^2 + \tfrac12\bigl|(S_\ell+\lambda)v\bigr|^2,
\end{align*}
where, writing $a=\lambda-\Delta_g\ell$,
\[
Q_\pm(X)=2D^2\ell(X,X)\pm a\,G(X), \qquad q_\ell = G(\nabla\ell)-\Delta_g\ell, \qquad S_\ell = 2\nabla\ell,
\]
\[
B = -\bigl((S_\ell+\lambda)v\bigr)\nabla v + \bigl(G(\nabla v)-(q_\ell+\lambda)v^2\bigr)\nabla\ell + \tfrac12 v^2\nabla\lambda,
\]
\[
R = \bigl(\div(a\nabla\ell) - a\lambda - \tfrac12\Delta_g\lambda\bigr)v^2.
\]
\end{lemma}
Here $S_\ell$, a first-order operator, and $\lambda$ are notation local to this lemma and to Section~\ref{sec_proof_carleman}.

\medskip

The following lemma is the key positivity property. It shows that the boundary terms on the leaf $\mathcal H_\tau$ in the Carleman identity of Lemma~\ref{lem_carleman_eq} have a favourable sign. The same quadratic form appears in \cite[Lemma~24.1.2]{hormander2007}, where its positivity is proved when $Y_1$ and $Y_2$ are both timelike. Here $Y_2$ is also allowed to be null, which is the case needed on the leaf $\mathcal H_\tau$.

\begin{lemma}\label{lem_algebra}
Let $V$ be a finite dimensional vector space with a Lorentzian metric $g$.
Let $Y_1, Y_2, Z \in V$ and suppose that $Y_1$ is timelike, $Y_2$ is causal and that $Y_1$ and $Y_2$ are in the same timecone. Then
\begin{equation}\label{abs_nonpos}
2 g(Z, Y_1) g(Z, Y_2) - g(Z, Z) g(Y_1, Y_2) \ge 0.
\end{equation}
The above inequality is strict or $Z = cY_2$ for some $c \in \R$.
\end{lemma}
Observe that \eqref{abs_nonpos} holds with equality if $Y_2$ is null and $Z = cY_2$ for any $c \in \R$.
\begin{proof}
We write $Z = s Y_1 + Z'$ where $g(Y_1, Z') = 0$, and consider
\[
\tilde Z = 2 g(Z, Y_1) Z - g(Z, Z) Y_1.
\]
Observe that \eqref{abs_nonpos} is equivalent to $g(\tilde Z, Y_2) \ge 0$.

The case $Z = 0$ is trivial. Let us assume that $Z \ne 0$. We have
\begin{align*}
g(\tilde Z, Y_1)
&= 2 g(Z, Y_1)^2 - g(Z, Z) g(Y_1, Y_1)\\
&= 2 s^2 g(Y_1, Y_1)^2 - \bigl(s^2 g(Y_1, Y_1) + g(Z', Z')\bigr) g(Y_1, Y_1)\\
&= s^2 g(Y_1, Y_1)^2 - g(Z', Z') g(Y_1, Y_1) > 0,
\end{align*}
since $Z'$ is spacelike due to the orthogonality $g(Y_1, Z') = 0$, and since $s \ne 0$ or $Z' \ne 0$ due to $Z \ne 0$. In particular, $\tilde Z \ne 0$. Moreover,
\begin{equation}\label{z_causal}
\begin{aligned}
g(\tilde Z, \tilde Z)
&= 4 g(Z, Y_1)^2 g(Z, Z) - 4 g(Z, Y_1)^2 g(Z, Z) + g(Z, Z)^2 g(Y_1, Y_1)\\
&= g(Z, Z)^2 g(Y_1, Y_1) \le 0.
\end{aligned}
\end{equation}
Thus $\tilde Z$ is causal. By $g(\tilde Z, Y_1) > 0$ the vectors $Y_1$ and $\tilde Z$ lie in opposite timecones, as do $Y_2$ and $\tilde Z$, since $Y_1$ and $Y_2$ lie in the same timecone. Hence $g(\tilde Z, Y_2) \ge 0$.

It remains to consider the case $g(\tilde Z, Y_2) = 0$. As $\tilde Z$ and $Y_2$ are causal, the orthogonality implies that $Y_2$ is null and $\tilde Z = aY_2$ for some $a \in \R \setminus 0$. Hence $Z$ is null due to \eqref{z_causal}. As \eqref{abs_nonpos} holds with equality and $Z$ is null, we have
\[
g(Z, Y_1) g(Z, Y_2) = 0.
\]
As $Y_1$ is timelike and $Z$ is null, we have $g(Z, Y_1) \ne 0$. Hence $g(Z, Y_2) = 0$. As both $Z$ and $Y_2$ are null, there is $c \ne 0$ such that $Z = cY_2$.
\end{proof}

In the proof of Theorem~\ref{thm_carleman}, Lemma~\ref{lem_algebra} enters through the following corollary.

\begin{corollary}\label{cor_algebra}
Let $V$ be a finite dimensional vector space with a Lorentzian metric $g$.
Let $Y_1, Y_2, Z \in V$ and suppose that $Y_1$ is timelike, $Y_2$ is causal and that $Y_1$ and $Y_2$ are in the same timecone.
Let a subspace $W \subset V$ satisfy $V = \R Y_2 \oplus W$, and let $g^+$ be an inner product on $W$.
Then there is $\epsilon > 0$ such that
\begin{equation}\label{abs_nonpos_compl}
\epsilon\, g^+(Z', Z') \le 2 g(Z, Y_1) g(Z, Y_2) - g(Z, Z) g(Y_1, Y_2),
\end{equation}
where $Z = sY_2 + Z' \in \R Y_2 \oplus W$.
\end{corollary}

\begin{proof}
Let $\Phi(Z)$ denote the right-hand side of \eqref{abs_nonpos_compl}, a quadratic form in $Z$. By Lemma~\ref{lem_algebra}, $\Phi$ can only degenerate in the direction $Y_2$, so we distinguish the cases $Y_2$ timelike and $Y_2$ null.

If $Y_2$ is timelike, then $\Phi(cY_2)=c^2g(Y_2,Y_1)g(Y_2,Y_2)>0$ for $c\neq0$, so $\Phi$ is positive definite on $V$ by Lemma~\ref{lem_algebra}, and \eqref{abs_nonpos_compl} follows. If $Y_2$ is null, then $\Phi(Y_2)=0$, and the polar form of $\Phi$ satisfies
\begin{align*}
\Phi(Y_2,Z')&=g(Y_2,Y_1)g(Z',Y_2)+g(Z',Y_1)g(Y_2,Y_2)-g(Y_2,Z')g(Y_1,Y_2)\\
&=g(Z',Y_1)g(Y_2,Y_2)=0;
\end{align*}
hence $\Phi(sY_2+Z')=\Phi(Z')$. By Lemma~\ref{lem_algebra}, $\Phi(Z')>0$ unless $Z'\in\R Y_2\cap W=\{0\}$. So $\Phi$ is positive definite on $W$, and $\epsilon$ exists by compactness of the unit sphere of $g^+$.
\end{proof}

The next lemma gives an integration by parts formula on $\overline Q\cap\overline\uu$ and makes use of the surface measure induced by the Riemannian metric $h$.

\begin{lemma}\label{lem_stokes}
Let $X=\overline Q\cap\overline\uu$ be as in Section~\ref{sec_setup}, and let $\nu$ be the outer unit $h$-normal of $\uu$ along $\Gamma$. Then for every smooth vector field $B$ on $X$,
\[
\int_{Q\cap\uu}\div_g(B)\,dV \;=\; -\int_{\mathcal H\cap\overline\uu}\frac{\langle B,d\varphi\rangle}{|\nabla_h\varphi|_h}\,d\sigma_h \;+\; \int_\Gamma h(B,\nu)\,d\sigma_h^\Gamma .
\]
\end{lemma}
\begin{proof}
By Section~\ref{sec_setup}, $X$ is a compact manifold with corners and $\partial X=(\mathcal H\cap\overline\uu)\cup\Gamma$. Since $dV_g=dV_h$, we have $\div_g(B)\,dV_g=d(\iota_B\,dV_h)$. Stokes' theorem on $X$  is valid on manifolds with corners~\cite{lee2002} and gives
\[
\int_X\div_g(B)\,dV_g=\int_{\partial X}\iota_B\,dV_h .
\]

On each of the two smooth pieces of $\partial X$, let $\eta$ be the outer unit $h$-normal. Thus $\eta=\nu$ on $\Gamma$, and $\eta=\nu_{\mathcal H}:=-\nabla_h\varphi/|\nabla_h\varphi|_h$ on $\mathcal H$, since $X\subset\{\varphi\ge0\}$. Write $B=h(B,\eta)\,\eta+B^{\tan}$, with $B^{\tan}$ tangent to the boundary. Let $e_1,\dots,e_n$ be an $h$-orthonormal frame of the boundary, positively oriented with respect to the boundary orientation. Then $(\eta,e_1,\dots,e_n)$ is $h$-orthonormal and positively oriented, so
\[
dV_h(\eta,e_1,\dots,e_n)=1=d\sigma_h(e_1,\dots,e_n).
\]
Since $dV_h$ is alternating and $B^{\tan}$ is tangential, $dV_h(B^{\tan},e_1,\dots,e_n)=0$. Hence
\[
(\iota_B\,dV_h)(e_1,\dots,e_n)=dV_h(B,e_1,\dots,e_n)=h(B,\eta),
\]
that is, $\iota_B\,dV_h=h(B,\eta)\,d\sigma_h$ on each piece. This gives the formula, since on $\mathcal H$
\[
h(B,\nu_{\mathcal H})=-\frac{h(B,\nabla_h\varphi)}{|\nabla_h\varphi|_h}=-\frac{\langle B,d\varphi\rangle}{|\nabla_h\varphi|_h}.
\]
\end{proof}

\subsection{Proof of the Carleman estimate}\label{sec_proof_carleman}
\begin{proof}[Proof of Theorem~\ref{thm_carleman}]

We first prove the estimate with $\Box_gu$ in place of $Pu$, and for real valued $u$. Since $\Box_g$ is real, this estimate then holds for complex valued $u$, by applying it to $\operatorname{Re}u$ and $\operatorname{Im}u$. Constants $c,C>0$ do not depend on $\sigma$, $u$ or $\tau$, and may change from line to line. All pointwise bounds below are obtained from continuous functions on the compact set $\overline\uu$.

\medskip\noindent
\textbf{1. The bulk terms.}
Fix $\mu,\delta>0$ with $|D^2\psi(Z,Z)|\le\mu|Z|_h^2$ and $|G(\nabla\psi)|\ge\delta$ on $\overline\uu$. Let $\zeta\ge\zeta_0$, with $\zeta_0$ fixed below, and set $\tilde\psi=e^{\zeta\psi}$ and $\ell=\sigma\tilde\psi$. Then $\nabla\ell=\sigma\zeta\tilde\psi\,\nabla\psi$ and
\begin{equation}\label{eq_hess_ell}
D^2\ell(Z,Z)=\sigma\zeta\tilde\psi\bigl(D^2\psi(Z,Z)+\zeta\,g(\nabla\psi,Z)^2\bigr).
\end{equation}
Let $v=e^\ell u$, $a=\sigma\zeta^2\delta\tilde\psi$ and $\lambda=a+\Delta_g\ell$. Lemma~\ref{lem_carleman_eq} with $U=Q\cap\uu$, without its two squares, gives
\begin{equation}\label{star}
Q_+(\nabla v)+Q_-(\nabla\ell)\,v^2\le\tfrac12\bigl|e^\ell\Box_gu\bigr|^2-\div(B)-R .
\end{equation}
By \eqref{eq_hess_ell},
\[
Q_+(Z)=2\sigma\zeta\tilde\psi\,D^2\psi(Z,Z)+2\sigma\zeta^2\tilde\psi\bigl(g(\nabla\psi,Z)^2+\tfrac\delta2G(Z)\bigr).
\]
Write $Z=s\nabla\psi+Z'$ with $g(\nabla\psi,Z')=0$; then $Z'$ is spacelike and
\[
g(\nabla\psi,Z)^2+\tfrac\delta2G(Z)=s^2|G(\nabla\psi)|\bigl(|G(\nabla\psi)|-\tfrac\delta2\bigr)+\tfrac\delta2G(Z')\ge\tfrac{\delta^2}2s^2+\tfrac\delta2G(Z'),
\]
which yields $Q_+(Z)\ge2\sigma\zeta\tilde\psi\,(\zeta c-\mu)\,|Z|_h^2$.
Since $G(\nabla\ell)<0$, we have $-a\,G(\nabla\ell)\ge0$, and
\[
Q_-(\nabla\ell)\ge2D^2\ell(\nabla\ell,\nabla\ell)=2\sigma^3\zeta^3\tilde\psi^3\bigl(D^2\psi(\nabla\psi,\nabla\psi)+\zeta\,G(\nabla\psi)^2\bigr)\ge2\sigma^3\zeta^3\tilde\psi^3\bigl(\zeta\delta^2-\mu|\nabla\psi|_h^2\bigr).
\]
Taking $\zeta_0$ large enough, both brackets are positive on $\overline\uu$ for $\zeta\ge\zeta_0$. Moreover $a$, $\lambda$ and $\nabla\ell$ are $O(\sigma)$. Thus
\begin{equation}\label{eq_bulk}
Q_+(Z)\ge c\sigma|Z|_h^2,\qquad Q_-(\nabla\ell)\ge c\sigma^3,\qquad |R|\le C\sigma^2v^2\qquad\text{on }\overline\uu .
\end{equation}
Integrating \eqref{star} over $Q\cap\uu$, using Lemma~\ref{lem_stokes} and \eqref{eq_bulk}, and absorbing $C\sigma^2\|v\|^2_Q$ for $\sigma_0$ large enough,
\begin{equation}\label{eq_integrated}
c\sigma\|\nabla_h v\|^2_Q+c\sigma^3\|v\|^2_Q-\int_{\mathcal H\cap\overline\uu}\frac{\langle B,d\varphi\rangle}{|\nabla_h\varphi|_h}\,d\sigma_h\le\tfrac12\|e^\ell\Box_gu\|^2_Q+\Bigl|\int_\Gamma h(B,\nu)\,d\sigma_h^\Gamma\Bigr| .
\end{equation}

\medskip\noindent
\textbf{2. The boundary terms.}
With $\beta:=\Delta_g\tilde\psi$ and $\Upsilon:=\zeta^2\delta\tilde\psi+\beta$, we have $\lambda=\sigma\Upsilon$, $q_\ell+\lambda=\sigma^2\zeta^2\tilde\psi^2G(\nabla\psi)+\sigma\zeta^2\delta\tilde\psi$ and $S_\ell v=2\sigma\zeta\tilde\psi\,g(\nabla\psi,\nabla v)$. Hence $B=B_{\rm quad}+B_{\rm mix}+B_0$, where
\begin{align*}
B_{\rm quad}&=\sigma\zeta\tilde\psi\bigl(G(\nabla v)\nabla\psi-2g(\nabla\psi,\nabla v)\nabla v\bigr),\\
B_{\rm mix}&=-\sigma\Upsilon\,v\nabla v,\\
B_0&=-\bigl(\sigma^3\zeta^3\tilde\psi^3G(\nabla\psi)+\sigma^2\zeta^3\delta\tilde\psi^2\bigr)v^2\nabla\psi+\tfrac\sigma2v^2\nabla\Upsilon .
\end{align*}
The vector $\nabla\psi$ is timelike and $\nabla f=-L$ is null, and both are past-directed (Lemma~\ref{lem_monotone}(c)). Hence
\begin{equation}\label{eq_psi_f}
g(\nabla\psi,\nabla f)<0\qquad\text{on }\overline\uu .
\end{equation}

On the null boundary $\mathcal H\cap\overline\uu$ we use $\langle Z,d\varphi\rangle=g(Z,\nabla f)$. The terms coming from $B_0$ and $B_{\rm mix}$ are easy. The main term comes from $B_{\rm quad}$, and requires Lemma~\ref{lem_algebra}.

We start with the two easy terms. Since $G(\nabla\psi)<0$, \eqref{eq_psi_f} gives $G(\nabla\psi)g(\nabla\psi,\nabla f)\ge c$ on $\overline\uu$. The other terms in $B_0$ carry at most $\sigma^2$, so
\[
-\langle B_0,d\varphi\rangle=\sigma^3\zeta^3\tilde\psi^3\,G(\nabla\psi)g(\nabla\psi,\nabla f)\,v^2+O(\sigma^2)v^2\ge c\sigma^3v^2-C\sigma^2v^2 .
\]
Since $g(\nabla v,\nabla f)=-Lv$,
\[
|\langle B_{\rm mix},d\varphi\rangle|=\sigma|\Upsilon||v||Lv|\le C\sigma^2v^2+C(Lv)^2 .
\]

We now turn to the main term. We have
\[
-\langle B_{\rm quad},d\varphi\rangle=\sigma\zeta\tilde\psi\,\Phi(\nabla v),
\qquad
\Phi(Z)=2g(Z,\nabla\psi)g(Z,\nabla f)-G(Z)\,g(\nabla\psi,\nabla f).
\]
This is the quadratic form of Lemma~\ref{lem_algebra}, with
\[
Y_1=\nabla\psi\ \text{timelike},\qquad Y_2=\nabla f=-L\ \text{null},
\]
both past-directed. Since $Y_2$ is null, $\Phi(L)=0$, so $\Phi$ does not control the component of $\nabla v$ along $L$. We write
\[
\nabla v=sL+Z',\qquad h(Z',L)=0 .
\]
Corollary~\ref{cor_algebra}, with $W=L^{\perp_h}$ and $g^+=h|_W$, gives, with a constant uniform on $\overline\uu$,
\[
\Phi(\nabla v)\ge c\,|Z'|_h^2 .
\]
The component $Z'$ controls $Lv$. Since $G(L)=0$,
\[
Lv=g(\nabla v,L)=s\,G(L)+g(Z',L)=g(Z',L),\qquad\text{so}\qquad (Lv)^2\le C|Z'|_h^2 .
\]
Hence
\[
-\langle B_{\rm quad},d\varphi\rangle=\sigma\zeta\tilde\psi\,\Phi(\nabla v)\ge c\sigma|Z'|_h^2\ge c^\prime \sigma(Lv)^2 .
\]
So the main term controls the derivative of $v$ along the generators of $\mathcal H$.

Adding the three terms, for $\sigma_0$ large enough,
\[
-\langle B,d\varphi\rangle\ge c\bigl(\sigma^3v^2+\sigma(Lv)^2\bigr)\qquad\text{on }\mathcal H\cap\overline\uu ,
\]
and therefore
\begin{equation}\label{eq_bdry_H}
-\int_{\mathcal H\cap\overline\uu}\frac{\langle B,d\varphi\rangle}{|\nabla_h\varphi|_h}\,d\sigma_h\ge c\bigl(\sigma^3|v|^2_{\mathcal H}+\sigma|Lv|^2_{\mathcal H}\bigr).
\end{equation}

On the lateral boundary $\Gamma$, which has no causal type, we only bound absolute values. Since $|\nu|_h=1$, the decomposition of $B$ gives, for $\sigma\ge1$,
\[
|h(B,\nu)|\le|B|_h\le C\bigl(\sigma^3v^2+\sigma|v||\nabla_h v|_h+\sigma|\nabla_h v|_h^2\bigr)\le C\bigl(\sigma^3v^2+\sigma|\nabla_h v|_h^2\bigr),
\]
and $|\nabla_h v|_h^2=(\partial_\nu v)^2+|\nabla_h^{\tan}v|_h^2$ on $\Gamma$. Hence
\begin{equation}\label{eq_bdry_Gamma}
\Bigl|\int_\Gamma h(B,\nu)\,d\sigma_h^\Gamma\Bigr|\le C\bigl(\sigma^3|v|^2_\Gamma+\sigma|\partial_\nu v|^2_\Gamma+\sigma|\nabla_h^{\tan}v|^2_\Gamma\bigr).
\end{equation}

\medskip\noindent
\textbf{3. End of the proof.}
Inserting \eqref{eq_bdry_H} and \eqref{eq_bdry_Gamma} into \eqref{eq_integrated}, we get
\begin{multline}\label{eq_v_estimate}
\sigma^3\|v\|^2_Q+\sigma\|\nabla_h v\|^2_Q+\sigma^3|v|^2_{\mathcal H}+\sigma|Lv|^2_{\mathcal H}\\
\le C\|e^\ell\Box_gu\|^2_Q+C\bigl(\sigma^3|v|^2_\Gamma+\sigma|\partial_\nu v|^2_\Gamma+\sigma|\nabla_h^{\tan}v|^2_\Gamma\bigr).
\end{multline}
It remains to return to $u$ and $P$. We have
\[
e^{\sigma\tilde\psi}\nabla_h u=\nabla_h v-\sigma v\,\nabla_h\tilde\psi,
\qquad
e^{\sigma\tilde\psi}\Box_gu=e^{\sigma\tilde\psi}Pu-q\,v ,
\]
and similarly for $L$, $\partial_\nu$ and $\nabla_h^{\tan}$. The estimate of Theorem~\ref{thm_carleman} is obtained upon replacing $v$ by $e^{\sigma\tilde\psi}u$ in \eqref{eq_v_estimate}, using these identities and the triangle inequality, and absorbing the errors by taking $\sigma\ge\sigma_0$ with $\sigma_0$ large enough, depending on $\|q\|_{L^\infty(\overline\uu)}$.
\end{proof}

\begin{remark}
In fact the proof controls the whole tangential gradient of $u$ on $\mathcal H_\tau$, and not only $Lu$. Indeed, $T\mathcal H=L^{\perp_g}$, so $dv(X)=g(Z',X)$ for every $X\in T\mathcal H$. We keep only $Lu$, since it is the only tangential derivative that enters the recovery formula of Lemma~\ref{lem_recovery}.
\end{remark}

\section{Proof of the stability results}\label{sec_proof_main}

\subsection{Proof of Theorem~\ref{thm_main}}

Let $q_1,q_2$ be admissible potentials. We write $V_{j,\tau}=V_{q_j,\tau}$ and $v_{j,\tau}=v_{q_j,\tau}$. We define as well
\begin{equation}
   \bar q=q_1-q_2, \quad w_\tau=v_{1,\tau}-v_{2,\tau}.
\end{equation}

Notice that there is an integer $k$, depending only on $n$, such that for admissible $q$ (that is $q$ satisfying~\eqref{def_of_admissible}) one has
\begin{equation}
    \label{bound_admissible}
    \sup_{\tau\in[\tau_a,\tau_b]}\|v_{q,\tau}\|_{L^\infty(Q_\tau\cap K)}\le C,
\end{equation}
with $C$ depending only on $A$, $\chi$ and the geometry. Indeed, fix $N>(n+1)/2$, let $\alpha,w_0,\dots,w_N$ be as in Steps~1--2 of the proof of Proposition~\ref{prop_prog_wave}, and let
\[
v_N=\sum_{j=0}^Nw_j\,\varphi_\tau^j .
\]
The functions $\alpha,w_0,\dots,w_N$ do not depend on $\tau$, and their $C^{N+1}$ norms on compact sets are bounded in terms of $A$ if $k$ is large enough. By \eqref{eq:ansatz} and \eqref{eq:recursion}, the remainder
\[
e=V_{q,\tau}-\alpha\,\delta(\varphi_\tau)-v_N\,H(\varphi_\tau)
\]
vanishes outside $J^+(\supp\chi)$ and solves
\[
Pe=-\varphi_\tau^N\,(Pw_N)\,H(\varphi_\tau),
\]
a source in $C^{N-1}$ bounded in terms of $A$. Energy estimates for the wave equation~\cite[Chapter 23]{hormander2007} and the Sobolev embedding give $\|e\|_{L^\infty(K)}\le C$, uniformly in $\tau\in[\tau_a,\tau_b]$. Since $v_{q,\tau}=v_N+e$ on $Q_\tau$, this gives~\eqref{bound_admissible}. We fix $k$ such that~\eqref{bound_admissible} holds.

\medskip

The following lemma is an immediate consequence of the system~\eqref{system_a_v} and of the strictly positive lower bound~\eqref{eq:alpha-lower} for $\alpha$ on $K$. It is the crucial point for recovering the unknown $\bar q$ since it is multiplied by $\alpha$ on the leaf $\mathcal H_\tau$ which is~\textit{the region controlled by our data} via the Carleman estimate of Theorem~\ref{thm_carleman}.

\begin{lemma}[Recovery formula]\label{lem_recovery}
Let $\tau\in[\tau_a,\tau_b]$. Then
\begin{equation}\label{eq_recovery}
\bar q\,\alpha+2Lw_\tau+\theta w_\tau=0\ \text{ on }\Ht,\qquad (\Box_g+q_1)w_\tau=-\bar q\,v_{2,\tau}\ \text{ in }Q_\tau .
\end{equation}
In particular, ~\eqref{eq:alpha-lower} implies $|\bar q|^2\le C \bigl(|Lw_\tau|^2+|w_\tau|^2\bigr)$ on $\Ht\cap K$.
\end{lemma}

We can now finish the proof of Theorem~\ref{thm_main}.

\begin{proof}[Proof of Theorem~\ref{thm_main}]
Pick a function $\psi$ such that the assumptions of Theorem~\ref{thm_carleman} hold. One can take for instance $\psi=t_\diam$ a temporal function given by global hyperbolicity. We apply Theorem~\ref{thm_carleman} with $\zeta=\zeta_0$ and $q=q_1$. Since $\|q_1\|_{L^\infty}\le A$, its constants depend only on $A$ and the geometry, and not on $\tau$.

Using~\eqref{eq_recovery} we can apply Theorem~\ref{thm_carleman} to $u=w_\tau|_{X_\tau}$ with $Pu=-\bar q\,v_{2,\tau}$. Dropping the bulk terms on the left and using \eqref{bound_admissible} we get
\[
\sigma^3|e^{\sigma\tilde\psi}w_\tau|^2_{\Ht}+\sigma|e^{\sigma\tilde\psi}Lw_\tau|^2_{\Ht}\le C\|e^{\sigma\tilde\psi}\bar q\|^2_{L^2(K)}+C E_{\tau,\sigma},
\]
where $E_{\tau,\sigma}=\sigma^3|e^{\sigma\tilde\psi}w_\tau|^2_{\Gamma_\tau}+\sigma|e^{\sigma\tilde\psi}\partial_\nu w_\tau|^2_{\Gamma_\tau}+\sigma|e^{\sigma\tilde\psi}\nabla_h^{\tan}w_\tau|^2_{\Gamma_\tau}$. The pointwise estimate of Lemma~\ref{lem_recovery} yields
\[
|e^{\sigma\tilde\psi}\bar q|^2_{\Ht\cap K}\le C_1\bigl(|e^{\sigma\tilde\psi}Lw_\tau|^2_{\Ht}+|e^{\sigma\tilde\psi}w_\tau|^2_{\Ht}\bigr)\le\frac{C_2}{\sigma}\Bigl(\|e^{\sigma\tilde\psi}\bar q\|^2_{L^2(K)}+ E_{\tau,\sigma}\Bigr).
\]
Recall that $f(K) \subset[\tau_a,\tau_b]$. We can now integrate over the leaves $\mathcal H_\tau$ using the coarea formula (see for instance~\cite[Section 3.4]{evans2015}). We find
\begin{align*}
\|e^{\sigma\tilde\psi}\bar q\|^2_{L^2(K)}&=\int_{\tau_a}^{\tau_b}\!\!\int_{\Ht\cap K}\frac{e^{2\sigma\tilde\psi}|\bar q|^2}{|\nabla_hf|_h}\,d\sigma_h\,d\tau\le C \int_{\tau_a}^{\tau_b}|e^{\sigma\tilde\psi}\bar q|^2_{\Ht\cap K}\,d\tau\\
&\le\frac{C}{\sigma}\|e^{\sigma\tilde\psi}\bar q\|^2_{L^2(K)}+\frac{C}{\sigma}\int_{\tau_a}^{\tau_b}E_{\tau,\sigma}\,d\tau .
\end{align*}
We now fix $\sigma=\sigma_1$ large enough to absorb the first term. We then have $E_{\tau,\sigma_1}\le\sigma_1^3e^{2\sigma_1\tilde\psi_+}E_\tau$ with $\tilde\psi_+=\max_{\overline\uu}\tilde\psi$, and $e^{2\sigma_1\tilde\psi}\ge e^{2\sigma_1\tilde\psi_-}$ on $K$ with $\tilde\psi_-=\min_K\tilde\psi$. Hence
\[
\|q_1-q_2\|^2_{L^2(K)}\le C ^2\int_{\tau_a}^{\tau_b}E_\tau\,d\tau=C^2\,\mathcal E^2,
\]
and the theorem is proved.
\end{proof}

\begin{remark}\label{rem_weight}
The proof of Theorem~\ref{thm_main} uses only the properties of $\psi$ required in Theorem~\ref{thm_carleman}: $d\psi$ is timelike and $\psi$ increases along future-directed timelike curves. Besides the constants of Theorem~\ref{thm_carleman}, the weight enters the final constant only through the factor $e^{2\sigma_1(\tilde\psi_+-\tilde\psi_-)}$, which is finite since $\overline\uu$ is compact. No further property of the weight is needed. This is in sharp contrast with the usual applications of Carleman estimates in which the weight must be chosen carefully. This is another manifestation of the fact that, in spirit, Theorem~\ref{thm_carleman} is closer to an energy estimate than to a typical Carleman estimate (see Remark~\ref{rem_powers}).
\end{remark}

\subsection{Proof of Theorem~\ref{thm_calderon}}

The proof of Theorem~\ref{thm_calderon} essentially follows from Theorem~\ref{thm_main} combined with a gluing argument. Firstly, we need to make sure that one can choose a source away from $\overline\uu$.

\begin{lemma}\label{lem_exterior_source}
There is $\chi$ satisfying Assumption~\ref{assumption_sources} with $\supp\chi\cap\overline\uu=\emptyset$.
\end{lemma}
\begin{proof}
Let $\Phi_s$ be the flow of $L$ on $\dg$. Recall that the generator through $x\in\dg$ emanates from $\gamma(f(x))$. So $\Phi_s(x)$ tends to $\gamma(f(x))$ as $s$ decreases. Let $\mathcal N$ be an open neighbourhood of $\gamma(I)$ with $\overline{\mathcal N}\cap\overline\uu=\emptyset$. For $x\in K$ there is $s_x<0$ with $y_x:=\Phi_{s_x}(x)\in\mathcal N\setminus\gamma(I)$. Choose $b_x\in C^\infty_c(\mathcal N\setminus\gamma(I))$, $b_x\ge0$, with $b_x(y_x)>0$. By continuity of the flow, $b_x(\Phi_{s_x}(x'))>0$ for $x'$ in a neighbourhood $O_x$ of $x$. Cover $K$ by $O_{x_1},\dots,O_{x_m}$ and set $\chi=\sum_ib_{x_i}$.
\end{proof}

Theorem~\ref{thm_main} uses the traces on $\partial\uu$ of the waves $V_{q_1,\tau}$ and $V_{q_2,\tau}$, sent from the same source outside $\overline\uu$. The inclusion $\mathcal C_{q_1}(\uu)\subset\mathcal C_{q_2}(\uu)$ is only a statement about sets. For the wave $U_1$ sent from a smooth source $F$ outside $\overline\uu$, it gives some solution $u_2$ of $(\Box_g+q_2)u_2=0$ in $\uu$ with the same Cauchy data as $U_1$, but it does not say that $u_2$ is the restriction of the wave $U_2$ sent from $F$. The next lemma shows that it is. The idea is to glue $u_2$ inside $\uu$ to $U_1$ outside. Since $q_1=q_2$ outside $\uu$, and both the Dirichlet and the Neumann data match on $\partial\uu$, no source is created on $\partial\uu$, and the glued function is $U_2$ by uniqueness. In other words, the Cauchy data sets determine the waves sent from outside $\overline\uu$.

\begin{lemma}\label{lem_gluing}
Let $q_1,q_2\in C^\infty(\M)$, $\supp q_j\subset K$, with $\mathcal C_{q_1}(\uu)\subset\mathcal C_{q_2}(\uu)$. Let $F\in\mathcal E'(\diam)$ with $\supp F\cap\overline\uu=\emptyset$, and let $U_j\in\mathcal D'(\diam)$ solve $(\Box_g+q_j)U_j=F$ in $\diam$ with $\supp U_j\subset J^+(\supp F)$. Then $U_1=U_2$ in $\diam\setminus\overline\uu$.
\end{lemma}

\begin{proof}
Assume first that $F\in C^\infty_c(\diam\setminus\overline\uu)$. Then $U_1\in C^\infty(\diam)$ and $(\Box_g+q_1)U_1=0$ in $\uu$. By assumption, there is $u_2\in C^\infty(\overline\uu)$ with $(\Box_g+q_2)u_2=0$ in $\uu$ and the same Cauchy data as $U_1$, that is, $u_2=U_1$ and $du_2=dU_1$ on $\partial\uu$. Let
\[
\tilde U=\begin{cases}u_2&\text{in }\overline\uu,\\ U_1&\text{in }\diam\setminus\overline\uu.\end{cases}
\]
Then $\tilde U\in C^1(\diam)$, and since $q_1=q_2$ outside $K$, Green's formula on both sides of $\partial\uu$ gives
\[
(\Box_g+q_2)\tilde U=F\qquad\text{in }\diam .
\]
So $\tilde U-U_2$ solves the homogeneous equation and is supported in $\overline\uu\cup J^+(\supp F)$. Since a Cauchy temporal function $t_\diam$ of $\diam$ increases along future-directed causal curves, $\tilde U-U_2$ vanishes in $\{t_\diam<\min_{\overline\uu\cup\supp F}t_\diam\}$. Hence $\tilde U=U_2$ by uniqueness of the Cauchy problem and $U_1=U_2$ in $\diam\setminus\overline\uu$.

The case $F \in \mathcal E'(\diam)$ follows by approximation by functions in $C^\infty_c(\diam\setminus\overline\uu)$ and by using the continuity of the solution operator.
\end{proof}

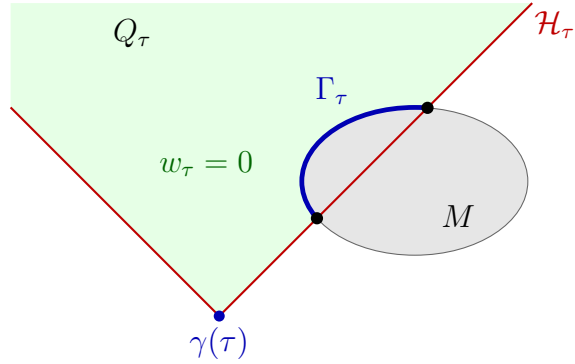
\begin{figure}[!ht]
\centering
\begin{tikzpicture}[scale=1.15]
\fill[green!10] (0,0) -- (3.6,3.6) -- (-2.4,3.6) -- (-2.4,2.4) -- cycle;
\path[name path=ellA] (2.25,1.55) ellipse[x radius=1.3, y radius=0.85];
\fill[gray!20] (2.25,1.55) ellipse[x radius=1.3, y radius=0.85];
\draw[black!60] (2.25,1.55) ellipse[x radius=1.3, y radius=0.85];
\begin{scope}
  \clip (0,0) -- (4,4) -- (-3,4) -- (-3,3) -- cycle;
  \draw[blue!70!black, line width=1.8pt] (2.25,1.55) ellipse[x radius=1.3, y radius=0.85];
\end{scope}
\path[name path=rayA] (0,0) -- (3.6,3.6);
\draw[red!75!black, thick] (-2.4,2.4) -- (0,0) -- (3.6,3.6);
\fill[name intersections={of=ellA and rayA, name=iA}]
  (iA-1) circle (2pt) (iA-2) circle (2pt);
\fill[blue!70!black] (0,0) circle (1.8pt);
\node[blue!70!black, below] at (0,0) {$\gamma(\tau)$};
\node[red!75!black, right] at (3.5,3.35) {$\mathcal H_\tau$};
\node at (-1.0,3.25) {$Q_\tau$};
\node[green!40!black] at (-0.15,1.75) {$w_\tau=0$};
\node at (2.75,1.15) {$\uu$};
\node[blue!70!black] at (1.3,2.6) {$\Gamma_\tau$};
\end{tikzpicture}
\caption{End of the proof of Theorem~\ref{thm_calderon}, for $\tau\in\mathcal R$. The waves agree outside $\overline\uu$, so $w_\tau$ and $dw_\tau$ vanish on the shaded set $Q_\tau\setminus\overline\uu$, hence on $\Gamma_\tau$.}
\label{fig:gamma_density}
\end{figure}

We can now conclude the proof of Theorem~\ref{thm_calderon}.

\begin{proof}[Proof of Theorem~\ref{thm_calderon}]
Let $K=\supp q_1\cup\supp q_2$, a compact subset of $\uu$, and extend $q_1,q_2$ by zero to $\M$. Choose $\chi$ as in Lemma~\ref{lem_exterior_source}, and let $\tau\in\mathcal R$. Lemma~\ref{lem_gluing} with $F=\chi\,\delta'(\varphi_\tau)$ gives $V_{q_1,\tau}=V_{q_2,\tau}$ in $\diam\setminus\overline\uu$. Since $V_{q_j,\tau}=v_{j,\tau}$ on $Q_\tau$ by \eqref{eq:main},
\[
w_\tau=0\qquad\text{on }Q_\tau\setminus\overline\uu,
\]
and, by continuity, $w_\tau$ and $dw_\tau$ vanish on $\Gamma_\tau$ (Figure~\ref{fig:gamma_density}), which means $E_\tau=0$. This holds for every $\tau\in\mathcal R$, so $\mathcal E=0$, and Theorem~\ref{thm_main} gives $q_1=q_2$.
\end{proof}

\appendix

\section{Facts from Lorentzian Geometry}\label{ap_lorentz}

For an introduction to Lorentzian geometry we refer to~\cite{oneill1983, ringstrom2009}. We consider $(\M, g)$ a connected Lorentzian manifold of dimension $1+n$ with signature $(-,+,...,+)$. A vector $v \in T_p \M \backslash \{0\}, \: p \in \M$, is timelike, spacelike or null if
\begin{equation}
    g(v,v)<0,\quad g(v,v)>0, \quad g(v,v)=0,
\end{equation}
respectively. We say that $v \in T_p \M \backslash \{0\} $ is a causal vector if it is timelike or null, that is $g(v,v) \leq 0$. The manifold $(\M, g)$ will be assumed to be time oriented. That means that there exists a vector field $T$ such that $T(p)$ is timelike for all $p \in \M$. This implies that for a causal vector $v$ one has $g(v,T)<0$ or $g(v,T)>0$. Then $v$ is said to be future-directed if $g(v,T)<0$ and past-directed if the opposite inequality holds.

 A submanifold $S$ of $\M$ is called spacelike if all the non-zero tangent vectors of $S$ are spacelike. A smooth curve $\gamma$ is timelike if its tangent vector $\dot \gamma(s)$ is timelike for all $s \in [a,b]$ with analogous definitions for spacelike, null, causal and future/past directed.

Given $p,q \in \M$ we can define the~\textit{causality relations} on $\M$ as follows.
\begin{enumerate}
    \item  We write $p \leq q$ if $p=q$ or there is a future-directed \textit{causal} curve from $p$ to $q$.

    \item We write $p \ll q$ if there is a future-directed \textit{timelike} curve from $p$ to $q$.
\end{enumerate}

This allows to define the~\textit{causal future} of a subset $A \subset \M$ as the set
\begin{equation}
    J^+(A)=\{q \in \M \textnormal{: there is $p$} \in A \textnormal{ such that } p\leq q \},
\end{equation}
as well as the~\textit{chronological future} of $A \subset \M$ as
\begin{equation}
    I^+(A)=\{q \in \M \textnormal{: there is $p$} \in A \textnormal{ such that } p\ll q \}.
\end{equation}

 We have analogous definitions for the~\textit{causal past} and~\textit{chronological past} $J^-(A)$ and $I^{-}(A)$ of $A$ by reversing the sign of the inequalities above. We write $J^{\pm}(p)= J^{\pm}(\{p\})$ and $I^{\pm}(p)= I^{\pm}(\{p\})$ for points $p \in \M$. The sets $I^{\pm}(A)$ are always open in $\M$. However, $J^{\pm}(A)$ need not be closed.

 We now define global hyperbolicity.
\begin{definition}
    The manifold $(\M, g)$ is~\textit{globally hyperbolic} if
    \begin{enumerate}
        \item \label{point_one} there is no closed causal curve,

        \item \label{point_two} the causal diamonds $J^+(p) \cap J^-(q)$ are compact sets for all $p, q \in \M$.
    \end{enumerate}
\end{definition}

We refer to~\cite{hounnonkpe2019} for this relatively recent definition. If $n \geq 2$ and $\M$ is non-compact then~\eqref{point_one} can be omitted from the definition above~\cite{hounnonkpe2019}. On the other hand, globally hyperbolic manifolds are non-compact. Global hyperbolicity has far reaching consequences, which are summarized in the next theorem. A hypersurface $\Sigma \subset \M$ is called a~\textit{Cauchy surface} if every inextendible timelike curve intersects $\Sigma$ exactly once.

\begin{theorem}
\label{prop_of_globhyper}
    The following are equivalent for $(\M, g)$.

\begin{enumerate}
    \item $(\M, g)$ is globally hyperbolic.

    \item There exists a smooth spacelike Cauchy surface in $\M$.

 \item There exists a Cauchy temporal function, that is, a smooth function $t\colon\M\to\R$ whose gradient $\nabla t$ is timelike and past-directed, and whose level sets $\Sigma_s=t^{-1}(s)$ are smooth spacelike Cauchy surfaces. The function $t$ can be taken surjective, and its restriction to any inextendible future-directed causal curve is strictly increasing and surjective onto $\R$. Moreover, $\M$ is diffeomorphic to $\R\times\Sigma_0$ via $F\colon\R\times\Sigma_0\to\M$ with
\begin{equation}
    (F^*g)(s,x)=c(s,x)\bigl(-ds^2+h_s(x)\bigr),\qquad t\bigl(F(s,x)\bigr)=s,
\end{equation}
where $c>0$ is smooth and $h_s$ is a Riemannian metric on $\Sigma_0$ depending smoothly on $s$.
\end{enumerate}
\end{theorem}

This equivalence has been established in~\cite{bernal2003, bernal2005}. See also the references in~\cite[Proposition 2.3]{oksanen2024rigiditylorentziancalderonproblem} from which comes the formulation in Theorem~\ref{prop_of_globhyper}.

\medskip

Globally hyperbolic manifolds provide with a general geometric setting in which the Cauchy problem for the wave equation is well posed. This is the content of the following proposition.
\begin{proposition}
\label{prop_wave_lorentz}
  Let $(\M, g)$ be globally hyperbolic and $P$ a normally hyperbolic operator. Then for any $F \in \mathcal{E}^\prime(\M)$ there exists a unique $u \in \mathcal{D}^\prime (\M)$ such that
 \begin{equation}
\begin{cases}
    P u=F  \quad \textnormal{in } \M,\\
    \supp (u) \subset J^+(\supp (F)).
    \end{cases}
\end{equation}
\end{proposition}
For the existence part of the above proposition we refer to~\cite[Theorem 3.3.1]{bar2007} and the uniqueness follows from~\cite[Theorem 3.1.1]{bar2007}. See as well~\cite[Theorem 6.3.1]{friedlander1975book}. An operator $P: C^\infty(\M) \to C^\infty(\M) $ with $\M$ a \textit{Lorentzian manifold} is called normally hyperbolic if its principal symbol is given by the metric
\begin{equation}
    \sigma_P(\xi)=g^{-1}(\xi,\xi).
\end{equation}
In particular $P=\bg+q$ with $q \in C^\infty(\M)$ is normally hyperbolic.

\subsection{Auxiliary Riemannian metric}

\begin{lemma}
\label{lem_aux_riem_metric}
Let $(\M,g)$ be a Lorentzian manifold and let $t : \M \to \mathbb{R}$ be a smooth time function such that $\nabla t$ is everywhere timelike.
Define the symmetric $2$-tensor
\[
h(u,v) = g(u,v) + 2\, g(u, T) g(v, T),
\quad \text{where } T = \frac{\nabla t}{\sqrt{-g(\nabla t,\nabla t)}}.
\]
Then $h$ is a Riemannian metric on $\M$. Moreover, the density function associated to $h$ is the same as the density of $g$, that is $d V_h=d V_g=\sqrt{|\det(g_{ij})|}\, dx^0 \wedge \cdots \wedge dx^n$.
\end{lemma}

\begin{proof}
Let $p \in \M$. Since $T$ is timelike we have a direct sum
\[
T_p \M = \mathrm{span}\{T\} \oplus  T^\perp.
\]
We can now choose a $g$-orthonormal frame $\{T, e_1, \dots, e_n\}$ at $p$, so that the matrix of $g$ in this frame is
\[
(g_{ij}) = \mathrm{diag}(-1,1,\dots,1).
\]
We have $g(e_i,T)=0$ for $i \ge 1$ and
\[
h(T,T) = g(T,T) + 2 g(T,T)^2 = -1 + 2 = 1,
\]
\[
h(T,e_i) = 0,
\quad
h(e_i,e_j) = g(e_i,e_j) = \delta_{ij}.
\]
Thus the matrix of $h$ in this frame is
\[
(h_{ij}) = \mathrm{diag}(1,1,\dots,1).
\]
so $h$ is positive definite and therefore a Riemannian metric.
This proves the lemma, since both properties are independent of the choice of frame.
\end{proof}

\subsection{Static and null-cone foliation}

\begin{lemma}[Null-cone foliation in the ultrastatic case]
\label{lem:ultrastatic-foliation}
Let $(N,h)$ be a complete Riemannian manifold and $x_0\in N$ a point
at which $\exp_{x_0}\colon T_{x_0}N\to N$ is a diffeomorphism.
Consider the ultrastatic spacetime $(\mathcal{M},g)=(\mathbb{R}\times N,-dt^2+h)$
and let $\uu\subset\mathcal{M}$ be open with compact closure, such that $x_0\notin\pi_N(\overline\uu)$,
where $\pi_N\colon\mathbb{R}\times N\to N$ is the spatial projection.
Then Assumption~\ref{assumption_null_cone_foliation} is satisfied with
$\gamma(\tau)=(\tau,x_0)$ and $f(t,x)=t-d_h(x_0,x)$.
\end{lemma}

\begin{proof}
Since $\overline\uu$ is compact, we may choose
$t_-<t_+$ such that $\overline\uu\subset\diam:=I^+(p_-)\cap I^-(p_+)$ with
$p_\pm=(t_\pm,x_0)$.
The diamond can be expressed as
\[
  \diam=\bigl\{(t,x):d_h(x_0,x)<\min(t-t_-,\,t_+-t)\bigr\}.
\]
Define $I=(t_-,t_+)$. Then $\uu\Subset\diam$, and $\overline{\uu}\cap\gamma(I)=\emptyset$ since $x_0\notin\pi_N(\overline\uu)$. We now verify the requirements of Assumption~\ref{assumption_null_cone_foliation} for $f(t,x):=t-d_h(x_0,x)$. Since $\exp_{x_0}$ is a global diffeomorphism $d_h(x_0,\cdot)\in C^\infty(N\setminus\{x_0\})$ and $f\in C^\infty(\dg)$. We also have
\[
  g^{-1}(df,df)
  =-(\partial_t f)^2+|\nabla_h d_h(x_0,\cdot)|_h^2,
  =-1+1=0,
\]
since $|\nabla_h d_h(x_0,\cdot)|_h=1$ on $N\setminus\{x_0\}$. Hence, the level sets of $f$ are null hypersurfaces and they clearly foliate $\dg$.

The description of the null cones emanating from $\gamma$ is a consequence of the expression
\[
J^+(\gamma(\tau))
=
\{(t,x):d_h(x_0,x)\leq t-\tau\}.
\]
For the inclusion ``$\subset$'', let $(t,x)\in J^+(\gamma(\tau))$ and let
$[0,1] \ni s \to \Gamma(s)=(t(s),x(s))$ be a future-directed causal curve from
$(\tau,x_0)$ to $(t,x)$. Then
\[
0\geq g(\dot\Gamma,\dot\Gamma)
=-\dot t^2+|\dot x|_h^2,
\]
and $\dot t>0$, so $|\dot x|_h\leq\dot t$. Hence
\[
d_h(x_0,x) \leq
\int_0^1 |\dot x|_h\,ds
\leq\int_0^1 \dot t\,ds
=t-\tau.
\]
Conversely, suppose that
\[
d_h(x_0,x)\leq t-\tau.
\]
By completeness of $(N,h)$, there exists a unit-speed minimizing
geodesic $\sigma:[0,L]\to N$ from $x_0$ to $x$, where
$L=d_h(x_0,x)$. Let
\[
\Gamma(s)=
\left(\tau+\frac{t-\tau}{L}s,\sigma(s)\right),
\qquad 0\leq s\leq L.
\]
Then $\Gamma$ joins $(\tau,x_0)$ to $(t,x)$ and
\[
g(\dot\Gamma,\dot\Gamma)
=-\left(\frac{t-\tau}{L}\right)^2+1\leq0,
\]
while $\dot t>0$. Thus $\Gamma$ is future-directed causal, and
$(t,x)\in J^+(\gamma(\tau))$.
\end{proof}

\begin{remark}\label{rem_local_ultrastatic}
Lemma~\ref{lem:ultrastatic-foliation} still holds if, for some $R>0$, $\exp_{x_0}$ is only a diffeomorphism from $\{v\in T_{x_0}N:|v|_h<R\}$ onto $B_h(x_0,R)$, provided that
\[
|t-t_c|+d_h(x_0,x)<R\qquad\text{for all }(t,x)\in\overline\uu,
\]
for some $t_c\in\R$. In the proof, one takes $t_\pm=t_c\pm R$, so that $\overline\uu\subset\diam\subset\R\times B_h(x_0,R)$, and one uses that $d_h(x_0,\cdot)\in C^\infty(B_h(x_0,R)\setminus\{x_0\})$, with $|\nabla_h d_h(x_0,\cdot)|_h=1$ there. The rest of the proof is unchanged. For instance, on the round sphere $N=S^n$ one can take $R=\pi$, and the condition reads $|t-t_c|<d_h(x,x_S)$, where $x_S$ is the antipode of $x_0$: the diamond $\diam$ must close before the null cones emanating from $\gamma$ refocus at $x_S$.
\end{remark}

\section{Distributions on manifolds}
\label{ap_distrib_manifold}

We collect here some definitions and facts about distributions. We closely follow~\cite[Chapter 6]{hormander1990}. We start by recalling the notion of pullback of distributions, given by~\cite[Theorem 6.1.2]{hormander1990}. This extends the composition operation in the following sense.

Let $U_j \subset \R^{n_j}, j \in \{1,2\}$ be open sets and $f: U_1 \to U_2$ be a smooth map such that $df(x)$ is surjective for every $x \in U_1$. Then there is a unique linear map $f^*: \mathcal{D}^\prime(U_2) \to \mathcal{D}^\prime(U_1)$ such that $f^*u=u \circ f$ when $u$ is smooth, called the pullback of $u$ by $f$. When defined, the pullback satisfies the usual rules of computation. In particular, the chain rule holds
\begin{equation}
\label{eq_chain_rule}
    \partial_j f^*u=\sum_{k=1}^{n_2} (\partial_j f_k) f^* \partial_k u.
\end{equation}
In the case where $f$ is a diffeomorphism, the pullback is given by
\begin{equation}
\label{eq_pull_diffeo}
    (f^*u, \phi)_{\mathcal{D}^\prime(U_1), C^\infty_0(U_1)}=  (u, |\det d(f^{-1})|(\phi \circ f^{-1}))_{\mathcal{D}^\prime (U_2), C^\infty_0(U_2)}.
\end{equation}

\bigskip

In an arbitrary manifold $M$, there is not a canonic way of integrating. As a consequence, there is no natural embedding of a smooth function $u$ in the dual of $C^\infty_0(M)$ via $\phi \to \int u\phi$,  since $\int u \phi$ does not have an invariant meaning. The space $\mathcal{D}^\prime (M)$ of distributions on $M$ is instead defined by gluing together distributions $u_\kappa \in \mathcal{D}^\prime (\kappa(U_\kappa))$, given in local coordinates $(U_\kappa, \kappa )$, such that
\begin{equation}
\label{eq_trans_law_distrib}
    u_{\kappa^\prime}=(\kappa \circ \kappa^{\prime -1})^* u_\kappa \textnormal{ in } \kappa^\prime(U_\kappa \cap U_{\kappa^\prime}),
\end{equation}
holds on an overlapping chart.

As for the dual of $C^\infty_0(M)$, it can be identified with the space of~\textit{distribution densities} on $M$. Those are obtained by gluing together distributions $u_\kappa \in \mathcal{D}^\prime (\kappa(U_\kappa))$, given in local coordinates $(U_\kappa, \kappa )$, but satisfying a different transformation law
\begin{equation}
\label{eq_trans_law_density}
    u_{\kappa^\prime}=|\det d \psi| \psi ^* u_\kappa \textnormal{ in } \kappa^\prime(U_\kappa \cap U_{\kappa^\prime}), \quad \psi=\kappa \circ \kappa^{\prime -1},
\end{equation}
on an overlapping chart. Now given $u$ a distribution density on $M$ and $\phi \in C^\infty_0 (U_\kappa)$ supported on a coordinate neighborhood, we define the pairing
\begin{equation}
    (u, \phi)=(u_\kappa, \phi \circ \kappa^{-1})_{\mathcal{D}^\prime(\kappa(U_\kappa)), C^\infty_0 (\kappa(U_\kappa))},
\end{equation}
and this can be extended to any $\phi \in C^\infty_0(M)$ by using a partition of unity subordinate to a system of coordinates.  The transformation law~\eqref{eq_trans_law_density} ensures that $(u, \phi)$ is well defined. Indeed, we consider an overlapping chart $U_\kappa \cap U_{\kappa^\prime}$ and we write $U_1=\kappa^\prime (U_\kappa \cap U_{\kappa^\prime})$ and $U_2=\kappa (U_\kappa \cap U_{\kappa^\prime})$. Using~\eqref{eq_trans_law_density} and the pullback expression~\eqref{eq_pull_diffeo} we find
\begin{align}
    &(u_{\kappa^\prime}, \phi \circ \kappa^{\prime -1})_{\mathcal{D}^\prime(U_1), C^\infty_0 (U_1)}=(|\det d \psi| \psi ^* u_\kappa,\phi \circ \kappa^{\prime -1} )_{\mathcal{D}^\prime(U_1), C^\infty_0 (U_1)} \\
    &=( \psi ^* u_\kappa,|\det d \psi| (\phi \circ \kappa^{\prime -1} ))_{\mathcal{D}^\prime(U_1), C^\infty_0 (U_1)}\\&
    =( u_\kappa,|\det d \psi^{-1}| ( |\det d \psi| \circ \psi^{-1} )(\phi \circ \kappa^{ -1} ))_{\mathcal{D}^\prime(U_2), C^\infty_0 (U_2)} =(u_{\kappa}, \phi \circ \kappa^{-1})_{\mathcal{D}^\prime(U_2), C^\infty_0 (U_2)}.
\end{align}

We notice finally that if $\phi$ is a smooth compactly supported density and $u$ a distribution then $\phi u$ is a distribution density.  This identifies $\mathcal{D}^\prime (M)$ with the dual of smooth compactly supported densities by associating to a distribution $u$ the functional $\phi \to (u \phi, \chi )$ with $\chi=1$ on the support of $\phi$.

\bigskip

Now suppose $\mu$ is a fixed given density on $M$. The map
\begin{equation}
    u \to \mu u,
\end{equation}
allows to identify distributions with distributions densities and functions with densities. In particular, in a semi-Riemannian manifold $(M, g)$ endowed with its canonical density $\mu= |\mathrm{vol}_g| = \sqrt{|\det(g_{ij})|}\, |dx^1 \wedge \cdots \wedge dx^n|$ we have the identification $\mathcal{D}^\prime (M) = (C^\infty_0(M))^\prime$ by identifying a distribution $T \in  \mathcal{D}^\prime (M)$ with the functional $\phi \to (T, \phi \mu)$. In other words, once a density has been fixed on $M$ we have for $T \in  \mathcal{D}^\prime (M)$ the following dictionary
\begin{equation}
    (T, \phi)_{\mathcal{D}^\prime (M), C^\infty_0(M)}=(T , \phi \mu)=(\mu T, \phi).
\end{equation}
In the second equality $T$ is interpreted  as an element of the dual space of smooth compactly supported densities. In the last one we have a pairing between a distribution density and a smooth compactly supported function.

Now consider $(M,g)$ a semi-Riemannian manifold with its canonical density $\mu$ and $P$ a differential operator acting on functions $P: C^\infty(M) \to C^\infty(M)$. The action of $P$ extends to densities via conjugation by $\mu$. That is, for a density $\omega $ we define the density $P \omega$ as
\begin{equation}
    P \omega= \mu P(\mu^{-1}\omega),
\end{equation}
where $\mu^{-1} \omega$ is a smooth function on $M$.

Unwinding the definitions and identifications above we see that given a distribution $T \in \mathcal{D}^\prime(M)$ the distribution $PT \in \mathcal{D}^\prime (M) $ is well defined and acts on functions $\phi \in C^\infty_0(M)$ by
\begin{equation}
    (PT, \phi)=(T, P^*_\mu \phi),
\end{equation}
where the operator $P^*_\mu$ is the formal adjoint of $P$ for the density $\mu$, given by
$$
\int (Pf) g \mu = \int f (P^*_\mu g) \mu, \textnormal{ for } f,g \in C^\infty_0(M).
$$
In particular, one has the expected identity for the Laplace--Beltrami operator on $M$ and a distribution $T \in \mathcal{D}^\prime(M)$
\begin{equation}
    (\bg T,\phi)=(T, \bg \phi).
\end{equation}

\medskip

Now assume that $M$ is a Lorentzian manifold, let $\mathcal O\subset M$ be open and let $\phi\in C^\infty(\mathcal O)$ be real with $d\phi \neq 0 $ on $\n=\{\phi=0\}$. We define the pullback distribution $\delta(\phi)=\phi^*(\delta_0) \in \mathcal{D}^\prime(\mathcal O) $ where $\delta_0$ is the Dirac distribution on the real line. This is a well defined distribution on $\mathcal O$, and as such acts on compactly supported densities. By using the natural density of $M$, we can see $\delta(\phi)$ as a functional on $C^\infty_0(\mathcal O)$.

This object is well defined, but it is interesting to understand it in a more geometrically transparent way. More precisely, one would like to correlate this distribution with a notion of surface measure on $\n$. However, there is not a canonic way of integrating on a null hypersurface on a Lorentzian manifold, since the induced metric on $\n$ is degenerate. We can nevertheless relate $\delta(\phi)$ to a~\textit{Riemannian} surface measure. Indeed, we introduce an auxiliary Riemannian metric $h$ having the same density as $g$. Such a metric $h$ is given in Lemma~\ref{lem_aux_riem_metric}. Then, one can show as in~\cite[Theorem 6.1.5]{hormander1990} that in fact
\begin{equation}
    (\delta(\phi), w)=\int_\n \frac{w}{|\nabla_h \phi|_h} d \sigma_h,
\end{equation}
 where $d\sigma_h$ is the induced surface measure on $\n$ coming from \textit{the Riemannian metric } $h$. In other terms, we have the relationship
 \begin{equation}
     \delta(\phi)=\frac{d \sigma_h}{|\nabla_h \phi|_h}.
 \end{equation}

 We define as well the Heaviside function $H(\phi)$ associated to $\phi$ as $H(\phi)=\phi^*{H}$ where $H$ is the Heaviside function on $\R$. Using the chain rule~\eqref{eq_chain_rule} we get the expected relationship in coordinates between $\delta$ and $H$
 \begin{equation}
     \partial_jH(\phi)=\partial_j \phi\, \delta(\phi).
 \end{equation}

\bibliographystyle{alpha}
\bibliography{master,lorentzian_extra}

\end{document}